\documentclass{mcom-l}

\usepackage{amssymb}

\usepackage{graphicx}
\usepackage{algorithm}
\usepackage{algpseudocode}
\usepackage{verbatim}
\usepackage{tikz}
\usetikzlibrary{calc, shapes, angles, quotes, patterns, decorations.pathreplacing, cd}
\usepackage{subcaption}
\usepackage{bm}
\usepackage{hyperref}
\usepackage{cleveref}

\newtheorem{theorem}{Theorem}[section]
\newtheorem{lemma}[theorem]{Lemma}
\newtheorem{corollary}[theorem]{Corollary}

\theoremstyle{definition}

\theoremstyle{remark}
\newtheorem{remark}[theorem]{Remark}

\numberwithin{equation}{section}

\crefname{equation}{}{equations}
\crefname{theorem}{Theorem}{Theorems}
\crefname{lemma}{Lemma}{Lemmas}
\crefname{corollary}{Corollary}{Corollaries}
\crefname{remark}{Remark}{Remarks}
\crefname{algorithm}{Algorithm}{Algorithms}
\crefname{figure}{Figure}{Figures}
\crefname{appendix}{Appendix}{Appendices}

\newcommand{\bdd}[1]{ \boldsymbol{#1} }
\newcommand{\unitvec}[1]{\bdd{#1}}
\newcommand{\vertiii}[1]{{\left\vert\kern-0.25ex\left\vert\kern-0.25ex\left\vert #1 
		\right\vert\kern-0.25ex\right\vert\kern-0.25ex\right\vert}}

\DeclareMathOperator{\vcurl}{\mathbf{curl}}

\DeclareMathOperator{\dive}{div}

\DeclareMathOperator{\spann}{span}
\DeclareMathOperator{\supp}{supp}
\DeclareMathOperator{\cond}{cond}

\DeclareMathOperator{\Tr}{Tr}

\begin{document}

\title[Uniform Preconditioners for Linear Elasticity]{Uniform Preconditioners for High Order Finite Element Approximations of Planar Linear Elasticity}


\author{Mark Ainsworth}
\address{Division of Applied Mathematics, Brown University, Providence, RI}
\email{mark\_ainsworth@brown.edu}

\author{Charles Parker}
\address{Mathematical Institute, University of Oxford, Andrew Wiles Building, Woodstock Road, Oxford OX2 6GG, UK}
\email{charles.parker@maths.ox.ac.uk}
\thanks{This material is based upon work supported by the National Science Foundation under Award No. DMS-2201487.}

\subjclass[2020]{Primary 65N30, 65N55, 65F08, 74S05 }

\date{}

\begin{abstract}
		A new preconditioner is developed for high order finite element approximation of linear elastic problems on triangular meshes in two dimensions. The new preconditioner results in a condition number that is \emph{bounded independently of the degree $p$, the mesh-size $h$ and the ratio $\lambda/\mu$}. The resulting condition number is reduced to roughly $6.0$ for all values of the parameters and discretization parameters on standard test problems. Crucially, the overall cost of the new preconditioner is comparable to the cost of applying standard domain decomposition based preconditioners.
\end{abstract}

\maketitle


\section{Introduction}
\label{sec:intro}

We consider conforming degree $p$ finite element discretization of the linear elasticity problem on meshes of triangular elements 
\begin{align}
	\label{eq:le variational}
	\bdd{u} \in \bdd{H}^1_D(\Omega) : \qquad a_{\lambda}(\bdd{u}, \bdd{v}) = L(\bdd{v}) \qquad \forall \bdd{v} \in \bdd{H}^1_D(\Omega),
\end{align}
where $L(\cdot)$ is a continuous linear functional on $\bdd{H}^1_D(\Omega)$ and
\begin{align}
	\label{eq:alam def}
	a_{\lambda}(\bdd{u}, \bdd{v}) := \int_{\Omega} \left\{ 2\mu \bdd{\varepsilon}(\bdd{u}):\bdd{\varepsilon}(\bdd{v})+ \lambda (\dive \bdd{u})(\dive \bdd{v}) \right\} \ d\bdd{x} \qquad \forall \bdd{u}, \bdd{v} \in \bdd{H}^1(\Omega),
\end{align}
where $\bdd{\varepsilon}(\cdot)$ is the strain tensor and $\bdd{H}_D^1(\Omega)$ is the space of admissible displacements. 

While the parameters $\mu$ and $\lambda$ need only be positive, the case where the ratio $\lambda/\mu \gg 1$ arises in several applications, e.g. nearly incompressible linear elasticity, where $\mu$ and $\lambda$ are the Lam\'{e} parameters, or in iterated penalty methods for incompressible flow \cite{AinCP22SCIP,FortinGlow83,Glowinski84}, where $\nu = 2\mu$ is the viscosity and $\lambda$ serves the role of an (artificial) penalty parameter associated with the incompressibility condition. 

In each case, the fact that the ratio $\lambda/\mu \gg 1$ has various implications for the numerical approximation. Firstly, if the polynomial degree $p\in\{1,2,3\}$, then the finite element solution suffers from \textit{locking} in the limit $\lambda/\mu \to \infty$ (see e.g.  \cite{BabSuri92b}). However, if the degree $p\geq 4$ then, under mild conditions on the mesh (see \cite[Theorem 8.1]{AinCP21LE}), both the $h$-version method with $p \geq 4$ and the pure $p$-version will be locking-free. For this reason, we will assume that $p \geq 4$ in what follows. Secondly, a more insidious problem is that the linear systems arising from discretizations of \cref{eq:le variational} become increasingly ill-conditioned as the ratio $\lambda/\mu$ increases, in addition to the usual ill-conditioning due to the mesh-size $h$ and the polynomial degree $p$. The objective of the current work is to develop a preconditioner which does not degenerate in any of the limits $p \to \infty$, $h\to 0$, or $\lambda/\mu \to \infty$.

In order to illustrate both the seriousness of the problem and the fact that standard remedies fall short, we consider a simple example. The geometry for the Cook's membrane problem and the domain are displayed in \cref{fig:cook mesh} where the left hand face is held fixed and a vertical shear is applied on the right hand face. The problem is approximated using elements of degree $p$ in the range $\{4,\ldots,16\}$ on the fixed mesh shown in \cref{fig:cook mesh} with material parameters $\mu=1$ and $\lambda \in \{10^1, 10^3, 10^5, 10^7 \}$. 

\begin{figure}[htb]
	\centering
	\begin{subfigure}[b]{0.49\linewidth}
		\centering
		\includegraphics[height=6cm]{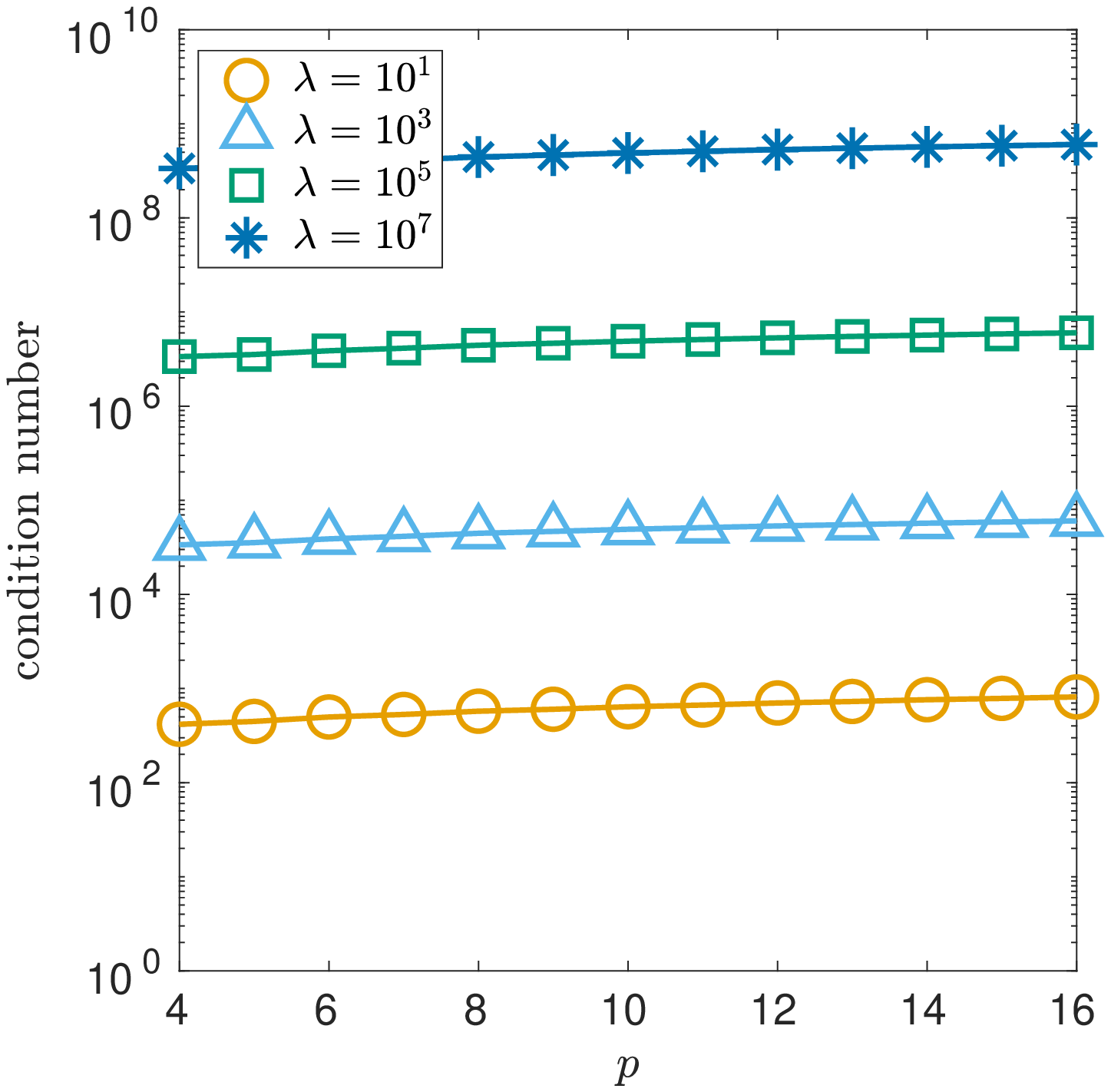}
		\caption{}
		\label{fig:bcmp p1 coarse}
	\end{subfigure}
	\hfill
	\begin{subfigure}[b]{0.49\linewidth}
		\centering
		\includegraphics[height=6cm]{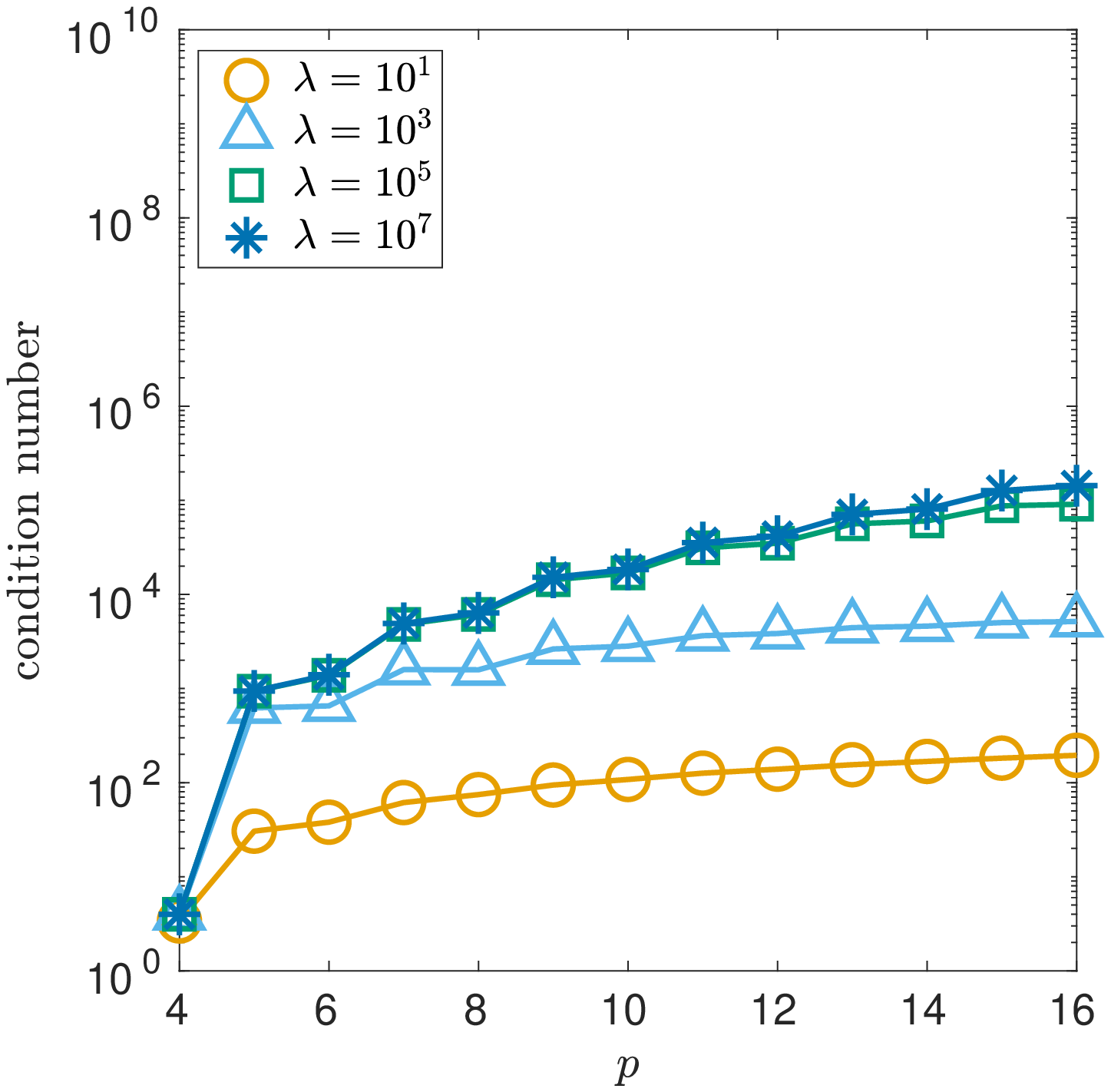}
		\caption{}
		\label{fig:bcmp p4 coarse}
	\end{subfigure}
	\\
	\begin{subfigure}[b]{0.49\linewidth}
		\centering
		\includegraphics[height=6cm]{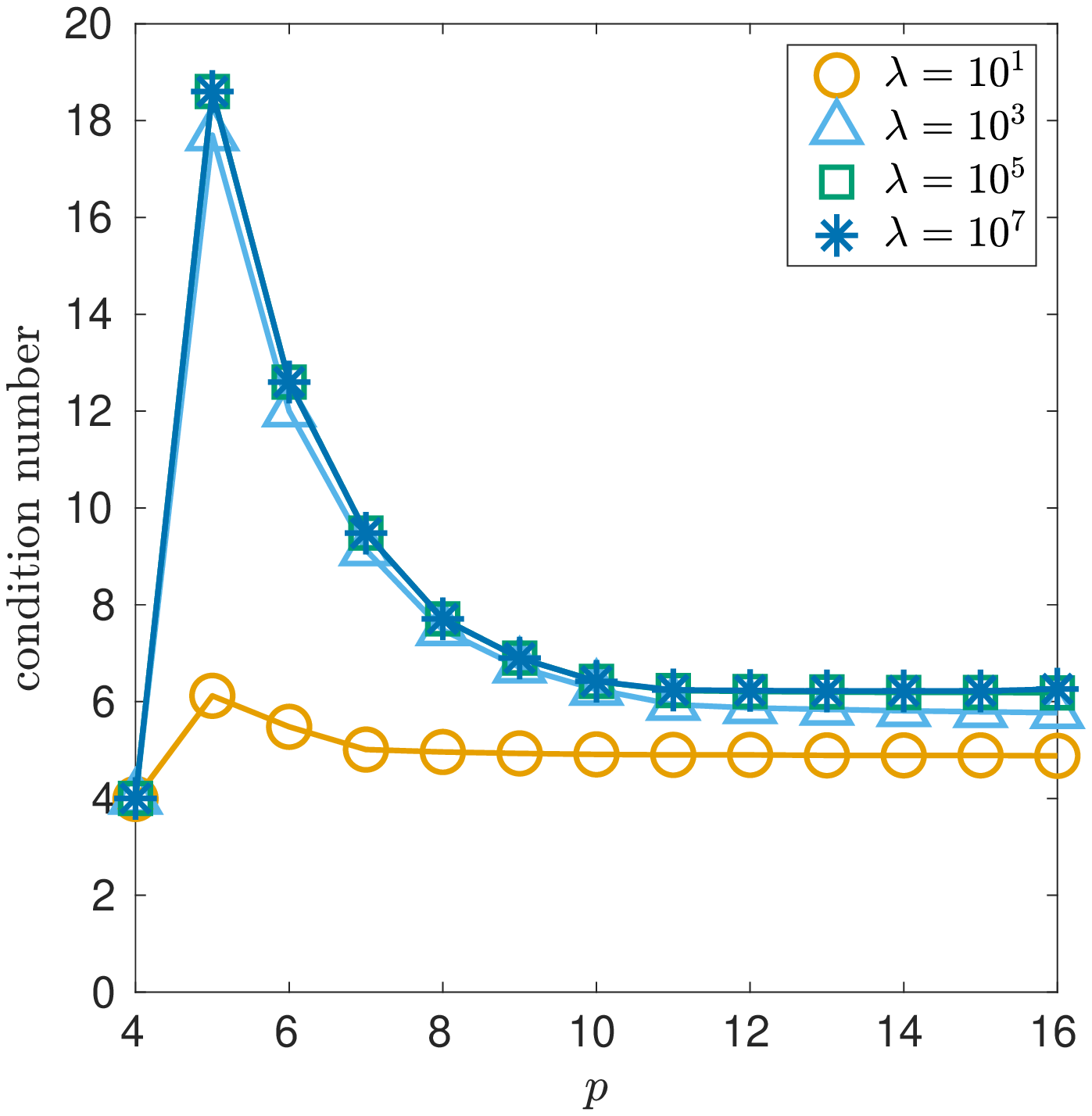}
		\caption{}
		\label{fig:intro cns}
	\end{subfigure}
	\hfill
	\begin{subfigure}[b]{0.49\linewidth}
		\centering
		\includegraphics[height=6cm]{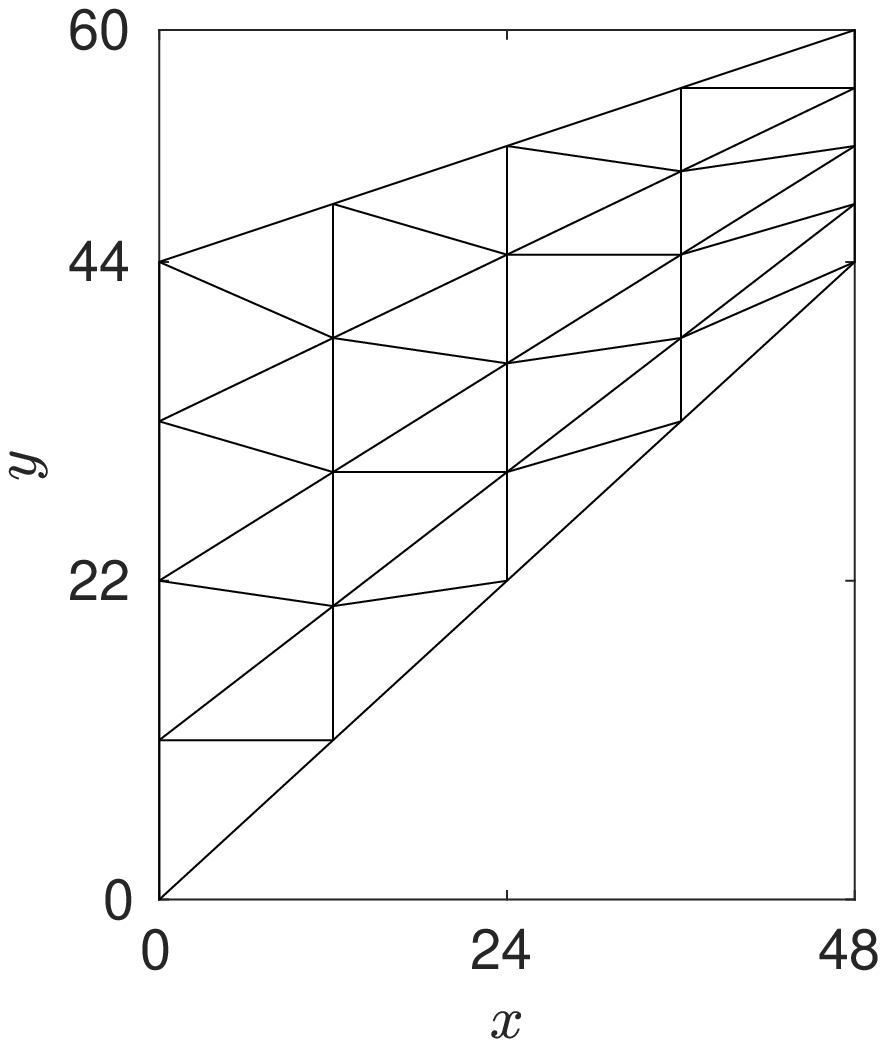}
		\caption{}
		\label{fig:cook mesh}
	\end{subfigure}
	\caption{(a-b) Condition numbers using the preconditioner in \cite{BCMP91} with (a) $p=1$ coarse space and (b) $p=4$ coarse space, (c) condition numbers using the preconditioner in \cref{sec:asm}, and (d) computational mesh. Observe that the condition numbers in (b) are bounded as $\lambda$ grows, but grow more rapidly than those in (a) as a function of $p$, while the condition numbers in (c) are uniformly bounded in $p$ and $\lambda$.}
	\label{fig:bcmp cns}
\end{figure}

\Cref{fig:bcmp p1 coarse} shows the condition number obtained when the standard preconditioner developed by Babu\v{s}ka et al. \cite{BCMP91} is used, which is known to reduce the growth of the condition number to $\mathcal{O}(1+\log^2 p)$ and to be independent of the mesh size $h$. While the results shown in \cref{fig:bcmp p1 coarse} are consistent with this fact, one observes that the actual condition number is proportional to the ratio $\lambda/\mu$ which means that even though the effects of $h$ and $p$ are controlled, the hidden constant in the $\mathcal{O}(1 + \log^2 p)$ estimate degenerates as $\lambda/\mu \to \infty$. 

The preconditioner of Babu\v{s}ka et al. is based on domain decomposition in conjunction with a coarse grid solver using a piecewise linear finite element subspace on the mesh. Interestingly, enlarging the coarse space to consist of piecewise \textit{quartic} functions results in a preconditioner that \emph{does not} degenerate at all as the ratio $\lambda/\mu$ increases as shown in \cref{fig:bcmp p4 coarse}. Unfortunately, the results indicate that the condition number now grows rapidly with the polynomial degree $p$, even though the only change compared with the standard Babu\v{s}ka et al. preconditioner was to enlarge the coarse space.

The current work provides an explanation for these observations but, more importantly, a new preconditioner is developed that gives rise to a condition number that is \emph{bounded independently of the degree $p$, the mesh-size $h$ and the ratio $\lambda/\mu$}. The performance of the new preconditioner on the Cook's membrane problem is illustrated in \cref{fig:intro cns} where it is seen that the condition number is reduced to roughly $6.0$ for all values of the parameters
and discretization parameters. Crucially, the overall cost of the new preconditioner is comparable to the cost of applying standard domain decomposition based preconditioners such as the Babu\v{s}ka et al. scheme. 

While various preconditioners for conforming finite element approximations of  \cref{eq:le variational} have already received some attention in the literature, e.g.~\cite{Lee09,Wu2014}, these works do not attempt to quantify or control the dependence on the polynomial degree $p$. Here, we fully account for the $p$ dependence and show that the resulting condition number is independent of $p$ as well as the mesh size $h$. 

In \cref{sec:constructing asm}, we describe a family of Additive Schwarz preconditioners that control the condition number of the preconditioned system not just with respect to $\mu$, $\lambda$, and $h$, but also with respect to the polynomial degree $p$. The subspace decomposition consists of a coarse space containing the
lowest order $(p = 4)$ elements and, for each vertex, a local solve associated with the functions supported on the patch of elements abutting the vertex. The use of a low order coarse space is standard for high order preconditioners
\cite{Ain96b,Ain96,AinCP19StokesIII,BCMP91,Guo96,Pavarino93,Pavarino98,Pavarino99,SchMelPechZag08}, while the vertex patch spaces, although more common in $H(\mathrm{div})$ and $H(\mathrm{curl})$ applications (see e.g.
\cite{Arn97,Arn00,Hiptmair98,Hiptmair07}), have also previously appeared in the $H^1$ context \cite{Lee09,Pavarino93,SchMelPechZag08,Wu2014}. Using technical results established in \cite{Lee09,Lee07,Wu2014}, we show that the condition number of the preconditioned system is uniform in $\mu$, $\lambda$, $h$, and $p$ for suitable choices of the space of functions associated with element boundaries. Several of these choices are detailed and the effectiveness of the preconditioner is showcased in several numerical examples in \cref{sec:numerics}, while the remainder of the paper is devoted to the analysis.

\section{Problem Setup}

Let $X \subset H^1(\Omega)$ be the space of continuous, piecewise polynomials of degree $p \in \mathbb{N}$ on a triangulation $\mathcal{T}$ of $\Omega$:
\begin{align}
	\label{eq:x fem definition}
	X := \{ v \in C(\bar{\Omega}) : v|_{K} \in \mathcal{P}_{p}(K) \ \forall K \in \mathcal{T} \},
\end{align}
where $\mathcal{P}_{p}(K)$ denotes the space of polynomials of degree at most $p$. In particular, we assume that the triangulation $\mathcal{T}$ is a shape-regular partitioning of the domain $\Omega$ into triangles such that the nonempty intersection of any two distinct elements from $\mathcal{T}$ is either a single common vertex or a single common edge of both elements with mesh size $h := \max_{K \in \mathcal{T}} h_K$ and $h_K := \mathrm{diam}(K)$. We also assume that mesh vertices are placed at the intersections of $\bar{\Gamma}_D$ and $\bar{\Gamma}_N$. The subspace $X_D := X \cap H^1_D(\Omega)$ then consists of functions in $X$ vanishing on the Dirichlet boundary $\Gamma_D$.

We discretize \cref{eq:le variational} using the space $\bdd{X}_D := X_D \times X_D$ as follows:
\begin{align}
	\label{eq:le variational fem}
	\bdd{u}_X \in \bdd{X}_D : \qquad a_{\lambda}(\bdd{u}_X, \bdd{v}) = L(\bdd{v}) \qquad \forall \bdd{v} \in \bdd{X}_D.
\end{align}
By fixing a basis for $\bdd{X}_D$, we may express \cref{eq:le variational fem} as a linear system of equations:
\begin{align}
	\label{eq:matrix system}
	\bdd{A} \vec{u} = \vec{L},
\end{align}
where $\bdd{A}$ is the stiffness matrix, $\vec{L}$ is the load vector, and $\vec{u}$ is the vector of degrees of freedom of $\bdd{u}$. As mentioned in \cref{sec:intro} and shown in the numerical examples in \cref{sec:numerics}, the conditioning of the matrix $\bdd{A}$ generally degenerates rapidly as $\mu \to 0$,  as $\lambda \to \infty$, as $h \to 0$, and/or as $p \to \infty$. In order to help solve the matrix equation \cref{eq:matrix system} efficiently, we seek ASM preconditioners for $\bdd{A}$ which remove the growth of the condition number in these parameters.

\section{Additive Schwarz Preconditioners}
\label{sec:constructing asm}

Let
\begin{align}
	\label{eq:interior space}
	\bdd{X}_I := \bigoplus_{K \in \mathcal{T}} \bdd{X}_I(K), 
	\qquad \text{where} \quad \bdd{X}_I(K) := \{ \bdd{v} \in \bdd{X}_D : \supp \bdd{v} \subseteq K\}
\end{align}
be the space of interior functions, i.e. functions vanishing on element boundaries, and let $\bdd{X}_B \subset \bdd{X}$ be \textit{any} subspace such that the following (direct sum) identity holds:
\begin{align}
	\label{eq:direct sum identity}
	\bdd{X}_D = \bdd{X}_I \oplus \bdd{X}_B.
\end{align}
Each $\bdd{v}_I \in \bdd{X}_I$ vanishes on the element boundaries, i.e. $\bdd{v}_I |_{\partial K} = \bdd{0}$ for all $K \in \mathcal{T}$, and so the space $\bdd{X}_B$ corresponds to the degrees of freedom associated with element boundaries.
\Cref{eq:direct sum identity} means that for each $\bdd{v} \in \bdd{X}_D$, there exist unique $\bdd{v}_I \in \bdd{X}_I$ and $\bdd{v}_B \in \bdd{X}_B$ such that $\bdd{v} = \bdd{v}_I + \bdd{v}_B$. Specifically, the uniqueness of $\bdd{v}_B$ means that one can define a mapping $\mathbb{T}_B : \bdd{X}_D \to \bdd{X}_B$ by the rule $\mathbb{T}_B \bdd{v} = \bdd{v}_B$. However, it is important to distinguish this uniqueness from the fact that the choice of $\bdd{X}_B$ itself is far from unique owing to the considerable freedom in choosing the form of the functions in $\bdd{X}_B$ on the element interiors. Moreover, the effectiveness of the ASM based on the decomposition \cref{eq:direct sum identity} will depend heavily on the choice of $\bdd{X}_B$.  We shall discuss this further along with some suitable choices in \cref{sec:numerics}.

\subsection[Decomposition of Boundary Space]{Decomposition of Boundary Space}
\label{sec:boundary decomposition}

As mentioned above, the space $\bdd{X}_B$ corresponds to the degrees of freedom associated with element vertices $\mathcal{V}$ and edges $\mathcal{E}$. In the interests of efficiency, it is useful to further decompose the boundary space $\bdd{X}_B$ into (overlapping) subspaces $\bdd{X}_C$ and $\{ \bdd{X}_{\bdd{a}} : \bdd{a} \in \mathcal{V} \}$ as follows: The coarse space consists of the lowest order (in this case $p = 4$) elements on the element boundaries:
\begin{align}
	\label{eq:tildexc def}
	\bdd{X}_C &:=  \{ \bdd{v} \in \bdd{X}_B : \bdd{v} |_{\gamma} \in \bm{\mathcal{P}}_4(\gamma) \ \forall \gamma \in \mathcal{E} \}. 
\end{align}
Although $\bdd{X}_C$ consists of piecewise polynomials whose restrictions to element edges are polynomials of degree at most 4, the fact that $ \bdd{X}_C \subseteq \bdd{X}_B$ means their values on the element interiors will be polynomials of degree $p$. The subspace $\bdd{X}_{\bdd{a}}$ associated with a vertex $\bdd{a} \in \mathcal{V}$ consists of functions supported on the patch $\mathcal{T}_{\bdd{a}}$ comprising of elements which abut $\bdd{a}$:
\begin{align}
	\bdd{X}_{\bdd{a}} := \{ \bdd{v} \in \bdd{X}_B : \supp \bdd{v} \subseteq \mathcal{T}_{\bdd{a}} \}, \ \bdd{a} \in \mathcal{V}.
\end{align}
The space $\bdd{X}_B$ then admits the following (overlapping) decomposition:
\begin{align}
	\label{eq:tilde subspace decomp}
	\bdd{X}_B = \bdd{X}_C + \sum_{\bdd{a} \in \mathcal{V}} \bdd{X}_{\bdd{a}}.
\end{align}

\subsection{Preconditioners}
\label{sec:asm}

\begin{algorithm}[thb]
	\caption{Action of ASM Preconditioner $P^{-1}$ on residual $r \in \bdd{X}_D^*$ }
	\label{alg:asm variational}
	\begin{subequations}
		\label{eq:asm variational components}
		\begin{alignat}{3} 
			\label{eq:interior solve}
			\bdd{u}_{I,K} \in \bdd{X}_I(K) &: \qquad & a_{\lambda,K}(\bdd{u}_{I,K}, \bdd{v}) &= r(\bdd{v}) \qquad & &\forall \bdd{v} \in \bdd{X}_I(K), \ \forall K \in \mathcal{T}, \\
			\label{eq:coarse solve}
			\bdd{u}_C \in \bdd{X}_C &: \qquad & a_{\lambda}(\bdd{u}_C, \bdd{v}) &= r(\bdd{v}) \qquad & &\forall \bdd{v} \in \bdd{X}_C, \\
			\label{eq:spdier solve}
			\bdd{u}_{\bdd{a}} \in \bdd{X}_{\bdd{a}} &: \qquad & a_{\lambda}(\bdd{u}_{\bdd{a}}, \bdd{v}) &= r(\bdd{v}) \qquad & &\forall \bdd{v} \in \bdd{X}_{\bdd{a}}, \ \forall \bdd{a} \in \mathcal{V}. 
		\end{alignat}
	\end{subequations}

	$P^{-1} r := \sum_{K \in \mathcal{T}} \bdd{u}_{I,K} + \bdd{u}_C + \sum_{\bdd{a} \in \mathcal{V}} \bdd{u}_{\bdd{a}}$. 	
\end{algorithm}

The subspace decomposition \cref{eq:direct sum identity} along with the further decomposition \cref{eq:tilde subspace decomp} gives rise to an ASM preconditioner $P^{-1} : \bdd{X}_D^* \to \bdd{X}_D$ whose action on a residual $r \in \bdd{X}_D^*$ is defined in \cref{alg:asm variational}. The sequence of interior corrections in \cref{eq:interior solve} is equivalent to a single interior correction
\begin{align*}
	\bdd{u}_I \in \bdd{X}_I : \qquad a_{\lambda}(\bdd{u}_{I}, \bdd{v}) = r(\bdd{v}) \qquad \forall \bdd{v} \in \bdd{X}_I
\end{align*}
thanks to \cref{eq:interior space} and the identity $a_{\lambda}(\bdd{u}, \bdd{v}) = \sum_{K \in \mathcal{T}} a_{\lambda, K}(\bdd{u}, \bdd{v})$, where $a_{\lambda, K}(\cdot,\cdot)$ denotes the restriction of $a_{\lambda}(\cdot,\cdot)$ to an element $K$.

We shall quantify the effectiveness of the preconditioner $P^{-1}$ in terms of two parameters $\beta_X$ and $\tau_B$. The first of these parameters is the inf-sup constant $\beta_X$ for the pair $\bdd{X}_D \times \dive \bdd{X}_D$:
\begin{align}
	\label{eq:inf-sup global}	
	\beta_X := \inf_{0 \neq q \in \dive \bdd{X}_D} \sup_{\bdd{0} \neq \bdd{v} \in \bdd{X}_D} \frac{(q, \dive \bdd{v})}{\|q\| \ \|\bdd{v}\|_1},
\end{align}
where $(\cdot,\cdot)$ denotes the $L^2(\Omega)$ or $\bdd{L}^2(\Omega)$ inner product and $\|\cdot\|_{s}$ denotes the $H^s(\Omega)$ or $\bdd{H}^s(\Omega)$ norm.  For elements of degree $p\geq 4$, it is known \cite[Theorem 5.1]{AinCP21LE} that $\beta_X$ is bounded below by a positive constant, $\beta_0$ independent of $h$ and $p,$ under certain (mild) assumptions on the mesh $\mathcal{T}$; see e.g. \cite[p. 35]{AinCP21LE} for a detailed characterization of the conditions. The inf-sup constant $\beta_X$ is more commonly associated with the mixed FEM discretization of the Stokes problem. However, its appearance in the current setting is to be expected since the linear elasticity problem (formally) converges to the Stokes flow in the incompressible limit $\lambda \to \infty$ with which we are concerned.

The second parameter $\tau_B$ quantifies the effect of the freedom in the choice of boundary space $\bdd{X}_B$ alluded to at the start of this section. Often, the first step in preconditioning high order elements is to \textit{statically condense} or eliminate the interior degrees of freedom \cite{Ain96b,Ain96,AinCP19StokesIII,BCMP91,Guo96,Pavarino98,Pavarino99,SchMelPechZag08}. Static condensation corresponds to defining the values of a function $\bdd{v} \in \bdd{X}_B$ on the interior of an element $K \in \mathcal{T}$ to be the minimum energy extension of the values of $\bdd{v}|_{\partial K}$. More precisely, if we define an operator $\mathbb{H} : \bdd{X}_D \to \bdd{X}_D$ by the rule
\begin{subequations}
	\label{eq:eh definition}
	\begin{alignat}{2}
		\label{eq:eh definition 1}
		a_{\lambda}(\mathbb{H} \bdd{u}, \bdd{v}) &= 0 \qquad & &\forall \bdd{v} \in \bdd{X}_I, \\
		\label{eq:eh definition 2}
		(\mathbb{H} \bdd{u} - \bdd{u})|_{\partial K} &= \bdd{0} \qquad & & \forall K \in \mathcal{T},
	\end{alignat}
\end{subequations}
then the boundary space $\bdd{X}_B$ corresponding to static condensation may be written equally well in any of the following ways:
\begin{align}
	\label{eq:static cond choice}
	\bdd{X}_B = \mathbb{H} \bdd{X}_D = \bdd{X}_D \cap \bdd{X}_I^{\perp} = \{ \bdd{v} \in \bdd{X}_D : a_{\lambda}(\bdd{v}, \bdd{w}) = 0 \ \forall \bdd{w} \in \bdd{X}_I \}.
\end{align}
In this case, the decomposition $\bdd{v} = \bdd{v}_I + \bdd{v}_B = \bdd{v}_I + \mathbb{H} \bdd{v}$ is orthogonal (so that in this case $\mathbb{T}_B = \mathbb{H}$) and satisfies the following estimate for all $\bdd{v} \in \bdd{X}_D$:
\begin{align}
	\notag
	\|\bdd{\varepsilon}(\mathbb{T}_B \bdd{v})\|^2 + \frac{\lambda}{2\mu} \|\dive \mathbb{T}_B \bdd{v}\|^2 = \frac{1}{2\mu} a_{\lambda}(\mathbb{T}_B \bdd{v}, \mathbb{T}_B \bdd{v}) &\leq \frac{1}{2\mu} a_{\lambda}(\bdd{v}, \bdd{v}) \\
	\label{eq:motivation tau static condense}
	&\leq \left(1 + \frac{\lambda}{\mu} \right) \|\bdd{\varepsilon}(\bdd{v})\|^2,
\end{align}
where the final inequality uses $\|\dive \bdd{v}\|^2 \leq 2\|\bdd{\varepsilon}(\bdd{v})\|^2$. Choosing $\bdd{X}_B$ to comprise of minimum energy extensions is natural and, in many cases, largely effective in controlling the growth of the condition number \cite{Ain96b,Ain96,AinCP19StokesIII,BCMP91,Guo96,Pavarino98,Pavarino99,SchMelPechZag08}. However, in the present setting, as we shall see, if the analysis of the preconditioner is based on \cref{eq:motivation tau static condense}, then the presence of $\lambda/\mu$ in the multiplier on the RHS gives rise to a bound on the condition number which \textit{degenerates as} $\lambda/\mu \to \infty$.

In order to obtain estimates for the performance of the preconditioner that are robust with respect to the ratio $\lambda/\mu$, one would like to choose $\bdd{X}_B$ so that the corresponding operator $\mathbb{T}_B$ satisfies an estimate of the form \cref{eq:motivation tau static condense} in which the multiplier appearing on the RHS is independent of $\lambda/\mu$. In general, this is not possible owing to the presence of the term $(\lambda/\mu) \|\dive \mathbb{T}_B \bdd{v}\|^2$ which blows up as $\lambda/\mu \to \infty$. 

Instead, we relax the condition \cref{eq:motivation tau static condense} by replacing the term $(\lambda/\mu) \|\dive \mathbb{T}_B \bdd{v}\|^2$ with a smaller quantity as follows; the parameter $\tau_B$ is defined to be any constant such that the following condition holds:
\begin{align}
	\label{eq:tau condition}
	\| \bdd{\varepsilon}(\mathbb{T}_B \bdd{v}) \|^2 + \lambda\mu^{-1} \| \Pi_I \dive \mathbb{T}_B \bdd{v} \|^2 \leq \tau_B^2 \| \bdd{\varepsilon}(\bdd{v})\|^2 \qquad \forall \bdd{v} \in \bdd{X}_D,
\end{align}
where $\Pi_I : L^2(\Omega) \to \dive \bdd{X}_I$ denotes the $L^2(\Omega)$ projection operator onto $\dive \bdd{X}_I$: i.e.
\begin{align*}
	(\Pi_I q, r) = (q, r) \qquad \forall r \in \dive \bdd{X}_I, \forall q \in L^2(\Omega).
\end{align*}
The key difference between \cref{eq:motivation tau static condense} and \cref{eq:tau condition} is that the full divergence $\dive \mathbb{T}_B \bdd{v}$ has been replaced with its projection onto the interior divergence space $\Pi_I \dive \mathbb{T}_B \bdd{v}$. The presence of the projection operator $\Pi_I$ in the term $\| \Pi_I \dive \mathbb{T}_B \bdd{v} \|$ is reminiscent of reduced integration operators that play an important role in finite element methods for nearly incompressible linear elasticity; see e.g. \cite[\S 8.12]{BoffiBrezziFortin13}.

Our main result quantifies the performance of $P^{-1}$ in terms of the factors $\beta_X$ and $\tau_B$ defined in \cref{eq:inf-sup global} and \cref{eq:tau condition}:
\begin{theorem}
	\label{thm:main result}
	Let $\bdd{P}^{-1}$ denote the matrix form of $P^{-1}$ and suppose that \cref{eq:inf-sup global} and \cref{eq:tau condition} hold. Then,
	\begin{align}
		\label{eq:main result}
		\cond(\bdd{P}^{-1} \bdd{A}) \leq C (1 + \beta_X^{-2}) (1 + \tau_B^2) ,
	\end{align}
	where $C$ is independent of $\beta_X$, $\tau_B$, $\mu$, $\lambda$, $h$, and $p$. 
\end{theorem}

\section[Three Choices of Boundary Space]{Three Choices of Boundary Space}
\label{sec:numerics}

In this section, we present three options for choosing the space $\bdd{X}_B$, all of which offer computational benefits in various setting, and will be shown to satisfy \cref{eq:tau condition} with $\tau_B$ independent of $\lambda/\mu$.

\subsection{Static Condensation}
\label{sec:static condensation}

We start by returning to the choice \cref{eq:static cond choice} where $\bdd{X}_B$ is defined in term of static condensation. As already pointed out, the multiplier on the RHS in \cref{eq:motivation tau static condense} degenerates as $\lambda/\mu \to \infty$. However, as we shall show in \cref{thm:elastic harmonic continuity element}, the weakened form \cref{eq:tau condition} \textit{does} hold with $\tau_B$ independent of $\mu$, $\lambda$, $h$, and $p$. \Cref{thm:main result} then gives the bound
\begin{align}
	\label{eq:static cond precon bound}
	\cond(\bdd{P}^{-1} \bdd{A}) \leq C  (1 + \beta_X^{-2}),
\end{align}  
where $C$ is independent of $\beta_X$, $\mu$, $\lambda$, $h$, and $p$.  

In practice, static condensation allows one to compute the solution to \cref{eq:le variational fem} more efficiently than suggested in \cref{alg:asm variational}. The orthogonality condition in \cref{eq:static cond choice} means that the stiffness matrix $\bdd{A}$ and preconditioner $\bdd{P}$ take the form
\begin{align*}
	\bdd{A} = \begin{bmatrix}
		\bdd{A}_{II} & \bdd{0} \\
		\bdd{0} & \bdd{A}_{BB}
	\end{bmatrix} \quad \text{and} \quad \bdd{P}^{-1} = \begin{bmatrix}
		\bdd{A}_{II}^{-1} & \bdd{0} \\
		\bdd{0} & \bdd{P}_{BB}^{-1}
	\end{bmatrix},
\end{align*}
where $\bdd{A}_{II}$ corresponds to the interactions between the interior degrees of freedom and $\bdd{A}_{BB}$ and $\bdd{P}_{BB}^{-1}$ correspond to the element boundary interactions. We partition the solution and load vectors analogously. As a result, the interior and boundary components of the solution $\bdd{u}$ are decoupled:
\begin{align}
	\label{eq:static condensation solve}
	\vec{u} = \begin{bmatrix}
		\vec{u}_I \\
		\vec{u}_B
	\end{bmatrix}, \qquad \text{where} \quad \bdd{A}_{II} \vec{u}_I = \vec{L}_I \quad \text{and} \quad \bdd{A}_{BB} \vec{u}_B = \vec{L}_B.
\end{align}
We choose to invert $\bdd{A}_{II}$ element-by-element as indicated in \cref{eq:interior solve} using a direct method, while we use the conjugate gradient method with preconditioner $\bdd{P}_{BB}^{-1}$ to compute $\vec{u}_B$. The key point is that the interior degrees of freedom are computed only once, and not every iteration of the iterative solver. In addition, the performance of $\bdd{P}_{BB}^{-1}$ is identical to the performance of $\bdd{P}^{-1}$ since in this case $\cond(\bdd{P}_{BB}^{-1} \bdd{A}_{BB}) = \cond(\bdd{P}^{-1} \bdd{A})$. 

\textbf{Cook's Membrane. } We now demonstrate the performance of the preconditioner $\bdd{P}$ on the Cook's membrane problem mentioned in \cref{sec:intro}. More precisely, the boundary conditions are
\begin{alignat*}{2}
	\bdd{u} &= \bdd{0}, \qquad & &\text{on } \{0\} \times (0, 44), \\
	\bdd{\sigma}(\bdd{u})\unitvec{n} &= (1, 0)^T, \qquad & &\text{on } \{48\} \times (44, 60), \\
	\bdd{\sigma}(\bdd{u})\unitvec{n} &= \bdd{0} \qquad & &\text{on the remainder of $\Gamma$},
\end{alignat*}
where $\bdd{\sigma}(\bdd{u})_{ij} := 2\mu \bdd{\varepsilon}(\bdd{u})_{ij} + \lambda (\dive \bdd{u}) \delta_{ij}$ is the stress tensor. The condition numbers of the full stiffness matrix $\bdd{A}$, the boundary component $\bdd{A}_{BB}$, and the preconditioned boundary system $\bdd{P}_{BB}^{-1} \bdd{A}_{BB}$ with $\mu = 1$ and $\lambda \in \{10^1, 10^3, 10^5, 10^7\}$ are shown in \cref{fig:cooks cns}. Consistent with \cref{eq:static cond precon bound}, the condition numbers of $\bdd{P}_{BB}^{-1} \bdd{A}_{BB}$ are uniformly bounded in $\lambda$ and $p$, and the preconditioned conjugate gradient method converges in a constant number of iterations, independently of $p$ and $\lambda$. 

The benefit of using high polynomial degree is demonstrated in \cref{fig:cook von mises}, which displays the log of the von Mises stress for the $p=16$ solution with $\lambda=10^5$. The von Mises stress \cite{vonmises13}, a failure criterion for a material, is given by $\sigma_v^2 := \sigma_{11}^2 + \sigma_{22}^2 + 3\sigma_{12}^2 - \sigma_{11} \sigma_{22}$, and in particular depends on the gradient of the solution. With a coarse mesh, the $p=16$ solution is able to capture both the singularity at $(0, 44)$ and the smooth profile of $\sigma_v$ away from the singularity.

The residual history for $p \in \{4, 7, 10, 13\}$ for the same choices of $\mu$ and $\lambda$ as above are displayed in \cref{fig:cooks resids}. In order to avoid biasing the results by the choice of data, we show an ensemble of curves for each value of $p$ corresponding to 100 random choices of the vector $\vec{b}$ with entries uniformly distributed in $(-1, 1)$ and zero starting vector. Here, and in the remaining examples, the relative residual is defined to be $\|\vec{r}_k\|_{\ell^2} / \|\vec{r}_0\|_{\ell^2}$, where $r_k := \bdd{A}_{BB} \vec{x}_k - \vec{b}$ is the residual vector at the $k$-th step, and iteration is terminated when the relative residual is smaller than $10^{-12}$. Consistent with the theory (see e.g. \cite[p. 636]{Golub}), conjugate gradient converges in a constant number of iterations that is insensitive to the value of $p$ or $\lambda$, and the relative residual decreases at the geometric rate $(\sqrt{\kappa} - 1)/(\sqrt{\kappa}+1)$, where $\kappa = \cond(\bdd{P}_{BB}^{-1} \bdd{A}_{BB}) \simeq 6.0$.

\begin{figure}[ht]
	\centering
	\begin{subfigure}[b]{0.42\linewidth}
		\centering
		\includegraphics[height=5.2cm]{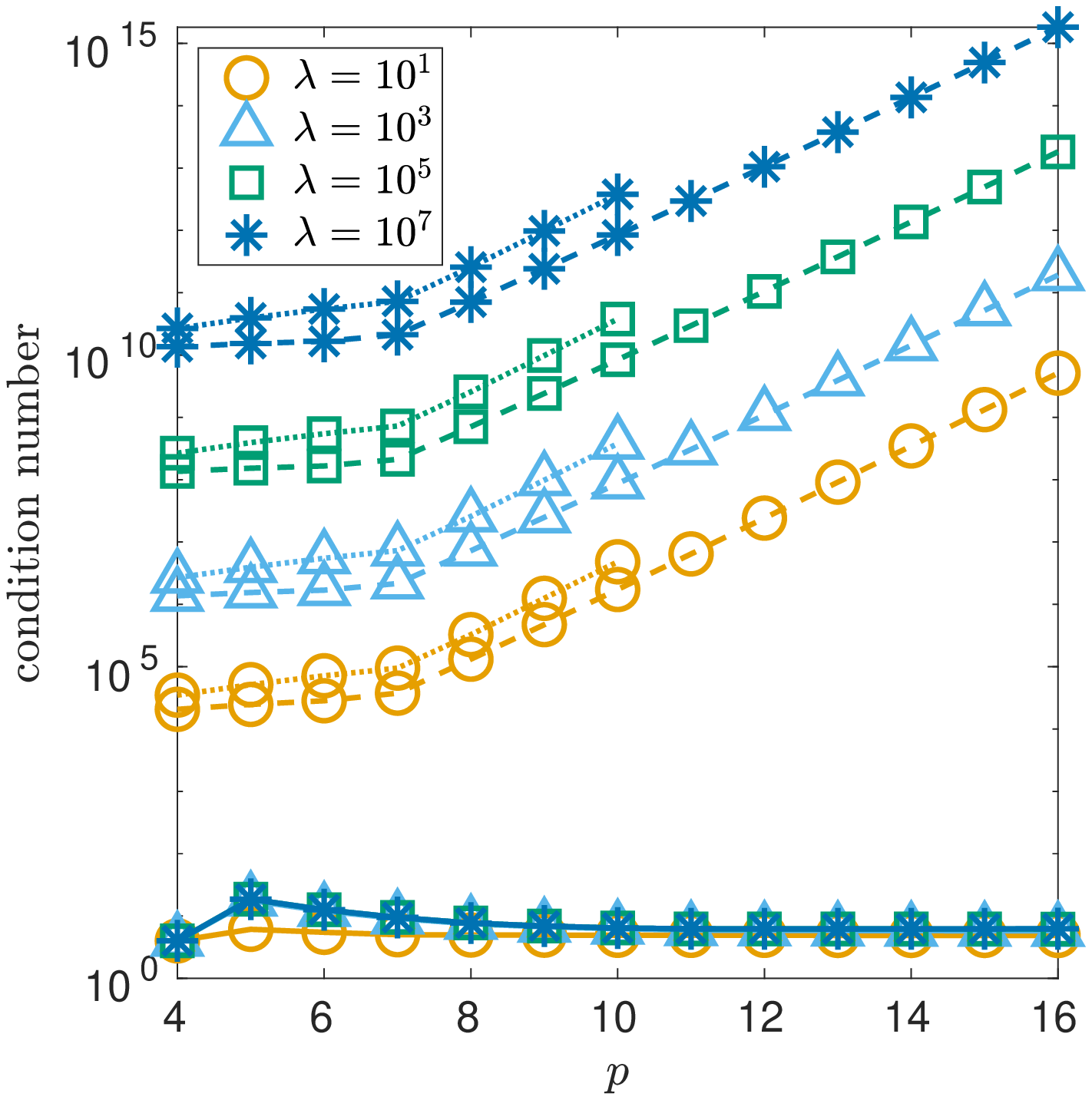}
		\caption{}
		\label{fig:cooks cns}
	\end{subfigure}
	\hfill
	\begin{subfigure}[b]{0.56\linewidth}
		\centering
		\includegraphics[height=5cm]{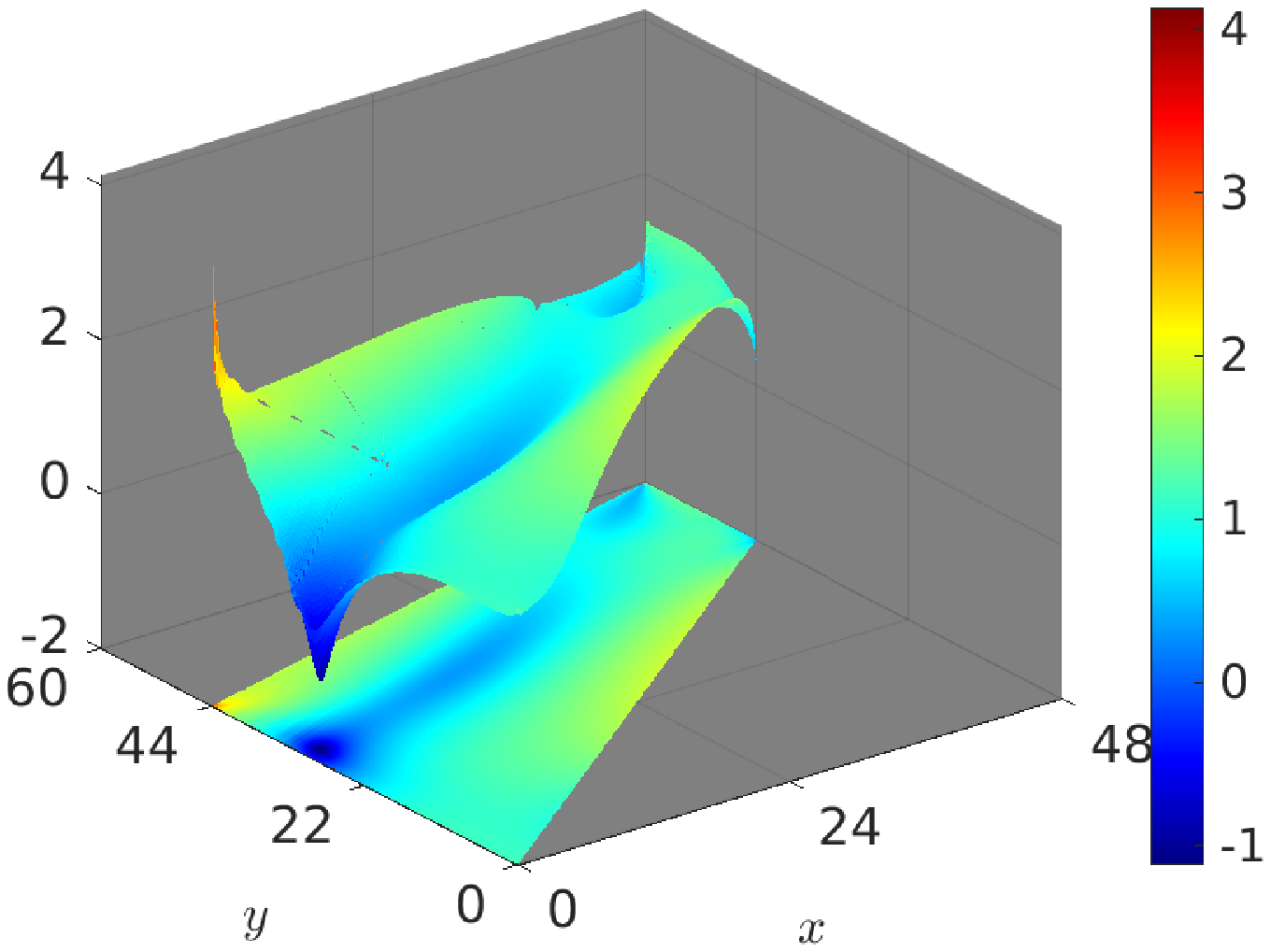}
		\caption{}
		\label{fig:cook von mises}
	\end{subfigure}
	\caption{Cook's membrane problem (a) condition numbers of $\bdd{A}$ (dotted lines), $\bdd{A}_{BB}$ (dashed lines), and $\bdd{P}_{BB}^{-1} \bdd{A}_{BB}$ (solid lines) and (b) log of the von Mises stress for the $p=16$ solution with $\lambda=10^5$. Observe that the condition numbers of $\bdd{P}_{BB}^{-1} \bdd{A}_{BB}$ level out around 6.}
	\label{fig:cooks figures}
\end{figure}

\begin{figure}[ht]
	\centering
	\begin{subfigure}[b]{0.49\linewidth}
		\centering
		\includegraphics[width=\linewidth]{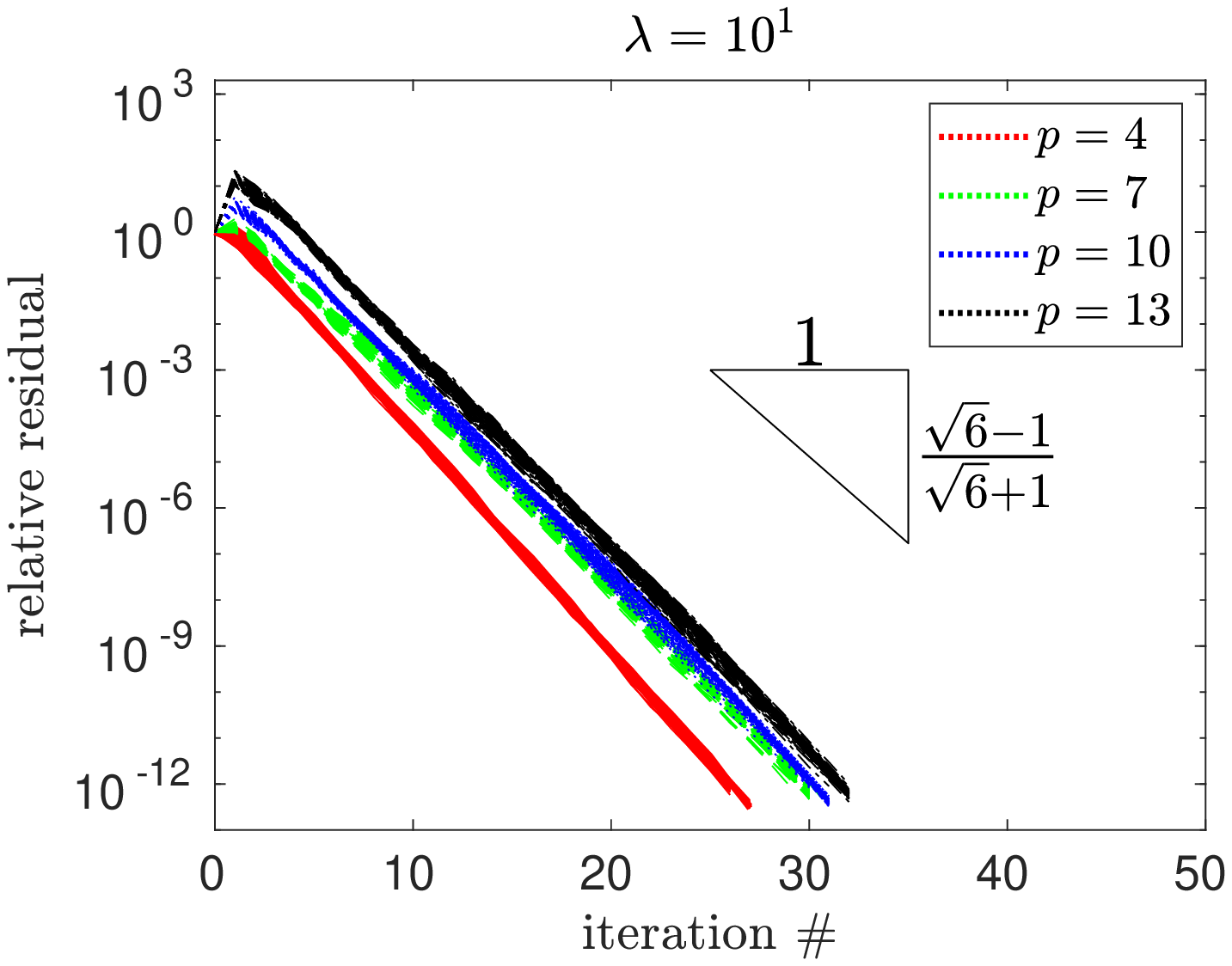}
		\caption{}
		\label{fig:cooks resids 1e1}
	\end{subfigure}
	\hfill
	\begin{subfigure}[b]{0.49\linewidth}
		\centering
		\includegraphics[width=\linewidth]{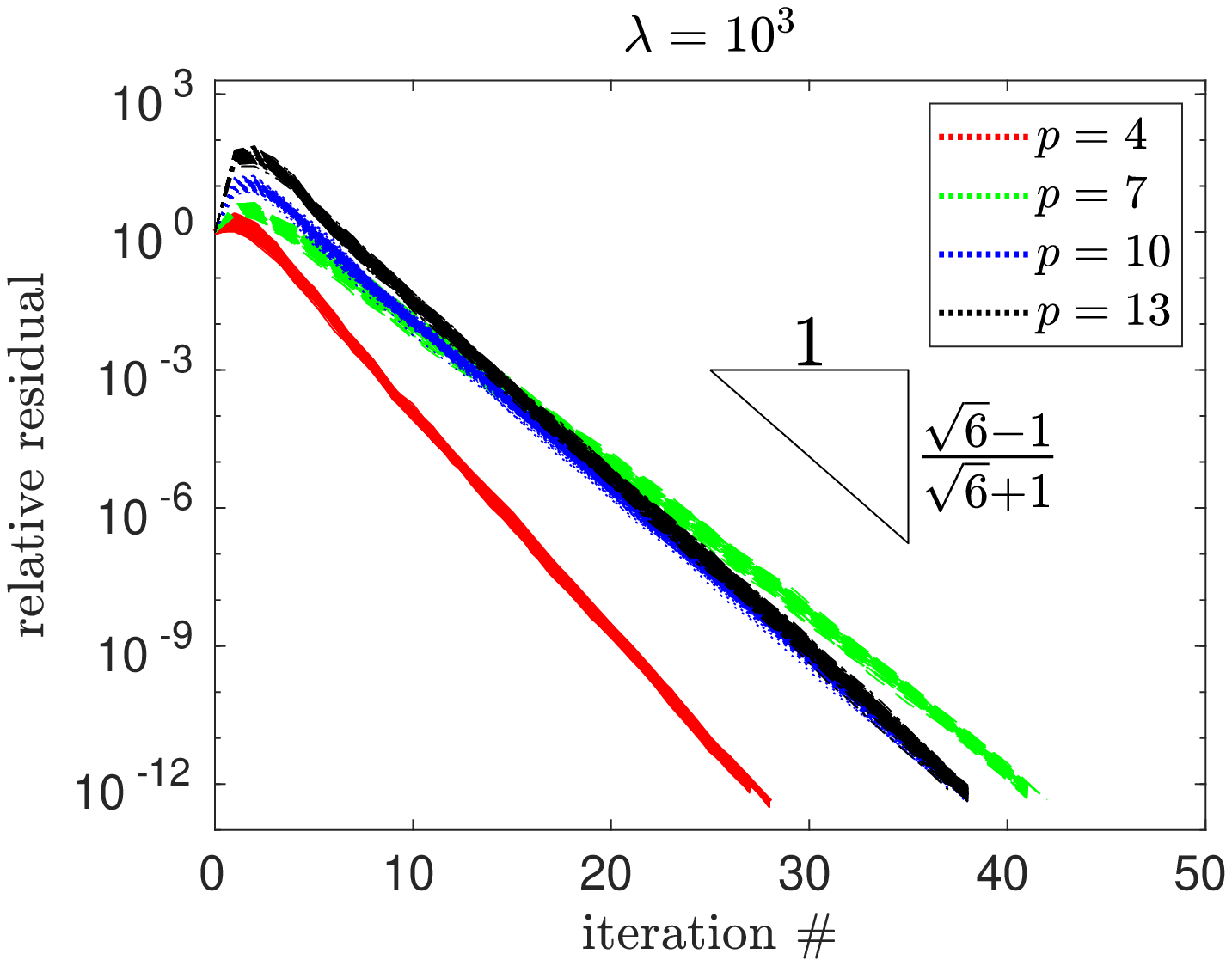}
		\caption{}
		\label{fig:cooks resids 1e3}
	\end{subfigure}
	\\
	\begin{subfigure}[b]{0.49\linewidth}
		\centering
		\includegraphics[width=\linewidth]{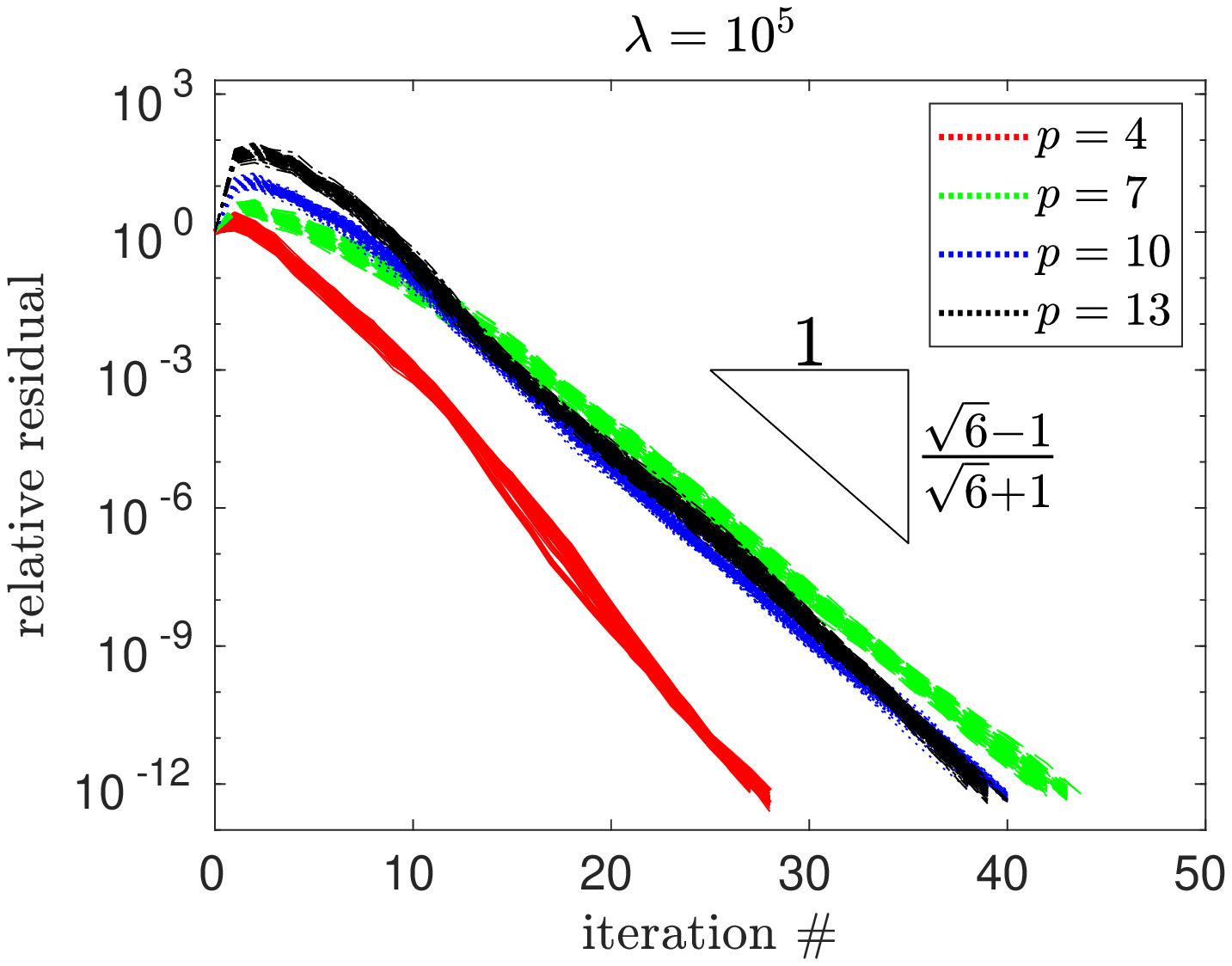}
		\caption{}
		\label{fig:cooks resids 1e5}
	\end{subfigure}
	\hfill
	\begin{subfigure}[b]{0.49\linewidth}
		\centering
		\includegraphics[width=\linewidth]{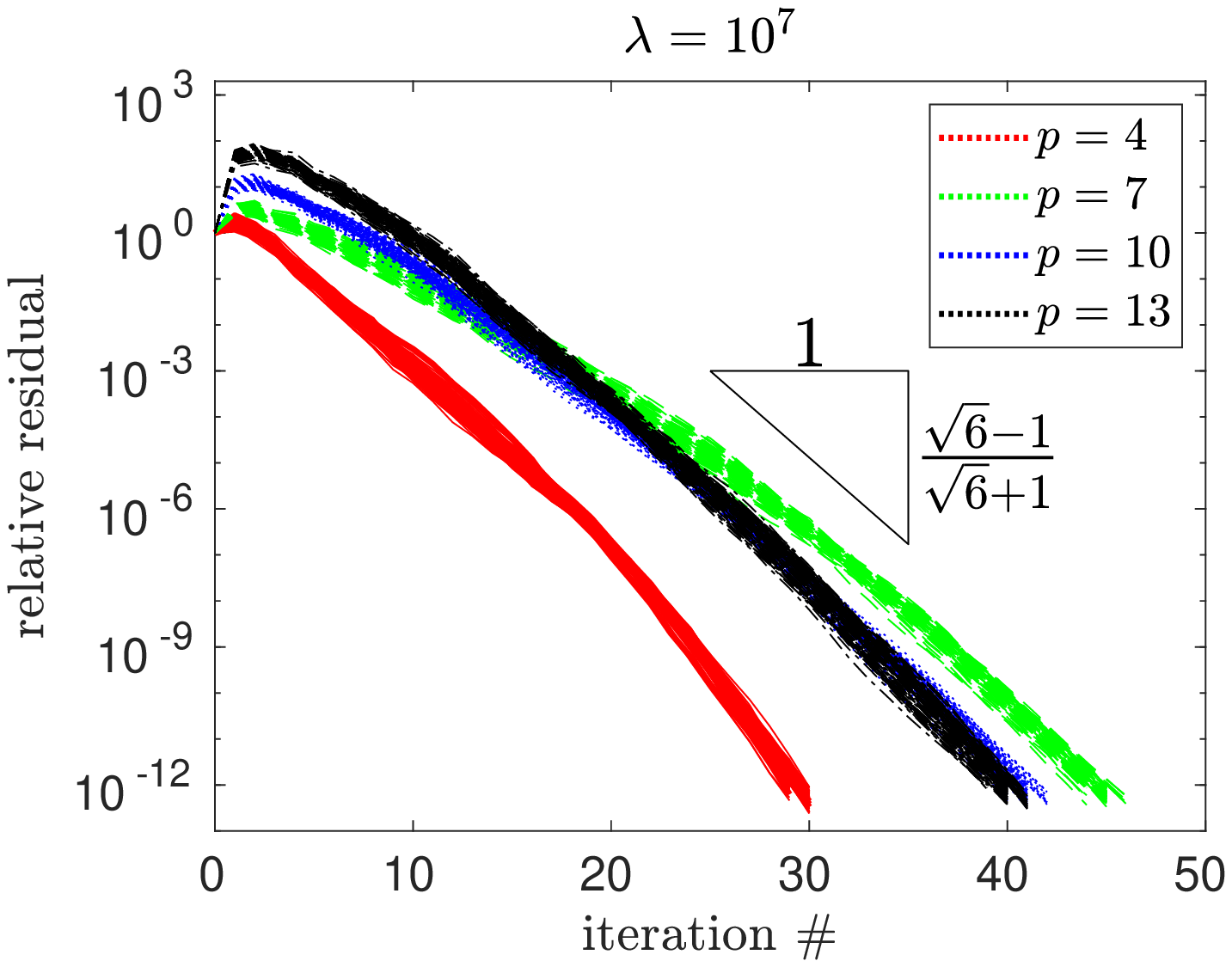}
		\caption{}
		\label{fig:cooks resids 1e7}
	\end{subfigure}
	\caption{Residual history of the conjugate gradient method with preconditioner $\bdd{P}$ applied to the Cook's membrane problem with (a) $\lambda = 10^{1}$, (b) $\lambda = 10^{3}$, (c) $\lambda = 10^{5}$, and (d) $\lambda = 10^{7}$.}
	\label{fig:cooks resids}
\end{figure}

\subsection{The SCIP Method}
\label{sec:scip}

For certain applications \cite{AinCP22SCIP}, it is advantageous to choose the boundary space as follows:
\begin{multline}
	\label{eq:tilde xb def}
	\tilde{\bdd{X}}_B := \{ \bdd{v} \in \bdd{X}_D : a_{\lambda}(\bdd{v}, \bdd{z}) = 0 \ \forall \bdd{z} \in \bdd{X}_I  : \dive \bdd{z} \equiv 0 \\
	\text{and} \quad  (\dive \bdd{v}, \dive \bdd{w}) = 0 \ \forall \bdd{w} \in \bdd{X}_I \}.
\end{multline}
As shown in \cref{rem:tilde tau}, the space $\tilde{\bdd{X}}_B$ satisfies condition \cref{eq:tau condition} with $\tau_B$ independent of $\mu$, $\lambda$, $h$, and $p$ so that \cref{thm:main result} again shows that the associated preconditioner remains effective:
\begin{align}
	\label{eq:scip condition number}
	\cond(\bdd{P}^{-1} \bdd{A}) \leq C  (1 + \beta_X^{-2}).
\end{align}
The space $\tilde{\bdd{X}}_B$ naturally arises in finite element analysis of incompressible flow. For example, as discussed in \cite{AinCP22SCIP}, the pair $\bdd{X}_D \times \dive \bdd{X}_D$, i.e. Scott-Vogelius elements, is an attractive option for high order mixed finite element discretization of Stokes flow: (1) they enforce the divergence-free constraint exactly (i.e. the divergence of the discrete solution vanishes pointwise); and (2) the inf-sup constant $\beta_X$ \cref{eq:inf-sup global} is bounded away from zero uniformly in $h$ and $p$ as mentioned in \cref{sec:asm}. The Statically Condensed Iterated Penalty (SCIP) method \cite{AinCP22SCIP} is an efficient method to compute the discrete solution. The heart of the SCIP algorithm requires the solution of the following variational problem:
\begin{align}
	\label{eq:scip bilinear}
	\tilde{\bdd{u}} \in \tilde{\bdd{X}}_B : \qquad a_{\lambda}(\tilde{\bdd{u}}, \bdd{v}) = \tilde{L}(\bdd{v}) \qquad \forall \bdd{v} \in \tilde{\bdd{X}}_B,
\end{align}
which must be solved with data $\tilde{L}$ that changes at each iteration. To apply the preconditioner $P^{-1}$ to problem \cref{eq:scip bilinear}, only steps \cref{eq:coarse solve,eq:spdier solve} of \cref{alg:asm variational} are required since only the boundary component is sought.

\textbf{Moffatt Eddy Problem. } We now demonstrate the performance of the preconditioner $P^{-1}$ on the SCIP bilinear form \cref{eq:scip bilinear} on a problem due to Moffatt \cite{Moffatt64}. The geometry of the problem and the computational mesh are displayed in \cref{fig:moffatt mesh}, while the boundary conditions are
\begin{align*}
	\bdd{u}(x, 0) = (1-x^2, 0)^T, \quad -1 < x < 1, \quad \text{and} \quad \bdd{u} = \bdd{0} \quad \text{on the remainder of $\Gamma$}.
\end{align*}
The numerical results in \cite{AinCP22SCIP} show that high order elements nicely capture the range of scales exhibited in the true solution. The condition numbers of the stiffness matrix $\bdd{A}$ corresponding to \cref{eq:scip bilinear} and the preconditioned system $\bdd{P}^{-1} \bdd{A}$ are displayed in \cref{fig:moffatt cns}, while the residual histories for the conjugate gradient method with the same setup as in Cook's membrane example in \cref{sec:static condensation} is shown in \cref{fig:moffatt resids}. Consistent with \cref{eq:scip condition number}, the condition numbers are uniformly bounded in $p$ and $\lambda$, and the conjugate gradient method again converges within a fixed number of iterations independent of $p$ and $\lambda$. In particular, the relative residual decreases geometrically as $(\sqrt{\kappa} - 1)/(\sqrt{\kappa}+1)$, where $\kappa = \cond(\bdd{P}^{-1} \bdd{A}) \simeq 5.0$.

\begin{figure}[ht]
	\centering
	\begin{subfigure}[b]{0.46\linewidth}
		\centering
		\includegraphics[height=6.4cm]{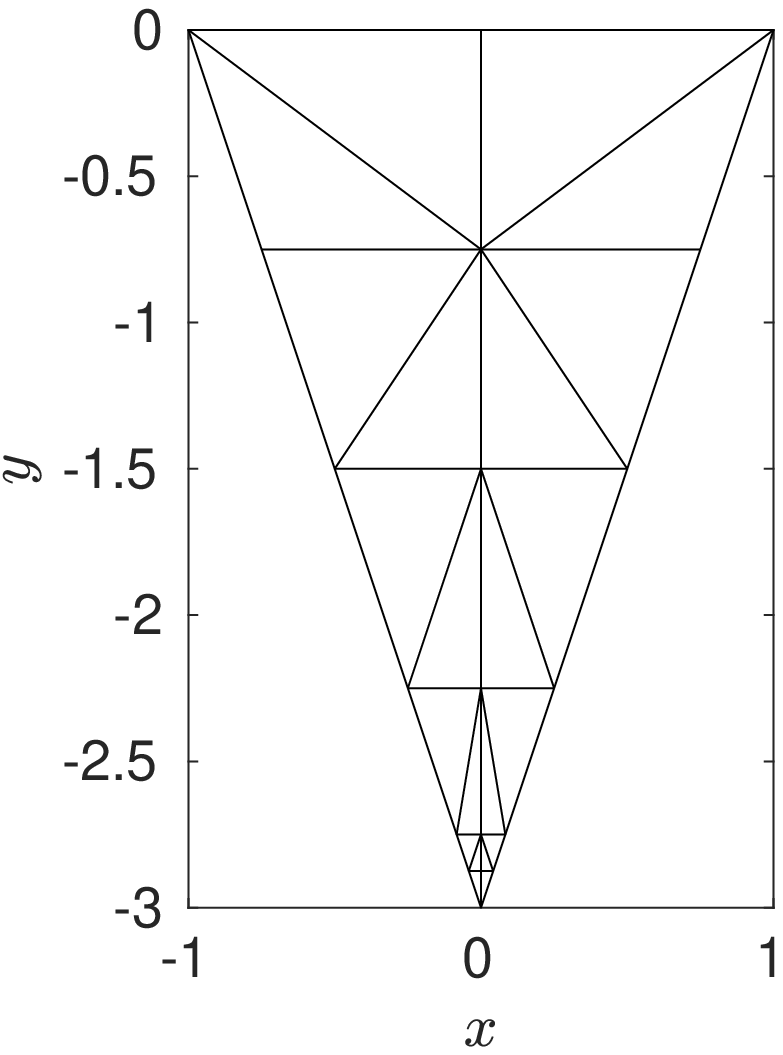}
		\caption{}
		\label{fig:moffatt mesh}
	\end{subfigure}
	\hfill
	\begin{subfigure}[b]{0.51\linewidth}
		\centering
		\includegraphics[height=6.4cm]{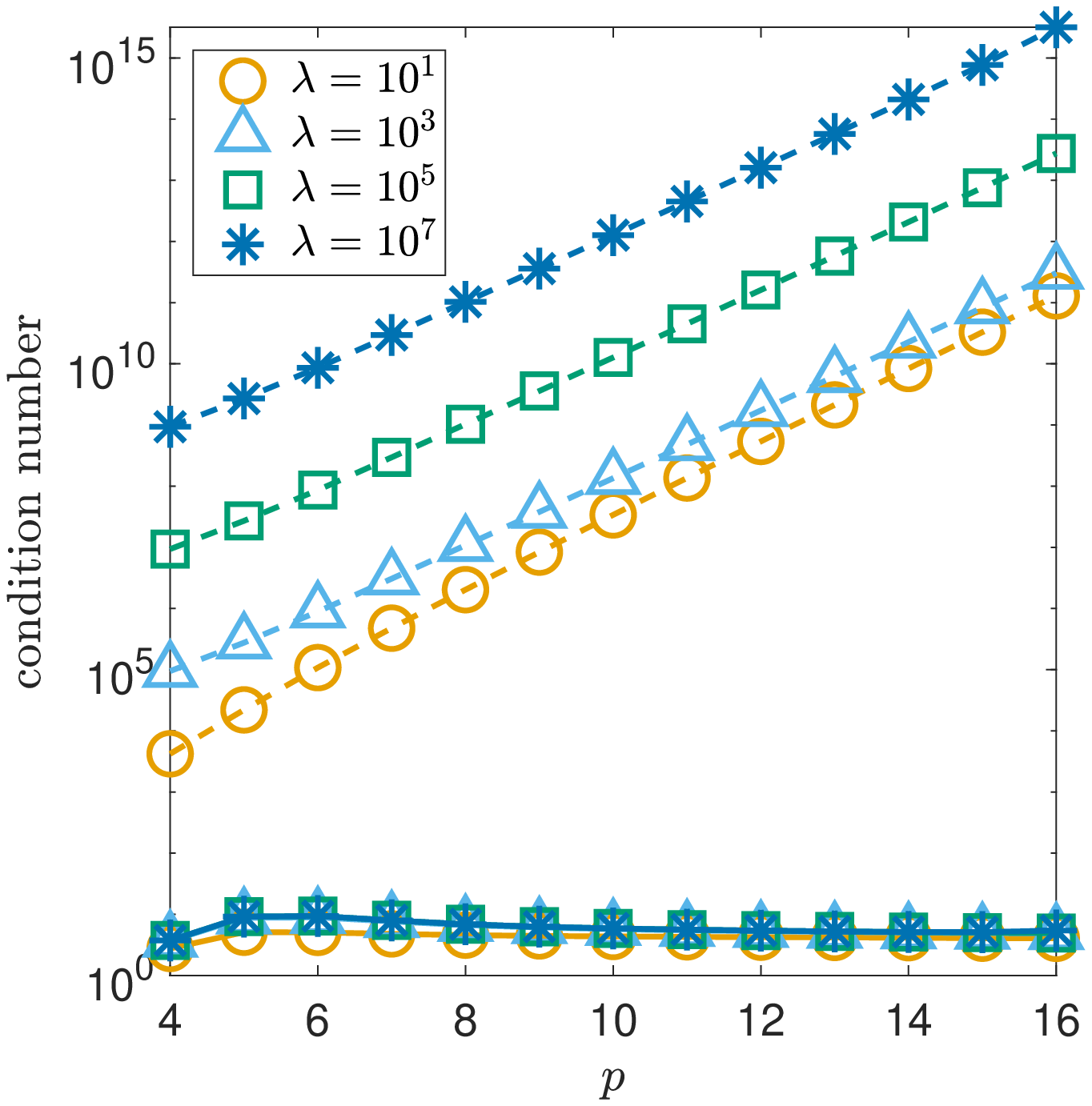}
		\caption{}
		\label{fig:moffatt cns}
	\end{subfigure}
	\caption{Moffatt eddies problem (a) mesh and (b) condition numbers of $\bdd{A}$ (dashed lines) and $\bdd{P}^{-1} \bdd{A}$ (solid lines). Observe that the condition numbers of $\bdd{P}^{-1} \bdd{A}$ level out around 5.}
	\label{fig:moffatt figures}
\end{figure}

\begin{figure}[ht]
	\centering
	\begin{subfigure}[b]{0.49\linewidth}
		\centering
		\includegraphics[width=\linewidth]{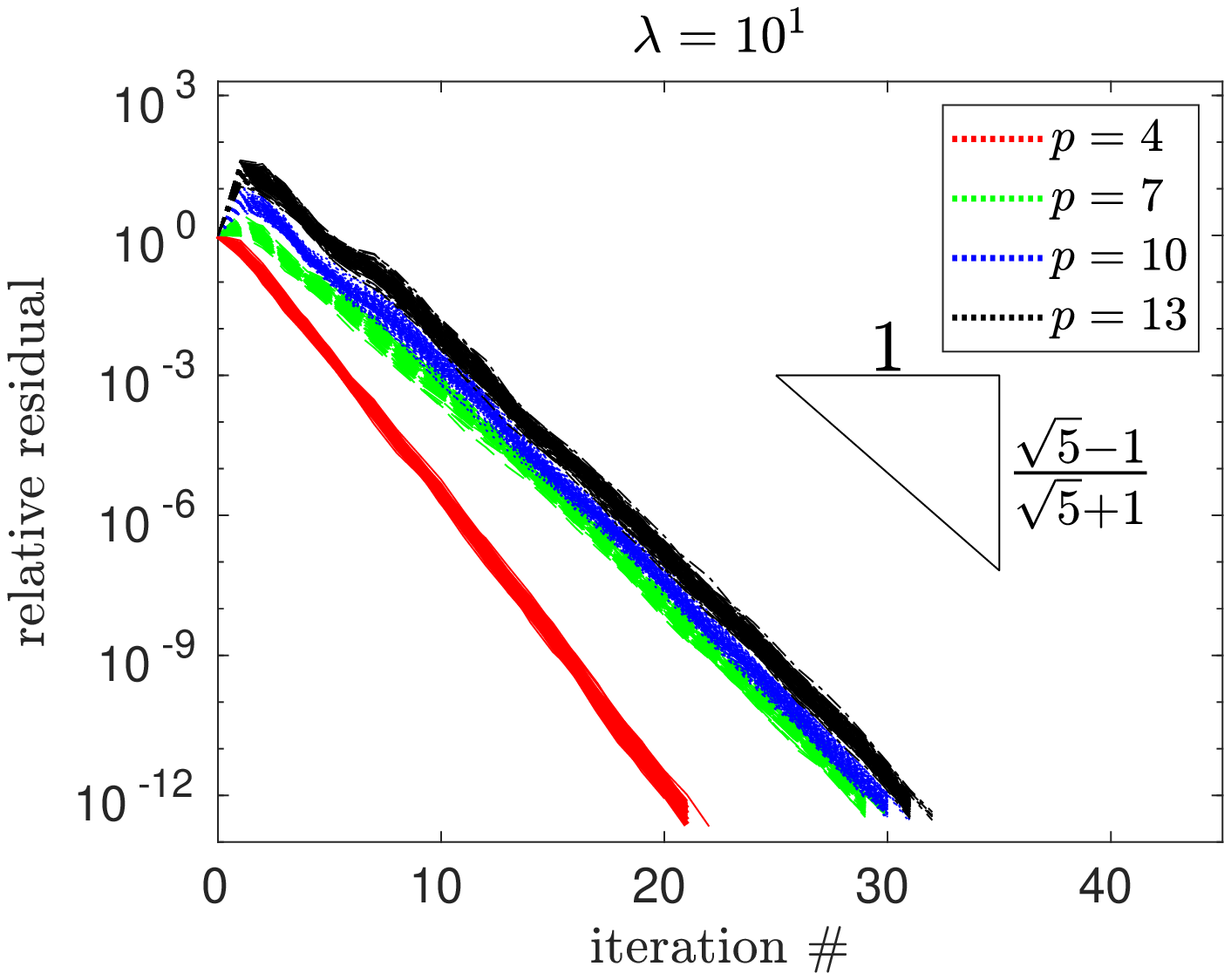}
		\caption{}
		\label{fig:moffatt resids 1e1}
	\end{subfigure}
	\hfill
	\begin{subfigure}[b]{0.49\linewidth}
		\centering
		\includegraphics[width=\linewidth]{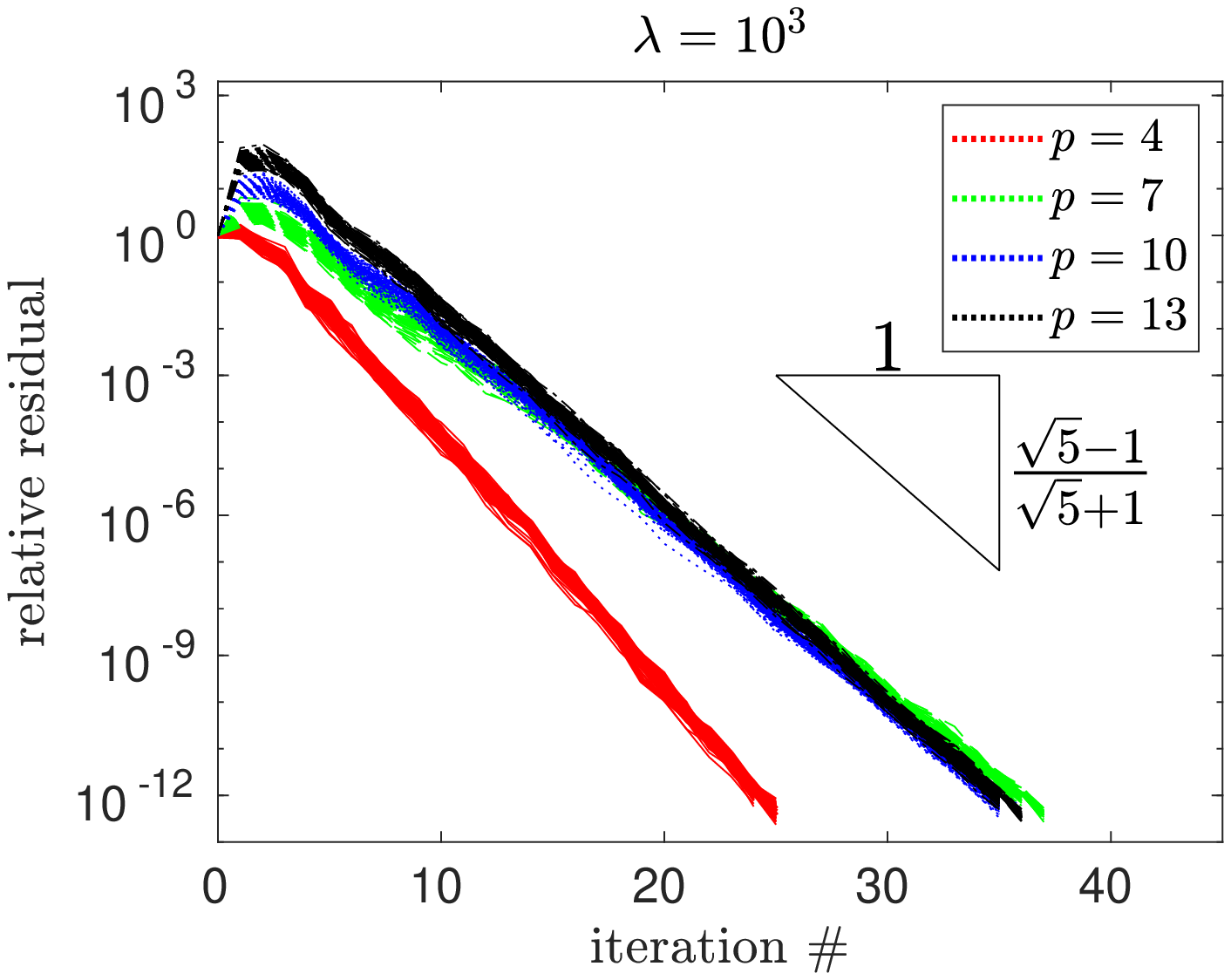}
		\caption{}
		\label{fig:moffatt resids 1e3}
	\end{subfigure}
	\\
	\begin{subfigure}[b]{0.49\linewidth}
		\centering
		\includegraphics[width=\linewidth]{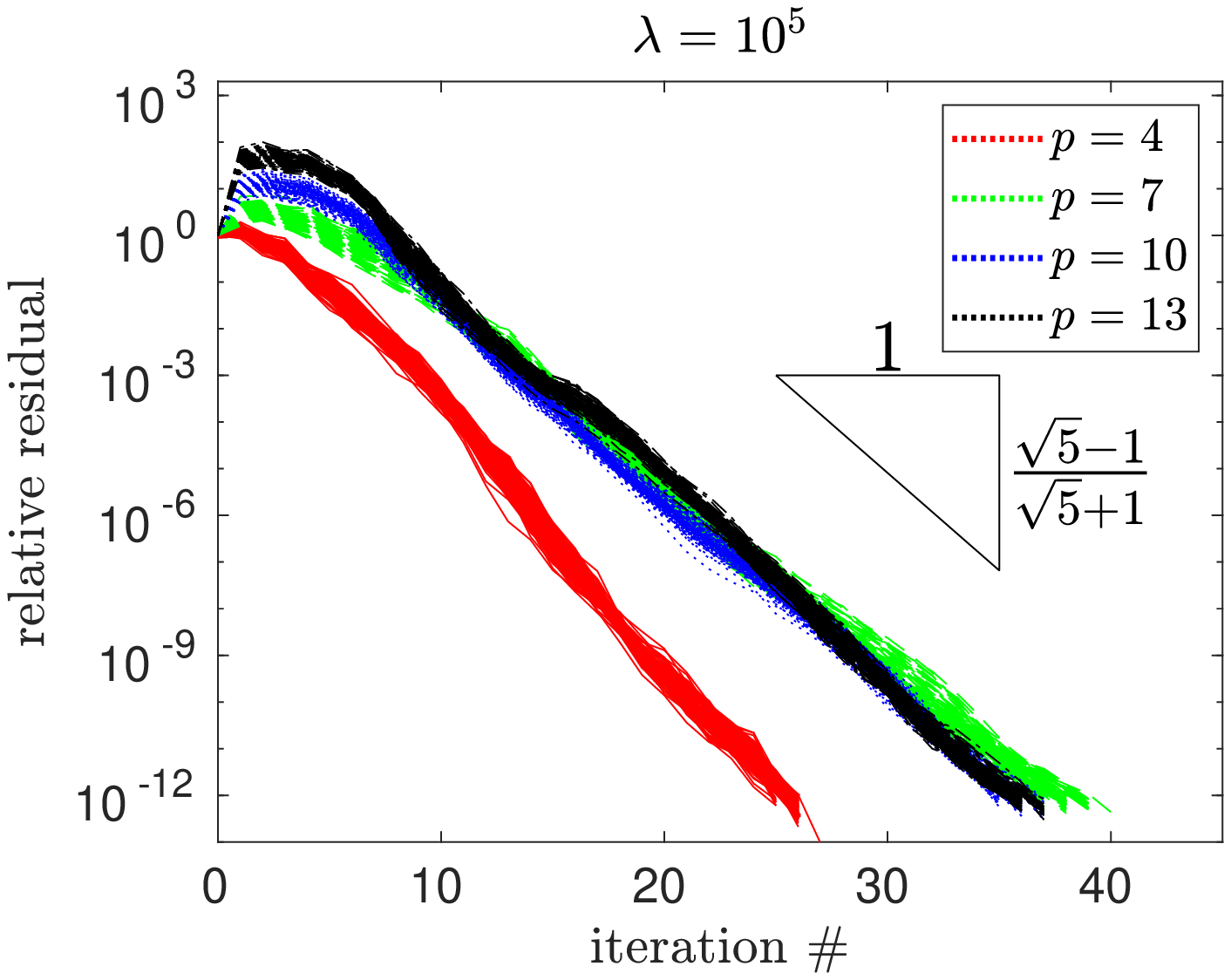}
		\caption{}
		\label{fig:moffatt resids 1e5}
	\end{subfigure}
	\hfill
	\begin{subfigure}[b]{0.49\linewidth}
		\centering
		\includegraphics[width=\linewidth]{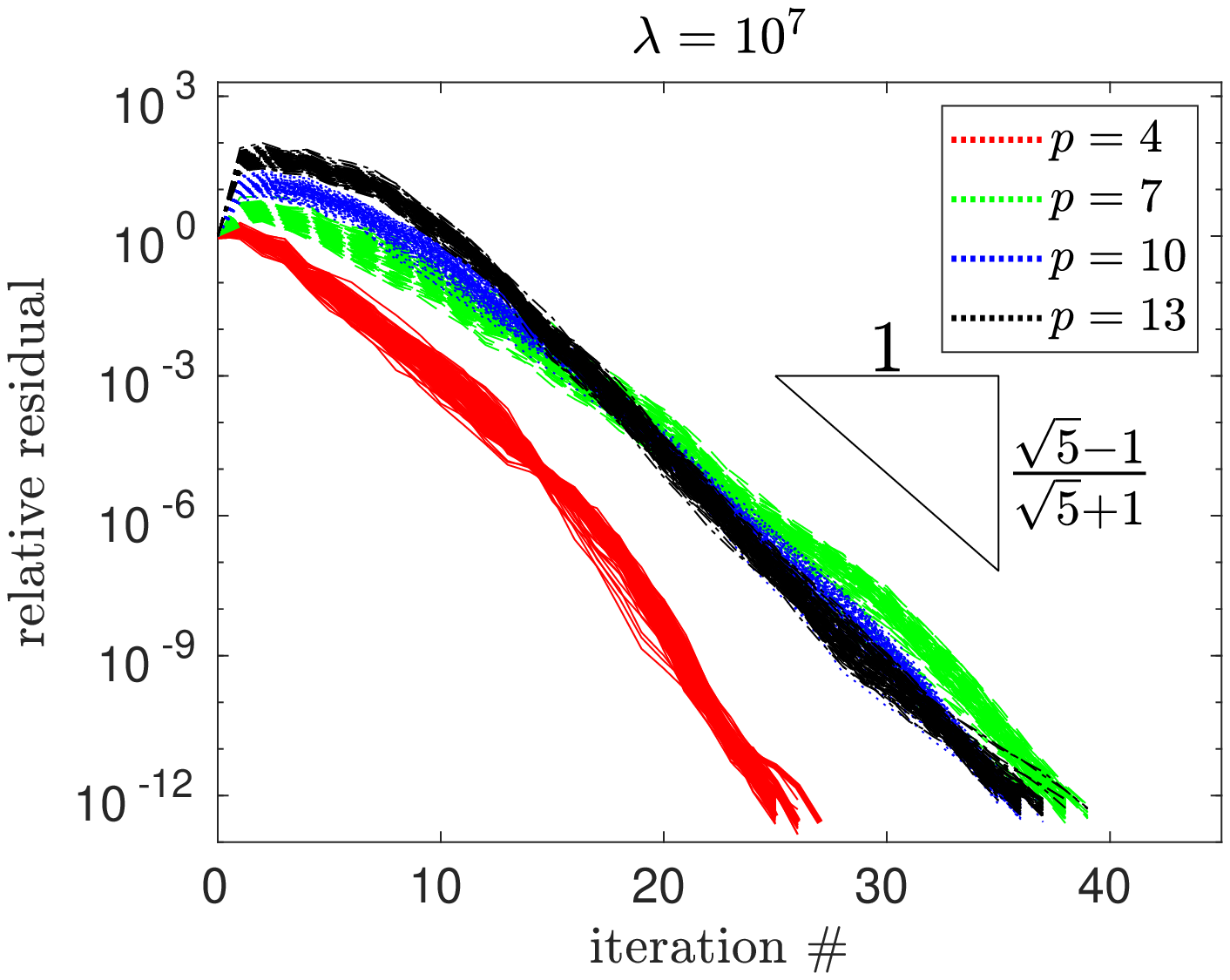}
		\caption{}
		\label{fig:moffatt resids 1e7}
	\end{subfigure}
	\caption{Residual history of the conjugate gradient method with preconditioner $\bdd{P}$ applied to the Moffatt eddies problem with (a) $\lambda = 10^{1}$, (b) $\lambda = 10^{3}$, (c) $\lambda = 10^{5}$, and (d) $\lambda = 10^{7}$.}
	\label{fig:moffatt resids}
\end{figure}

\subsection{Inexact Interior Solves}
\label{sec:inexact}

The most computationally expensive step of static condensation described in \cref{sec:static condensation} is the interior solve $\bdd{A}_{II} \vec{u}_I = \vec{L}_I$ in \cref{eq:static condensation solve}, which incurs a one-time cost of $\mathcal{O}(|\mathcal{T}| p^6)$ operations to factor the block diagonal matrix $\bdd{A}_{II}$ along with $\mathcal{O}(|\mathcal{T}| p^4)$ operations to apply the action of $\bdd{A}_{II}^{-1}$. One way in which to reduce the cost, at the expense of solving for the interior degrees of freedom on every iteration, is to replace $\bdd{A}_{II}$ with a matrix $\bdd{B}_{II}$ which can be inverted efficiently.

Let $b_{\lambda}(\cdot,\cdot)$ denote the bilinear form associated with such a matrix $\bdd{B}_{II}$ and suppose that
\begin{align}
	\label{eq:blam equiv alam}
	C_1 b_{\lambda}(\bdd{v}, \bdd{v}) \leq a_{\lambda}(\bdd{v}, \bdd{v}) \leq C_2 b_{\lambda}(\bdd{v}, \bdd{v}) \qquad \forall \bdd{v} \in \bdd{X}_I,
\end{align}
where $C_1$ and $C_2$ are independent of $\mu$, $\lambda$, $h$, and $p$. The corresponding boundary space is then
\begin{align}
	\label{eq:inexact interior choice}
	\check{\bdd{X}}_B := \{ \bdd{v} \in \bdd{X}_D : b_{\lambda}(\bdd{v}, \bdd{w}) = 0 \ \forall \bdd{w} \in \bdd{X}_I \}. 
\end{align}
We additionally replace $a_{\lambda, K}(\cdot,\cdot)$ with $b_{\lambda, K}(\cdot,\cdot)$ in \cref{eq:interior solve} in \cref{alg:asm variational} to obtain a new matrix preconditioner $\check{\bdd{P}}^{-1}$. We may argue similarly to the proof of \cref{thm:main result} to show that
\begin{align}
	\label{eq:inexact bilinear element perf}
	\cond(\check{\bdd{P}}^{-1} \bdd{A}) \leq C  (1 + \beta_X^{-2})(1 + \tau_B^2),
\end{align}  
where $C$ is independent of $\beta_X$, $\tau_B$, $\mu$, $\lambda$, $h$, and $p$. In summary, the use of the inexact interior solve $\bdd{B}_{II}$ in conjunction with the corresponding boundary space $\check{\bdd{X}}_B$ results in a preconditioner $\check{\bdd{P}}^{-1}$ whose effectiveness is the same as static condensation. One possible choice for $b_{\lambda}(\cdot,\cdot)$ satisfying the above conditions is described in \cref{sec:inexact interior}.

\section{The Stokes Extension Operator}
\label{sec:stokes extension}

Given a measurable subset $\omega$ of $\Omega$ or $\partial \Omega$, let $(\cdot,\cdot)_{\omega}$ denote the $L^2(\omega)$ or $\bdd{L}^2(\omega)$ inner product. More generally,	$|\cdot|_{s,\omega}$ and $\|\cdot\|_{s,\omega}$ denote the $H^s(\omega)$ or $\bdd{H}^s(\omega)$ semi-norm and norm, respectively.  We omit the subscript $\omega$ when $\omega = \Omega$. Here, and in what follows, $C > 0$ denotes a generic constant independent of $\mu$, $\lambda$, $h$, and $p$.

Let $\bdd{X} := X \times X$ and define the discrete Stokes extension pair $\mathbb{S} : \bdd{X} \to \bdd{X}$ and $\mathbb{Q} : \bdd{X} \to \dive \bdd{X}_I$ \cite[\S 6]{AinCP22SCIP} as follows: for $\bdd{u} \in \bdd{X}$ and $K \in \mathcal{T}$,
\begin{subequations}
	\label{eq:stokes extension}
	\begin{alignat}{2}
		\label{eq:stokes extension 1}
		a_K(\mathbb{S} \bdd{u}, \bdd{v}) - (\mathbb{Q} \bdd{u}, \dive \bdd{v})_K &= 0 \qquad & & \forall \bdd{v} \in \bdd{X}_I(K), \\
		\label{eq:stokes extension 2}
		-(r, \dive \mathbb{S}\bdd{u})_K &= 0 \qquad & & \forall r \in \dive \bdd{X}_I(K), \\
		\label{eq:stokes extension 3}
		(\bdd{u} - \mathbb{S} \bdd{u})|_{\partial K} &= \bdd{0}, \qquad & &
	\end{alignat}
\end{subequations}
where $a_K(\cdot, \cdot) = 2\mu(\bdd{\varepsilon}(\cdot), \bdd{\varepsilon}(\cdot))_K$. For instance, given any rigid body motion $\bdd{r} \in \bdd{RM} := \mathbb{R}^2 \oplus \spann\{ (-y, x)^T \}$, the pair  $(\bdd{r}, 0)$ satisfies \cref{eq:stokes extension}, and so $(\mathbb{S} \bdd{r}, \mathbb{Q} \bdd{r}) = (\bdd{r}, 0)$.

\begin{remark}
	\label{rem:stokes extension boundary}
	Condition \cref{eq:stokes extension 3} shows that for $\bdd{u} \in \bdd{X}$, the Stokes extension pair $(\mathbb{S} \bdd{u}, \mathbb{Q} \bdd{u})$ really only depends on the values of $\bdd{u}|_{\partial K}$ for $K \in \mathcal{T}$, which lie in the trace space
	\begin{align*}
		\Tr \bdd{X} &:= \left\{ \bdd{f} : \bigcup_{K \in \mathcal{T}} \partial K \to \mathbb{R}^2 : \text{ $\bdd{f}|_{\partial K} = \bdd{u}|_{\partial K}$ for some $\bdd{u} \in \bdd{X}$} \ \forall K \in \mathcal{T}  \right\} \\
		&= \left\{ \bdd{f} : \bigcup_{K \in \mathcal{T}} \partial K \to \mathbb{R}^2 : \bdd{f}|_{\gamma} \in \bm{\mathcal{P}}_{p}(\gamma) \ \forall \gamma \in \mathcal{E} \text{ and $\bdd{f}$ is continuous} \right\}.
	\end{align*}
	Here, the second equality is a consequence of the following fact:
	Given a function $\bdd{g} \in  \{ \bdd{f} : \cup_{K \in \mathcal{T}} \partial K \to \mathbb{R}^2 : \bdd{f}|_{\gamma} \in \bm{\mathcal{P}}_{p}(\gamma) \text{ and $\bdd{f}$ is continuous} \}$, applying \cite[Theorem 7.4]{BCMP91} element-by-element shows that there exists $\bdd{u} \in \bdd{X}$ satisfying $\bdd{u}|_{\partial K} = \bdd{g}|_{\partial K}$ for all $K \in \mathcal{T}$. As a result, the operators $(\mathbb{S}, \mathbb{Q})$ could be defined equally well as mapping $\Tr \bdd{X}$ to $(\bdd{X}, \dive \bdd{X})$. Henceforth, we shall not distinguish whether the domain of $(\mathbb{S}, \mathbb{Q})$ is $\bdd{X}$ or $\Tr \bdd{X}$. 
\end{remark}

The following properties of the operator $\mathbb{S}$ play a key role in the analysis of the preconditioner $P^{-1}$:
\begin{theorem}
	\label{thm:stokes continuity element}
	For all $K \in \mathcal{T}$, the discrete Stokes extension satisfies
	\begin{alignat}{2}
		\label{eq:stokes seminorm continuity element}
		\| \bdd{\varepsilon}(\mathbb{S} \bdd{u})\|_{K} + \mu^{-1} \|  \mathbb{Q} \bdd{u}\|_{K} &
		\leq C \inf_{ \substack{ \bdd{w} \in \bm{\mathcal{P}}_{p}(K) \\ \bdd{w}|_{\partial K} = \bdd{u}|_{\partial K} } } \| \bdd{\varepsilon}(\bdd{w})\|_{K} \qquad & &\forall \bdd{u} \in \bdd{X}, \\
		\label{eq:stokes continuity element}
		h_K^{-1} \| \mathbb{S} \bdd{u}\|_{K} + | \mathbb{S} \bdd{u}|_{1,K} + \mu^{-1} \| \mathbb{Q} \bdd{u}\|_{K} 
		&\leq C \left\{ h_{K}^{-1/2} \|\bdd{u}\|_{\partial K} + |\bdd{u}|_{1/2,\partial K} \right\} \qquad & &\forall \bdd{u} \in \bdd{X}, \\
		\label{eq:stokes approximation}
		h_K^{-1} \| \bdd{u} - \mathbb{S} \bdd{u} \|_K + | \bdd{u} - \mathbb{S} \bdd{u} |_{1,K} &\leq C \| \bdd{\varepsilon}(\bdd{u})\|_K \qquad & &\forall \bdd{u} \in \bdd{X}.
	\end{alignat}
\end{theorem}
\begin{proof}
	Let $K \in \mathcal{T}$. For $\bdd{u} \in \bdd{X}$, let $\bdd{w} \in \bm{\mathcal{P}}_p(K)$ be any polynomial with $\bdd{w}|_{\partial K} = \bdd{u}|_{\partial K}$. Choosing $\bdd{v} = \mathbb{S} \bdd{u} - \bdd{w}$ in \cref{eq:stokes extension 1} gives
	\begin{align*}
		2\mu \| \bdd{\varepsilon}(\mathbb{S} \bdd{u})\|_{K}^2 = a_K(\mathbb{S} \bdd{u}, \mathbb{S} \bdd{u}) = a_K(\mathbb{S} \bdd{u}, \bdd{w}) - (\mathbb{Q} \bdd{u}, \dive \bdd{w})_K,
	\end{align*}
	which, on applying the Cauchy-Schwarz inequality, gives
	\begin{align*}
		2\mu \| \bdd{\varepsilon}(\mathbb{S} \bdd{u})\|_{K}^2 \leq  \left( 2\mu \| \bdd{\varepsilon}(\mathbb{S} \bdd{u})\|_{K} + \sqrt{2} \| \mathbb{Q} \bdd{u}\|_K \right) \| \bdd{\varepsilon}(\bdd{w})\|_{K}.
	\end{align*}
	Thanks to \cite[Theorem 3.3]{AinCP19StokesIII}, there exists $\bdd{v} \in \bdd{X}_I(K)$ satisfying $\dive \bdd{v} = \mathbb{Q} \bdd{u}$ and $|\bdd{v}|_{1,K} \leq C \|\mathbb{Q} \bdd{u}\|_K$. As a result, there holds
	\begin{align*}
		C^{-1} \|\mathbb{Q} \bdd{u}\|_{K} \leq \frac{(\mathbb{Q} \bdd{u}, \dive \bdd{v})}{|\bdd{v}|_{1,K}} = \frac{a_K(\mathbb{S} \bdd{u}, \bdd{v})}{|\bdd{v}|_{1,K}} \leq 2\mu \| \bdd{\varepsilon}(\mathbb{S} \bdd{u})\|_{K},
	\end{align*}
	where we used \cref{eq:stokes extension 1} and that $\|\bdd{\varepsilon}(\bdd{v})\|_{K} \leq |\bdd{v}|_{1,K}$. Collecting results, we have
	\begin{align}
		\label{eq:proof:seminorm est}
		\| \bdd{\varepsilon}(\mathbb{S} \bdd{u})\|_{K} + \mu^{-1} \| \mathbb{Q} \bdd{u}\|_{K} \leq C\| \bdd{\varepsilon}(\bdd{w})\|_{K},
	\end{align}
	which completes the proof of \cref{eq:stokes seminorm continuity element}. 
	
	Thanks to the relation 	
	\begin{align}
		\label{eq:proof:korn inequality boundary}
		h_K^{-1} \|\bdd{v}\|_{K} + |\bdd{v}|_{1,K} \leq C \left( h_K^{-1/2} \|\bdd{v}\|_{\partial K} +	\| \bdd{\varepsilon}(\bdd{v})\|_{K} \right) \qquad \forall \bdd{v} \in \bdd{H}^1(K),
	\end{align}	
	which follows from a standard scaling argument using Korn's inequality (see e.g \cite[Theorem 11.2.6]{Brenner08}) and the compactness of the embedding $\bdd{H}^1(K) \hookrightarrow \bdd{L}^2(K)$, we obtain
	\begin{align*}
		h_K^{-1} \| \mathbb{S} \bdd{u}\|_{K} + | \mathbb{S} \bdd{u}|_{1,K} + \mu^{-1} \| \mathbb{Q} \bdd{u}\|_{K}  \leq C \left( h_K^{-1/2} \|\bdd{u}\|_{\partial K} + \|\bdd{\varepsilon}(\bdd{w})\|_{K} \right).
	\end{align*}
	By \cite[Theorem 7.4]{BCMP91} and a standard scaling argument, $\bdd{w}$ may be chosen so that
	\begin{align*}
		h_K^{-1} \|\bdd{w}\|_{K} + |\bdd{w}|_{1, K} \leq C \left\{ h_K^{-1/2} \|\bdd{u}\|_{ \partial K} + |\bdd{u}|_{1/2, \partial K} \right\},
	\end{align*}
	which completes the proof of \cref{eq:stokes continuity element}.
	
	Finally, we prove \cref{eq:stokes approximation} using \cref{eq:stokes continuity element} and the trace theorem to obtain
	\begin{align*}
		\sum_{j=0}^{1} h_K^{j-1} |\mathbb{S} \bdd{u} |_{j,K} &\leq C \left\{ h_{K}^{-1/2} \|\bdd{u}\|_{\partial K} + |\bdd{u}|_{1/2,\partial K} \right\}  \leq C  \sum_{j=0}^{1} h_K^{j-1} | \bdd{u}|_{j,K},
	\end{align*}
	and hence
	\begin{align*}
		\sum_{j=0}^{1} h_K^{j-1} |\bdd{u} - \mathbb{S} \bdd{u} |_{j,K} \leq C  \sum_{j=0}^{1} h_K^{j-1} | \bdd{u}|_{j,K}.
	\end{align*}
	As mentioned above, $\mathbb{S} \bdd{r} = \bdd{r}$ for $\bdd{r} \in \bdd{RM}$, and so
	\begin{align*}
		\sum_{j=0}^{1} h_K^{j-1} | \bdd{u} - \mathbb{S} \bdd{u} |_{j,K} \leq C  \inf_{\bdd{r} \in \bdd{RM}} \sum_{j=0}^{1} h_K^{j-1} | \bdd{u} - \bdd{r}|_{j,K} \leq C  \| \bdd{\varepsilon}(\bdd{u})\|_K,
	\end{align*}
	where we used \cref{eq:h1 korn inequality} in the last step. 
\end{proof}

\Cref{rem:stokes extension boundary} shows that the Stokes extension $\mathbb{S} \bdd{u}$ depends only on the values of $\bdd{u}$ on element boundaries, while \cref{eq:stokes seminorm continuity element} means that the operator $\mathbb{S}$ is stable in strain. In a similar vein, the operator $\mathbb{T}_B \bdd{u}$ defined in \cref{sec:constructing asm} also only depends on $\bdd{u}$ on element boundaries, and condition \cref{eq:tau condition} states that $\mathbb{T}_B$ is also stable in strain. This suggests that the strain of $\mathbb{S} \bdd{u}$ and $\mathbb{T}_B \bdd{u}$ should be comparable. The following lemma shows more, namely, that the \textit{energies} of $\mathbb{S} \bdd{u}$ and $\mathbb{T}_B \bdd{u}$ are comparable:
\begin{lemma}
	\label{lem:eq:t and s equivalences}
	For all $\bdd{v} \in \bdd{X}_D$, there holds
	\begin{align}
		\label{eq:t and s equivalences}
		C a_{\lambda}(\mathbb{S} \bdd{v}, \mathbb{S} \bdd{v})  \leq a_{\lambda}(\mathbb{T}_B \bdd{v}, \mathbb{T}_B \bdd{v}) \leq (1+\tau_B^2) a_{\lambda}(\mathbb{S} \bdd{v}, \mathbb{S} \bdd{v}),
	\end{align}
	where $\mathbb{T}_B : \bdd{X}_D \to \bdd{X}_B$ is the operator defined in \cref{sec:constructing asm} and $\tau_B$ is defined in \cref{eq:tau condition}.
\end{lemma}
\begin{proof}
	Let $\bdd{v} \in \bdd{X}_D$. Since $\mathbb{T}_B \bdd{v} = \bdd{v}$ on $\partial K$ for all $K \in \mathcal{T}$, we have
	\begin{align}
		\label{eq:proof:tb inf bound}
		\| \bdd{\varepsilon}(\mathbb{T}_B \bdd{v}) \|^2 + \lambda\mu^{-1} \| \Pi_I \dive \mathbb{T}_B \bdd{v} \|^2 \leq \tau_B^2 \inf_{ \substack{\bdd{w} \in \bdd{X}_D \\ \bdd{w}|_{\partial K} = \bdd{v}|_{\partial K} \ \forall K \in \mathcal{T}}}  \| \bdd{\varepsilon}(\bdd{w})\|^2
	\end{align}
	and
	\begin{align}
		\label{eq:piiperp def}
		\Pi_I^{\perp} \dive \mathbb{T}_B \bdd{v} =  	\Pi_I^{\perp} \dive \bdd{v} + \Pi_I^{\perp} \underbrace{\dive (\mathbb{T}_B \bdd{v} - \bdd{v})}_{\in \dive \bdd{X}_I} = \Pi_I^{\perp} \dive \bdd{v}, \quad \text{where } \Pi_I^{\perp} := I - \Pi_I.
	\end{align}
	The same argument with $\mathbb{T}_B = \mathbb{S}$ shows that $\Pi_I^{\perp} \dive \mathbb{S} \bdd{v} = \Pi_I^{\perp} \dive \bdd{v}$, while \cref{eq:stokes extension 2} gives $\Pi_I \dive \mathbb{S} \bdd{u} = 0$. Consequently,
	\begin{align}
		\label{eq:div stokes ext id}
		\dive \mathbb{S} \bdd{v} = \Pi_I^{\perp} \dive \bdd{v},	
	\end{align}
	and so
	\begin{align*}
		a_{\lambda}(\mathbb{T}_B \bdd{v}, \mathbb{T}_B \bdd{v}) &= 2\mu \| \bdd{\varepsilon}(\mathbb{T}_B \bdd{v}) \|^2 + \lambda \| \Pi_I \mathbb{T}_B \dive \bdd{v} \|^2 + \lambda \| \Pi_I^{\perp} \dive \bdd{v} \|^2 \\
		&= 2\mu \| \bdd{\varepsilon}(\mathbb{T}_B \bdd{v}) \|^2 + \lambda \| \Pi_I \mathbb{T}_B \dive \bdd{v} \|^2 + \lambda \| \dive \mathbb{S} \bdd{v} \|^2. 
	\end{align*}
	Choosing $\bdd{w} = \mathbb{S} \bdd{v}$ on the RHS of \cref{eq:proof:tb inf bound} give the rightmost inequality in \cref{eq:t and s equivalences}. Choosing $\bdd{w} = \mathbb{T}_B \bdd{v}$ in \cref{eq:stokes seminorm continuity element}, summing over the elements, and using the above relation gives the leftmost inequality in \cref{eq:t and s equivalences}.
\end{proof}

Property \cref{eq:t and s equivalences} of the operator $\mathbb{S}$ will play a central role in the stability properties of the subspace decomposition \cref{eq:tilde subspace decomp}. The next result relates $\mathbb{S}$ to the space $\tilde{\bdd{X}}_B$ in a similar way that $\mathbb{H}$ was related to $\bdd{X}_B$ in \cref{eq:static cond choice}.
\begin{lemma}
	\label{rem:tilde tau}
	There holds $\tilde{\bdd{X}}_B = \mathbb{S} \bdd{X}_D$. Additionally, $\bdd{X}_D = \bdd{X}_I \oplus \tilde{\bdd{X}}_B$ and the operator $\mathbb{S}$ satisfies \cref{eq:tau condition} with $\tau_B$ independent of $\mu$, $\lambda$, $h$, and $p$.
\end{lemma}
\begin{proof}
	The identity $\tilde{\bdd{X}}_B = \mathbb{S} \bdd{X}_D$ follows from \cite[Lemma 6.2]{AinCP22SCIP}. Moreover, for $\bdd{v} \in \bdd{X}_D$, $\bdd{v} = \mathbb{S} \bdd{v} + \bdd{v}_I$, where $\bdd{v}_I := \bdd{v} - \mathbb{S} \bdd{v}$ satisfies $\bdd{v} \in \bdd{X}_I$ by \cref{eq:stokes extension 3}. As shown in the proof of \cref{lem:eq:t and s equivalences}, $\Pi_I \dive \mathbb{S} \bdd{u} = 0$, and \cref{eq:tau condition} follows by squaring \cref{eq:stokes seminorm continuity element} and summing over the elements.
\end{proof}
\Cref{rem:tilde tau} shows that the choice of boundary space $\tilde{\bdd{X}}_B$ for the SCIP method is simply the image of $\bdd{X}_D$ under the Stokes extension operator $\mathbb{S}$. By way of contrast, static condensation corresponds to choosing the boundary space $\bdd{X}_B$ to be the image of $\bdd{X}_D$ under the extension operator $\mathbb{H}$ as in \cref{eq:static cond choice}. In the case of static condensation, the boundary space was decomposed further as in \cref{eq:tilde subspace decomp}. By the same token, we further decompose the SCIP space $\tilde{\bdd{X}}_B$ into subspaces
\begin{align}
	\label{eq:tilde subspace decomp 2}
	\tilde{\bdd{X}}_B = \tilde{\bdd{X}}_C + \sum_{ \bdd{a} \in \mathcal{V}} \tilde{\bdd{X}}_{\bdd{a}},
\end{align}
where
\begin{align*}
	\tilde{\bdd{X}}_C := \{ \bdd{v} \in \tilde{\bdd{X}}_B : \bdd{v} |_{\gamma} \in \bm{\mathcal{P}}_4(\gamma) \ \forall \gamma \in \mathcal{E} \} \quad \text{and} \quad \tilde{\bdd{X}}_{\bdd{a}} :=  \{ \bdd{v} \in \tilde{\bdd{X}}_B : \supp \bdd{v} \subseteq \mathcal{T}_{\bdd{a}} \}.
\end{align*} 
The next section is concerned with proving that the decomposition \cref{eq:tilde subspace decomp 2} is stable in the following sense:
\begin{theorem}
	\label{thm:h1 stable decomp 2}
	For every $\bdd{u} \in \tilde{\bdd{X}}_B$, there exist $\bdd{u}_C \in \tilde{\bdd{X}}_C$ and $\bdd{u}_{\bdd{a}} \in \tilde{\bdd{X}}_{\bdd{a}}$, $\bdd{a} \in \mathcal{V}$, such that
	\begin{align}
		\label{eq:h1 stable decomp decomp 2}
		\bdd{u} =  \bdd{u}_C + \sum_{ \bdd{a} \in \mathcal{V}} \bdd{u}_{\bdd{a}} \quad \text{and} \quad \|\bdd{\varepsilon}(\bdd{u}_C) \|^2  + \sum_{ \bdd{a} \in \mathcal{V}} \| \bdd{\varepsilon}(\bdd{u}_{\bdd{a}}) \|^2 
		\leq C \|\bdd{\varepsilon}( \bdd{u})\|^2.
	\end{align}
\end{theorem}

\section{Stable Decomposition with Respect to Strain}
\label{sec:h1 stable decomp 2}

Each edge $\gamma \in \mathcal{E}$ is assigned an arbitrary but fixed orientation and let $\unitvec{t}_{\gamma}$ and $\unitvec{n}_{\gamma}$ denote the unit tangent and normal vectors on $\gamma$; when $\gamma$ lies on the domain boundary $\Gamma$, $\unitvec{n}_{\gamma}$ coincides with the outward unit normal to $\Gamma$. Given an element $K \in \mathcal{T}$, let $\mathcal{V}_K$, $\mathcal{E}_K$, and $\mathcal{T}_K$ denote the vertices of, edges of, and elements abutting $K$.  Given an edge $\gamma \in \mathcal{E}$, let $\mathcal{T}_{\gamma}$  be the set of elements sharing the edge $\gamma$. 

\begin{lemma}
	\label{lem:coarse comp 2}
	There exists a linear operator $\mathcal{I}_C : \bdd{X}_D \to \tilde{\bdd{X}}_C$ 
	satisfying
	\begin{align}
		\label{eq:coarse continuity 2}
		\|\bdd{\varepsilon}(\mathcal{I}_C \bdd{u})\|_{K} + h_K^{-1} \| \bdd{u} - \mathcal{I}_C \bdd{u}\|_{K} + | \bdd{u} - \mathcal{I}_C \bdd{u}|_{1,K} \leq C \|\bdd{\varepsilon}(\bdd{u})\|_{\mathcal{T}_K} \quad \forall K \in \mathcal{T}, \forall \bdd{u} \in \bdd{X}_D.
	\end{align}
\end{lemma}
\begin{proof}
	Let $\bdd{u} \in \bdd{X}_D$ be given. Let $\mathcal{I}_{SZ} : \bdd{X} \to \{ \bdd{u} \in \bdd{X} : \bdd{u}|_{K} \in \bm{\mathcal{P}}_{1}(K) \ \forall K \in \mathcal{T} \}$ denote the piecewise-linear Scott-Zhang interpolant \cite{Scott90} satisfying \cite[eq. (4.3)]{Scott90}
	\begin{align*}
		\|\bdd{u} - \mathcal{I}_{SZ} \bdd{u}\|_{K} + h_K |\bdd{u} -  \mathcal{I}_{SZ} \bdd{u}|_{1,K} \leq C h_K |\bdd{u}|_{1,\mathcal{T}_K} \qquad \forall K \in \mathcal{T}.
	\end{align*}
	Since $\mathcal{I}_{SZ} \bdd{r} = \bdd{r}$ for any rigid body motion $\bdd{r} \in \bdd{RM}$, \cref{eq:h1 korn inequality patch} gives
	\begin{align*}
		\|\bdd{u} - \mathcal{I}_{SZ} \bdd{u}\|_{K} + h_K |\bdd{u} -  \mathcal{I}_{SZ} \bdd{u}|_{1,K} \leq C h_K \inf_{\bdd{r} \in \bdd{RM}}|\bdd{u} - \bdd{r}|_{1,\mathcal{T}_K} \leq C h_K \|\bdd{\varepsilon}(\bdd{u})\|_{\mathcal{T}_K}.
	\end{align*}
	The bound $\| \bdd{\varepsilon}(\mathcal{I}_{SZ} \bdd{u})\|_{K} \leq C \|\bdd{\varepsilon}(\bdd{u})\|_{K}$ then follows from the triangle inequality. The linear operator $\mathcal{I}_C : \bdd{X}_D \to \tilde{\bdd{X}}_C$ defined by the rule $\mathcal{I}_C \bdd{u} := \mathbb{S} \mathcal{I}_{SZ} \bdd{u}$ satisfies \cref{eq:coarse continuity 2} thanks to \cref{eq:stokes approximation} and the triangle inequality. Moreover, $\bdd{u} \in \bdd{X}_D$, $\mathcal{I}_{SZ} \bdd{u} \in \bdd{X}_D$ by \cite[Theorem 2.1]{Scott90}, and so $\mathcal{I}_C \bdd{u} \in \tilde{\bdd{X}}_C$.
\end{proof}

Given a vertex $\bdd{a} \in \mathcal{V}$, let $\mathcal{E}_{\bdd{a}}$ denote the set of edges sharing $\bdd{a}$ as a vertex and let $\xi_{\bdd{a}}$ denote the barycentric coordinate corresponding to $\bdd{a}$ on $K \in \mathcal{T}_{\bdd{a}}$. Additionally, given an edge $\gamma$, let $H^{1/2}_{00}(\gamma)$ and $\bdd{H}^{1/2}_{00}(\gamma)$ denote the usual interpolation spaces defined in \cite{Lions12}.

\begin{lemma}
	\label{lem:c0 vertex comp aux}
	For each $\bdd{a} \in \mathcal{V}$, there exists a linear operator $\mathcal{I}_{\bdd{a}} : \bdd{X}_D \to \tilde{\bdd{X}}_{\bdd{a}}$ such that for all $\bdd{u} \in \bdd{X}_D$:
	\begin{alignat}{2}
		\label{eq:c0 vertex comp interp}
		\mathcal{I}_{\bdd{a}} \bdd{u}(\bdd{a}) &= \bdd{u}(\bdd{a}), \quad & &\\
		\label{eq:c0 vertex comp aux h1 cont}
		h_K^{-1} \|\mathcal{I}_{\bdd{a}} \bdd{u}\|_{K} + |\mathcal{I}_{\bdd{a}} \bdd{u}|_{1,K} &\leq C \{ h_K^{-1} \|\bdd{u}\|_{K} + |\bdd{u}|_{1,K} \} \ & & \forall K \in \mathcal{T}_{\bdd{a}}, \\
		\label{eq:c0 vertex comp aux h1200 cont}
		|\gamma|^{-1/2} \|\xi_{\bdd{a}} \bdd{u} - \mathcal{I}_{\bdd{a}} \bdd{u}\|_{\gamma} + |\xi_{\bdd{a}} \bdd{u} - \mathcal{I}_{\bdd{a}} \bdd{u}|_{\bdd{H}^{1/2}_{00}(\gamma)} &\leq C \{ |\gamma|^{-1/2} \|\bdd{u}\|_{\gamma} + |\bdd{u}|_{1/2, \gamma} \} \  & &\forall \gamma \in \mathcal{E}_{\bdd{a}}.
	\end{alignat}
\end{lemma} 
\begin{proof}
	Let $\bdd{u} \in \bdd{X}_D$, $\bdd{a} \in \mathcal{V}$. We define the function $\bdd{v}$ by the rule $\bdd{v}|_{K} = \mathcal{I}_K^{\bdd{a}} \bdd{u}$ for $K \in \mathcal{T}_{\bdd{a}}$ and $\bdd{v}|_{K} \equiv 0$ for $K \in \mathcal{T} \setminus \mathcal{T}_{\bdd{a}}$, where $\mathcal{I}_K^{\bdd{a}}$ is the linear operator in \cref{thm:c0 vertex interp element}. By \cref{thm:c0 vertex interp element} (1) and (4), $\bdd{v} \in \bdd{X}_D$. 	The operator $\mathcal{I}_{\bdd{a}} : \bdd{X}_D \to \tilde{\bdd{X}}_{\bdd{a}}$ is then defined by the rule $\mathcal{I}_{\bdd{a}} \bdd{u} := \mathbb{S} \bdd{v}$. Then, \cref{eq:c0 vertex comp interp} follows from \cref{eq:stokes extension 3} and \cref{thm:c0 vertex interp element} (1). Moreover, \cref{eq:c0 vertex comp aux h1 cont} follows from \cref{eq:stokes continuity element}, the trace theorem, and \cref{thm:c0 vertex interp element} (3), while \cref{eq:c0 vertex comp aux h1200 cont} follows from \cref{thm:c0 vertex interp element} (4).
\end{proof}

\begin{lemma}
	\label{lem:c0 edge interp}
	Let $\bdd{X}_{\mathcal{E}} := \{ \bdd{u} \in \bdd{X}_D : \bdd{u}(\bdd{a}) = \bdd{0} \ \forall \bdd{a} \in \mathcal{V} \}$. For each $\gamma \in \mathcal{E}$, there exists a linear operator $\mathcal{I}_{\gamma} : \bdd{X}_{\mathcal{E}} \to \tilde{\bdd{X}}_B$ such that for all $\bdd{u} \in \bdd{X}_{\mathcal{E}}$:
	\begin{subequations}
		\label{eq:c0 edge interp prop}
		\begin{alignat}{2}
			\mathcal{I}_{\gamma} \bdd{u}(\bdd{a}) &= \bdd{0} \qquad & &\forall \bdd{a} \in \mathcal{V}, \\
			\mathcal{I}_{\gamma} \bdd{u} &= \bdd{u} \qquad & & \text{on } \gamma, \\
			\supp 	\mathcal{I}_{\gamma} \bdd{u} &\subseteq \mathcal{T}_{\gamma}, \qquad & &
		\end{alignat}
	\end{subequations}
	and
	\begin{align}
		\label{eq:c0 edge interp continuity}
		h_K^{-1} \|	\mathcal{I}_{\gamma} \bdd{u}\|_{K} + |	\mathcal{I}_{\gamma} \bdd{u}|_{1,K} &\leq C  |\bdd{u}|_{\bdd{H}^{1/2}_{00}(\gamma)}
		\qquad \forall K \in \mathcal{T}_{\gamma}.
	\end{align}	
\end{lemma}
\begin{proof}
	Let $\gamma \in \mathcal{E}$ and $\bdd{u} \in  \bdd{X}_{\mathcal{E}}$.
	Let $\bdd{v} \in \bdd{X}_D$ be any function supported on $\mathcal{T}_{\gamma}$ with $\bdd{v}|_{\gamma} = \bdd{u}|_{\gamma}$. Thanks to \cref{rem:stokes extension boundary}, $\mathbb{S} \bdd{v} \in \tilde{\bdd{X}}_B$ and hence the operator $\mathcal{I}_{\gamma} : \bdd{X}_D \to \tilde{\bdd{X}}_B$ given by the rule $\mathcal{I}_{\gamma} \bdd{u} := \mathbb{S} \bdd{v}$ is well-defined. Moreover, using \cref{eq:stokes continuity element} and Poincar\'{e}'s inequality, we obtain
	\begin{align*}
		h_K^{-1} \|\mathcal{I}_{\gamma} \bdd{u}\|_{K} + |\mathcal{I}_{\gamma} \bdd{u}|_{1,K} &\leq C \left\{ h_K^{-1/2} \|\mathcal{I}_{\gamma} \bdd{u}\|_{\partial K} + |\mathcal{I}_{\gamma} \bdd{u}|_{1/2,\partial K} \right\}  \\
		&\leq C \left\{ h_K^{-1/2} \| \bdd{u} \|_{\gamma} + | \bdd{u} |_{\bdd{H}^{1/2}_{00}(\gamma)} \right\} \\
		&\leq C | \bdd{u} |_{\bdd{H}^{1/2}_{00}(\gamma)}
	\end{align*}
	for all $K \in \mathcal{T}_{\gamma}$.
\end{proof}

\begin{proof}[Proof of \cref{thm:h1 stable decomp 2}]
	\Cref{eq:h1 stable decomp decomp 2} is a now a special case of \cref{lem:general stability}.
\end{proof}

\section{Stable Decomposition of Divergence Free Functions}
\label{sec:divfree stable decomp}

The subspace of $\tilde{\bdd{X}}_B$ consisting of divergence free functions 
\begin{align}
	\label{eq:tildenb definition}
	\tilde{\bdd{N}}_B := \{ \tilde{\bdd{v}} \in \tilde{\bdd{X}}_B : \dive  \tilde{\bdd{v}} \equiv 0 \}
\end{align}
will play a key role in the proof of \cref{thm:main result}. It is useful to first characterize the space $\tilde{\bdd{N}}_B$ as the curl of a suitable $H^2$-conforming finite element space. Let $\{ \Gamma_{D,j} \}_{j=1}^{J}$ denote the connected components of $\Gamma_D$ and define
\begin{align*}
	H^2_D(\Omega) &:= \{ \psi \in H^2(\Omega) : \psi|_{\Gamma_{D,1}} = 0, \ \text{$\psi|_{\Gamma_{D,j}}$ is constant, $2 \leq j \leq J$, $\partial_n \psi|_{\Gamma_D} = 0$} \},
\end{align*}
so that $\vcurl H^2_D(\Omega) \subset \bdd{H}^1_D(\Omega)$, where $\vcurl \phi = (\partial_y \phi, -\partial_x \phi)^T$. The corresponding discrete space
\begin{align}
	\label{eq:sigmad def}
	\Sigma_D = \Sigma \cap H^2_D(\Omega) \qquad \text{where} \quad \Sigma := \{ \psi \in C^1(\bar{\Omega}) : \psi|_{K} \in \mathcal{P}_{p+1}(K) \ \forall K \in \mathcal{T} \},
\end{align}
satisfies $\vcurl \Sigma_D = \{ \bdd{v} \in \bdd{X}_D : \dive \bdd{v} \equiv 0 \}$ thanks to \cref{thm:full complex}. Let
\begin{align}
	\label{eq:interior sigma spaces}
	\Sigma_I := \bigoplus_{K \in \mathcal{T}} \Sigma_I(K), \qquad \text{where} \quad 	\Sigma_I(K) := \{ \psi \in \Sigma : \supp \psi \subseteq K  \},
\end{align}
denote the subspace consisting of interior functions, while the boundary space is given by
\begin{align}
	\label{eq:boundary sigma spaces}
	\tilde{\Sigma}_B := \{ \psi \in \Sigma_D : a_{\lambda}(\vcurl \psi, \vcurl \rho) = a_0(\vcurl \psi, \vcurl \rho) = 0 \ \forall \rho \in \Sigma_I \}.
\end{align}
\Cref{thm:full complex} then shows that the space $\Sigma_D$ admits a decomposition similar to \cref{eq:direct sum identity}: $\Sigma_D = \Sigma_I \oplus \tilde{\Sigma}_B$, and that space $\tilde{\bdd{N}}_B$ has the following characterization:
\begin{lemma}
	\label{lem:tildenb characterization}
	There holds $\tilde{\bdd{N}}_B = \vcurl \tilde{\Sigma}_B$.
\end{lemma} 
We also need the following subspaces of $\tilde{\Sigma}_B$:
\begin{align*}
	\tilde{\Sigma}_C &:= \{ \psi \in \tilde{\Sigma}_{B} : \psi|_{\gamma} \in \mathcal{P}_5(\gamma) \ \forall \gamma \in \mathcal{E} \}  \  \text{and} \ 
	\tilde{\Sigma}_{\bdd{a}} := \{ \psi \in \tilde{\Sigma}_B : \supp \psi \subseteq \mathcal{T}_{\bdd{a}} \}, \ \bdd{a} \in \mathcal{V},
\end{align*}
which, as an immediate consequence of \cref{lem:tildenb characterization}, satisfy the inclusions
\begin{align}
	\label{eq:tildesigma curls inclusions}
	\vcurl \tilde{\Sigma}_C \subseteq \tilde{\bdd{X}}_C \cap \tilde{\bdd{N}}_B \quad \text{and} \quad \vcurl \tilde{\Sigma}_{\bdd{a}} \subseteq \tilde{\bdd{X}}_{\bdd{a}} \cap \tilde{\bdd{N}}_B \qquad \forall \bdd{a} \in \mathcal{V}.
\end{align}
These subspaces give an overlapping decomposition of the boundary space
\begin{align*}
	\tilde{\Sigma}_B = \tilde{\Sigma}_C + \sum_{ \bdd{a} \in \mathcal{V}} \tilde{\Sigma}_{\bdd{a}}
\end{align*}
which is stable with respect to strain:
\begin{theorem}
	\label{thm:h2 stable decomp}
	For every $\psi \in \tilde{\Sigma}_B$, there exists $\psi_C \in \tilde{\Sigma}_C$ and $\psi_{\bdd{a}} \in \tilde{\Sigma}_{\bdd{a}}$, $\bdd{a} \in \mathcal{V}$, such that
	\begin{align}
		\label{eq:h2 stable decomp decomp}		
		\psi &= \psi_C + \sum_{ \bdd{a} \in \mathcal{V}} \psi_{\bdd{a}} \ \ \text{and} \ \ \|\bdd{\varepsilon}(\vcurl \psi_C) \|^2  + \sum_{ \bdd{a} \in \mathcal{V}} \| \bdd{\varepsilon}(\vcurl \psi_{\bdd{a}}) \|^2 
		\leq C \|\bdd{\varepsilon}( \vcurl \psi)\|^2.
	\end{align}
\end{theorem}
The proof of \cref{thm:h2 stable decomp} is given at the end of this section. However, an immediate consequence of \cref{lem:tildenb characterization}, \cref{thm:h2 stable decomp}, and \cref{eq:tildesigma curls inclusions} is that the decomposition \cref{eq:tilde subspace decomp} respects divergence free functions in the following sense:
\begin{corollary}
	\label{thm:divfree stable decomp}
	For every $\bdd{v} \in \tilde{\bdd{N}}_B$, there exist $\bdd{v}_C \in \tilde{\bdd{X}}_C \cap \tilde{\bdd{N}}_B$ and $\bdd{v}_{\bdd{a}} \in \tilde{\bdd{X}}_{\bdd{a}} \cap \tilde{\bdd{N}}_B$, $\bdd{a} \in \mathcal{V}$, such that
	\begin{align}
		\label{eq:divfree stable decomp decomp}
		\bdd{v} =  \bdd{v}_C + \sum_{ \bdd{a} \in \mathcal{V}} \bdd{v}_{\bdd{a}} \quad \text{and} \quad 	a_{\lambda}(\bdd{v}_C, \bdd{v}_C) + \sum_{ \bdd{a} \in \mathcal{V}} a_{\lambda}(\bdd{v}_{\bdd{a}}, \bdd{v}_{\bdd{a}}) 
		\leq C a_{\lambda}(\bdd{v}, \bdd{v}).
	\end{align}
\end{corollary}

\subsection{An $H^2$-Stable Extension Operator}

Let  $\mathbb{B} : \Sigma \to \Sigma$ denote the extension operator defined elementwise as follows: for $\psi \in \Sigma$ and $K \in \mathcal{T}$, 
\begin{subequations}
	\label{eq:biharmonic extension}
	\begin{alignat}{2}
		\label{eq:biharmonic extension 1}
		a_K(\vcurl \mathbb{B} \psi, \vcurl \rho) &= 0 \qquad & & \forall \rho \in \Sigma_I(K), \\
		\label{eq:biharmonic extension 2}
		(\psi - \mathbb{B} \psi)|_{\partial K} &= 0, \qquad & & \\
		\label{eq:biharmonic extension 3}
		\partial_n (\psi - \mathbb{B} \psi)|_{\partial K} &= 0. \qquad & & 
	\end{alignat}
\end{subequations}
\begin{remark}
	\label{rem:biharmonic traces}
	A discussion similar to \cref{rem:stokes extension boundary} shows that $\mathbb{B} \psi|_{K}$ depends only on the traces $\psi|_{\partial K}$ and $\partial_n \psi|_{\partial K}$ for $K \in \mathcal{T}$. The relation $\tilde{\Sigma}_B = \mathbb{B} \Sigma_D$ is also readily verified.
\end{remark}
Moreover, we have the following properties:
\begin{lemma}
	For all $K \in \mathcal{T}$, there holds
	\begin{align}
		\label{eq:biharmonic continuity element}
		h_K^{-2} \|\mathbb{B} \psi\|_{K} &+ h_K^{-1} |\mathbb{B} \psi|_{1, K} + |\mathbb{B} \psi|_{2, K} \qquad & & \\ 
		&\leq C \left\{ h_K^{-3/2} \|\psi\|_{\partial K} + h_K^{-1/2} \|\nabla \psi\|_{\partial K} + |\nabla \psi|_{1/2,\partial K} \right\} \qquad & &\forall \psi \in \Sigma \notag
	\end{align}
	and
	\begin{align}
		\label{eq:biharmonic approximation}
		h_K^{-2} \|\psi - \mathbb{B} \psi \|_{K} + h_K^{-1} |\psi - \mathbb{B} \psi |_{1,K} + |\psi - \mathbb{B} \psi |_{2,K} &\leq C \| \bdd{\varepsilon}(\vcurl \psi)\|_K \qquad & &\forall \psi \in \Sigma.
	\end{align}
\end{lemma}
\begin{proof}
	Let $K \in \mathcal{T}$. For $\psi \in \Sigma$, there exists $\chi \in \mathcal{P}_{p+1}(K)$ satisfying $D^{\alpha} (\phi - \chi)|_{\partial K} = 0$ for $|\alpha| \leq 1$ and
	\begin{align}
		\label{eq:proof:h2ext thm}
		|\chi|_{2,K} \leq C \left\{ h_K^{-3/2} \|\psi\|_{\partial K} + h_K^{-1/2} \|\nabla \psi\|_{\partial K} + |\nabla \psi|_{1/2,\partial K}\right\}
	\end{align}
	thanks to \cite[Corollary A.1]{AinCP19Precon}. Choosing $\rho = \mathbb{B}\psi - \chi$ in \cref{eq:biharmonic extension 1} and applying the Cauchy-Schwarz inequality gives 
	\begin{align*}
		\| \bdd{\varepsilon}(\vcurl \mathbb{B} \psi) \|_{K}^2 = (2\mu)^{-1} a_K(\vcurl \mathbb{B} \psi, \vcurl \mathbb{B} \chi) \leq \|  \bdd{\varepsilon}(\vcurl \mathbb{B} \psi) \|_{K} \|  \bdd{\varepsilon}(\vcurl \chi) \|_{K}.
	\end{align*}
	Consequently, there holds
	\begin{align*}
		\| \bdd{\varepsilon}(\vcurl \mathbb{B} \psi) \|_{K} \leq  \|  \bdd{\varepsilon}(\vcurl \chi) \|_{K} \leq \sqrt{2} |\chi|_{2,K}.
	\end{align*}
	Using the relation
	\begin{align*}
		h_K^{-1} \| \rho \|_{K} \leq C \{ h_K^{-1/2} \|\rho\|_{\partial K} + \|\vcurl \rho\|_K \} \qquad \forall \rho \in H^1(K),
	\end{align*}
	which follows from the compactness of the embedding $H^1(K) \hookrightarrow L^2(K)$ and a standard scaling argument, together with \cref{eq:proof:korn inequality boundary} applied to $\vcurl \psi$, we obtain
	\begin{multline*}
		h_K^{-2} \|\psi\|_{K} + h_K^{-1} \|\vcurl \psi\|_{K} \\
			\leq C \{ h_K^{-3/2} \|\psi\|_{\partial K} + h_K^{-1/2} \|\vcurl \psi\|_{\partial K} + \|\bdd{\varepsilon}(\vcurl \psi)\|_{K} \}. 
	\end{multline*}
	Collecting results gives
	\begin{align*}
		\sum_{j=0}^{2} h_K^{j-2}  |\mathbb{B} \psi|_{j, K} &\leq C \left\{ h_K^{-3/2} \|\mathbb{B} \psi\|_{\partial K} + h_K^{-1/2} \|\vcurl \mathbb{B} \psi\|_{\partial K} + \| \bdd{\varepsilon}(\vcurl \mathbb{B} \psi) \|_{K} \right\} \\
		&\leq C \left\{ h_K^{-3/2} \|\psi\|_{\partial K} + h_K^{-1/2} \|\vcurl \psi\|_{\partial K} + |\chi|_{2,K} \right\}.
	\end{align*}
	Since $|\vcurl \psi| = |\nabla \psi|$, \cref{eq:biharmonic continuity element} now follows from \cref{eq:proof:h2ext thm}.
	
	Using \cref{eq:biharmonic continuity element} and the trace theorem, we obtain
	\begin{align*}
		\sum_{j=0}^{2} h_K^{j-2} \|\mathbb{B} \psi \|_{j,K} &\leq C  \left\{ h_K^{-3/2} \|\psi\|_{\partial K} + h_K^{-1/2} \|\nabla \psi\|_{\partial K} + |\nabla \psi|_{1/2, \partial K} \right\} \\
		&\leq C \sum_{j=0}^{2} h_K^{j-2} \|\psi \|_{j,K},
	\end{align*}
	and hence
	\begin{align*}
		\sum_{j=0}^{2} h_K^{j-2} \|\psi - \mathbb{B} \psi \|_{j,K} \leq C \sum_{j=0}^{2} h_K^{j-2} \|\psi \|_{j,K}.
	\end{align*}
	Since $\mathbb{B} u = u$ for any $u \in \mathcal{U}$ defined in \cref{lem:h2 korn inequality}, there holds
	\begin{align*}
		\sum_{j=0}^{2} h_K^{j-2} \|\psi - \mathbb{B} \psi \|_{j,K} \leq C  \inf_{u \in \mathcal{U}}\sum_{j=0}^{2} h_K^{j-2} \|\psi - u \|_{j,K} \leq C \| \bdd{\varepsilon}(\vcurl \psi)\|_K,
	\end{align*}
	where we used \cref{eq:h2 korn inequality} in the last step.
\end{proof}

\subsection{Nodal Interpolation Operators}

\begin{lemma}
	\label{lem:h2 coarse comp}
	There exists a linear nodal interpolation operator $\mathcal{J}_C : \Sigma_D \to \tilde{\Sigma}_C$ with the following properties:
	\begin{alignat}{2}
		\label{eq:h2 coarse interp vertices}
		\mathcal{J}_C \psi(\bdd{a}) &= \psi(\bdd{a}) \qquad & &\forall \bdd{a} \in \mathcal{V},  
	\end{alignat}
	and, for all $K \in \mathcal{T}$,
	\begin{multline}
		\label{eq:h2 coarse continuity}
		h_K^{-2} \|\psi - \mathcal{J}_C \psi\|_{K} + h_K^{-1} |\psi - \mathcal{J}_C \psi|_{1,K} + |\psi - \mathcal{J}_C \psi|_{2,K} + \| \bdd{\varepsilon}(\vcurl \mathcal{J}_C \psi)\|_K \\
		\leq C \|\bdd{\varepsilon}(\vcurl \psi)\|_{\mathcal{T}_K}.
	\end{multline}
\end{lemma}
\begin{proof}
	Let $\psi \in \Sigma_D$ be given. Let $\mathcal{J}_{GS} : H^2(\Omega) \to \Sigma_5 := \{ \rho \in \Sigma : \rho|_{K} \in \mathcal{P}_{5}(K) \ \forall K \in \mathcal{T} \}$ denote the projection operator defined by Girault \& Scott  \cite{Girault02}. By \cite[eq. (7.9)]{Girault02}, there holds
	\begin{align*}
		\sum_{j=0}^{2} h_K^{j-2} | \psi - \mathcal{J}_{GS} \psi |_{j,K} + \| \bdd{\varepsilon}(\vcurl \mathcal{J}_{GS} \psi)\|_K  \leq C |\psi|_{2,\mathcal{T}_K}.
	\end{align*}	
	Since $\mathcal{J}_{GS} u = u$ for any $u \in \mathcal{U}$, where $\mathcal{U}$ is defined in \cref{lem:h2 korn inequality}, we may replace $\psi$ with $\psi - u$, take the infimum over all such $u$, and apply \cref{eq:h2 korn inequality patch} to obtain
	\begin{align}
		\label{eq:proof:h2 scottzhang cont}
		\sum_{j=0}^{2} h_K^{j-2 } | \psi - \mathcal{J}_{GS} \psi |_{j,K} + \| \bdd{\varepsilon}(\vcurl \mathcal{J}_{GS} \psi)\|_K  \leq C \|\bdd{\varepsilon}(\vcurl \psi)\|_{\mathcal{T}_K}.
	\end{align}
	
	Given $\rho \in \Sigma$, we construct $\omega \in \Sigma_5$ by assigning the degrees of freedom \cite{MorganScott75} as follows:
	\begin{subequations}
		\begin{alignat}{2}
			\label{eq:proof:h2 vertex dofs}
			\omega(\bdd{a}) &= \rho(\bdd{a}) \qquad & &\forall \bdd{a} \in \mathcal{V},\\
			D^{\alpha} \omega(\bdd{a}) &= 0 \qquad & &\forall \bdd{a} \in \mathcal{V}, \ \forall 1 \leq |\alpha| \leq 2, \\
			\int_{\gamma} \partial_n \omega \ dS &= 0 \qquad & &\forall \gamma \in \mathcal{E}.
		\end{alignat}
	\end{subequations}
	It is easily seen that if $\rho \in \Sigma_D$, then $\omega \in \Sigma_5 \cap \Sigma_D$. Additionally, a standard norm equivalence argument on the finite dimensional space $\mathcal{P}_5(K)$, together with the embedding $H^2(K) \hookrightarrow C^0(\bar{K})$, shows that
	\begin{align*}
		\sum_{j=0}^{2} h_K^{j-2} |\omega|_{j,K} \leq C \max_{\bdd{a} \in \mathcal{V}_K} |\omega(\bdd{a})| \leq C \sum_{j=0}^{2} h_K^{j-2} |\rho|_{j,K}.
	\end{align*}
	Finally, we may define $\mathcal{I}_5 : \Sigma \to \Sigma_5$ by the rule $\mathcal{I}_5 \rho := \omega$. 
	
	Let $\rho_C := \mathcal{J}_{GS} \psi + \mathcal{I}_5 (\psi - \mathcal{J}_{GS} \psi)$. Thanks to \cref{eq:proof:h2 scottzhang cont,eq:proof:h2 vertex dofs}, there holds
	\begin{align*}
		\sum_{j=0}^{2} h_K^{j-2} |\psi - \rho_C|_{j,K} + \| 	\bdd{\varepsilon}(\vcurl \rho_C)\|_K \leq C \|\bdd{\varepsilon}(\vcurl \psi)\|_{\mathcal{T}_K}.
	\end{align*}
	The operator $\mathcal{J}_C : \Sigma_D \to \tilde{\Sigma}_C$ defined by the rule $\mathcal{J}_C \psi := \mathbb{B} \rho_C$ satisfies \cref{eq:h2 coarse interp vertices} thanks \cref{eq:biharmonic extension 2,eq:biharmonic extension 3} and \cref{eq:h2 coarse continuity} thanks to \cref{eq:biharmonic continuity element}, the triangle inequality, and the trace theorem. Moreover, since $\psi \in \Sigma_D$, $\mathcal{J}_{GS} \psi \in \Sigma_D$ by \cite[Theorem 5.1]{Girault02} and so $\mathcal{J}_C \psi \in \tilde{\Sigma}_B$. The inclusion $\mathcal{J}_C \psi \in \tilde{\Sigma}_C$  follows directly from construction.	
\end{proof}

\begin{lemma}
	\label{lem:c1 vertex comp aux}
	Let $\bdd{a} \in \mathcal{V}$, then there exists a linear nodal interpolation operator $\mathcal{J}_{\bdd{a}} : \Sigma_D \to \tilde{\Sigma}_{\bdd{a}}$ such that for all $\psi \in \Sigma_D$:
	\begin{enumerate}
		\item[(i)] 	$\mathcal{J}_{\bdd{a}} \psi(\bdd{a}) = 0$, $\nabla \mathcal{J}_{\bdd{a}} \psi(\bdd{a}) = \nabla \psi(\bdd{a})$, and $D^2 \mathcal{J}_{\bdd{a}} \psi|_{K}(\bdd{a}) = D^2 \psi|_{K}(\bdd{a})$  $\forall K \in \mathcal{T}$.
		
		\item[(ii)] For all $K \in \mathcal{T}_{\bdd{a}}$,
		\begin{multline}
			\label{eq:c1 vertex comp aux h2 cont}
			h_K^{-2} \| \mathcal{J}_{\bdd{a}} \psi \|_{K} + h_K^{-1} |\mathcal{J}_{\bdd{a}} \psi|_{1, K} + |\mathcal{J}_{\bdd{a}} \psi|_{2, K} \\
			\leq C \{ 	h_K^{-2} \| \psi \|_{K} + h_K^{-1} |\psi|_{1, K} + |\psi|_{2, K} \}.
		\end{multline}
		
		\item[(iv)] For all $\gamma \in \mathcal{E}_{\bdd{a}}$,
		\begin{multline}
			\label{eq:c1 vertex comp aux h1200 cont}
			|\gamma|^{-1/2} \|\xi_{\bdd{a}} \nabla \psi - \nabla \mathcal{J}_{\bdd{a}} \psi\|_{\gamma} + |\xi_{\bdd{a}} \nabla \psi - \nabla \mathcal{J}_{\bdd{a}} \psi|_{\bdd{H}^{1/2}_{00}(\gamma)} \\
			\leq C \{ |\gamma|^{-1/2} \|\nabla \psi\|_{\gamma} + |\nabla \psi|_{1/2, \gamma} \}.
		\end{multline}
	\end{enumerate}
\end{lemma} 
\begin{proof}
	Let $\psi \in \Sigma_D$ and $\bdd{a} \in \mathcal{V}$. We define the function $\rho$ by the rule $\rho|_{K} = \mathcal{J}_{K}^{\bdd{a}} \psi$ for $K \in \mathcal{T}_{\bdd{a}}$ and $\rho|_{K} \equiv 0$ for $K \in \mathcal{T} \setminus \mathcal{T}_{\bdd{a}}$, where $\mathcal{J}_{K}^{\bdd{a}}$ is the linear operator in \cref{thm:c1 vertex interp element}. By \cref{thm:c1 vertex interp element} (1) and (4), $\rho \in \Sigma_D$. Define $\mathcal{J}_{\bdd{a}} : \Sigma_D \to \tilde{\Sigma}_{\bdd{a}}$ by the rule $\mathcal{J}_{\bdd{a}} \psi := \mathbb{B} \rho$. (1) now follows from \cref{eq:biharmonic extension 2}, \cref{eq:biharmonic extension 3}, and \cref{thm:c1 vertex interp element} (1). The trace theorem, \cref{eq:biharmonic continuity element}, and \cref{thm:c1 vertex interp element} (2) give \cref{eq:c1 vertex comp aux h2 cont}. Finally, \cref{eq:c1 vertex comp aux h1200 cont} follows from \cref{eq:c1 vertex interp h1200}.
\end{proof}

\begin{lemma}
	\label{lem:h2 edge comp}
	Let $\Sigma_{\mathcal{E}} := \{ \psi \in \Sigma_D : D^{\alpha} \psi(\bdd{a}) = 0 \ \forall |\alpha| \leq 2, \ \bdd{a} \in \mathcal{V} \}$. For each $\gamma$, there exists a linear interpolation operator $\mathcal{J}_{\gamma} : \Sigma_{\mathcal{E}} \to \tilde{\Sigma}_B$ with the following properties for $\psi \in \Sigma_{\mathcal{E}}$:
	\begin{subequations}
		\label{eq:h2 edge comp prop}
		\begin{alignat}{2}
			D^{\alpha} \mathcal{J}_{\gamma} \psi(\bdd{a}) &= 0 \qquad & &\forall \bdd{a} \in \mathcal{V}, \ \forall |\alpha| \leq 2, \\
			D^{\beta}\mathcal{J}_{\gamma} \psi &= D^{\beta}  \psi \qquad & & \text{on } \gamma \ \forall |\beta| \leq 1, \\
			\supp \mathcal{J}_{\gamma} \psi &\subseteq \mathcal{T}_{\gamma}, \qquad & &
		\end{alignat}
	\end{subequations}
	and
	\begin{align}
		\label{eq:h2 edge comp continuity}
		h_K^{-2} \|\mathcal{J}_{\gamma} \psi\|_{K} + h_K^{-1} |\mathcal{J}_{\gamma} \psi|_{1, K} + |\mathcal{J}_{\gamma} \psi|_{2, K} &\leq C  |\nabla \psi|_{\bdd{H}^{1/2}_{00}(\gamma)} \qquad \forall K \in \mathcal{T}_{\gamma}.
	\end{align}
\end{lemma}
\begin{proof}
	Let $\gamma \in \mathcal{E}$, $K \in \mathcal{T}_{\gamma}$ and $\psi \in \Sigma_{\mathcal{E}}$. Let $\rho$ be any function supported on $\mathcal{T}_{\gamma}$ with $\rho|_{\gamma} = \psi|_{\gamma}$ and $\partial_n \rho|_{\gamma} = \partial_n \psi|_{\gamma}$. Thanks to \cref{rem:biharmonic traces}, the operator $\mathcal{J}_{\gamma} : \Sigma_{\mathcal{E}} \to \tilde{\Sigma}_B$ given by the rule $\mathcal{J}_{\gamma} \psi = \mathbb{B} \rho$ is well-defined. Moreover, there holds
	\begin{align}
		\label{eq:proof:rhok boundary cont}
		&h_K^{-2} \|\mathcal{J}_{\gamma} \psi\|_{K} + h_K^{-1} |\mathcal{J}_{\gamma} \psi|_{1, K} + |\mathcal{J}_{\gamma} \psi|_{2, K} \\
		&\qquad \leq C \left\{h_K^{-3/2} \|\rho\|_{\partial K} + h_K^{-1/2} \| \nabla \rho \|_{\partial K} + |\nabla \rho |_{1/2,\partial K} \right\} \notag \\
		&\qquad \leq C \left\{h_K^{-3/2} \|\psi\|_{\gamma} + h_K^{-1/2} \| \nabla \psi \|_{\gamma} + | \nabla \psi |_{\bdd{H}^{1/2}_{00}(\gamma)} \right\} \notag \\
		&\qquad= C |\nabla \psi|_{\bdd{H}^{1/2}_{00}(\gamma)} \notag
	\end{align}
	for each $K \in \mathcal{T}_{\gamma}$ by \cref{eq:biharmonic continuity element} thanks to Poincar\'{e}'s inequality.
\end{proof}

\begin{proof}[Proof of \cref{thm:h2 stable decomp}]
	\Cref{eq:h2 stable decomp decomp} is a now a special case of \cref{lem:general stability}.
\end{proof}

\section{Proof of Main Result}
\label{sec:proof of main result}

The following stable decomposition result plays a crucial role in the proof of \cref{thm:main result}: 
\begin{theorem}
	\label{thm:stokes stable decomp lam}
	For every $\bdd{u} \in \tilde{\bdd{X}}_B$, there exist $\bdd{u}_C \in \tilde{\bdd{X}}_C $ and $\bdd{u}_{\bdd{a}} \in \tilde{\bdd{X}}_{\bdd{a}}$, $\bdd{a} \in \mathcal{V}$, such that
	\begin{align}
		\label{eq:stokes stable decomp lam decomp}
		\bdd{u} =  \bdd{u}_C +  \sum_{ \bdd{a} \in \mathcal{V}} \bdd{u}_{\bdd{a}} \quad \text{and} \quad a_{\lambda}(\bdd{u}_C, \bdd{u}_C)  + \sum_{ \bdd{a} \in \mathcal{V}} a_{\lambda}(\bdd{u}_{\bdd{a}}, \bdd{u}_{\bdd{a}})  \leq C  a_{\lambda}(\bdd{u}, \bdd{u}).
	\end{align}
\end{theorem}
\begin{proof}
	The space $\tilde{\bdd{N}}_B$ \cref{eq:tildenb definition} is a closed subspace of $\tilde{\bdd{X}}_B$ and $a_{\lambda}( \cdot, \cdot)$ is an inner product on $\tilde{\bdd{X}}_B$ thanks to Korn's inequality (see e.g. \cite[Corollary 11.2.22]{Brenner08}), and so the space $\tilde{\bdd{X}}_B$ admits the following orthogonal decomposition with respect to $a_{\lambda}(\cdot, \cdot)$:
	\begin{align}
		\label{eq:stokes divfree orthog decomp}
		\tilde{\bdd{X}}_B = \tilde{\bdd{N}}_B \oplus \tilde{\bdd{N}}_B^{\perp}, \quad \text{where} \quad \tilde{\bdd{N}}_B^{\perp} := \{ \bdd{w} \in  \tilde{\bdd{X}}_B :  a_{\lambda}(\bdd{w}, \bdd{v}) = 0 \ \forall \bdd{v} \in \tilde{\bdd{N}}_B\}.
	\end{align} 
	
	Let $\bdd{u} \in \tilde{\bdd{X}}_B$ be given. Thanks to the decomposition \cref{eq:stokes divfree orthog decomp}, there exist $\bdd{v} \in \tilde{\bdd{N}}_B$ and $\bdd{w} \in \tilde{\bdd{N}}_B^{\perp}$ satisfying
	\begin{align}
		\label{eq:proof:tilde u orthog decomp}
		\bdd{u} = \bdd{v} + \bdd{w} \quad \text{and} \quad a_{\lambda}(\bdd{u}, \bdd{u}) = a_{\lambda}(\bdd{v}, \bdd{v}) + a_{\lambda}(\bdd{w}, \bdd{w}).	
	\end{align}	
	Let $\bdd{v}_C \in \tilde{\bdd{X}}_C \cap \tilde{\bdd{N}}_B$ and $\bdd{v}_{\bdd{a}} \in \tilde{\bdd{X}}_{\bdd{a}} \cap \tilde{\bdd{N}}_B$, $\bdd{a} \in \mathcal{V}$, be chosen as in \cref{thm:divfree stable decomp} and $\bdd{w}_C \in \tilde{\bdd{X}}_C$ and $\bdd{w}_{\bdd{a}} \in \tilde{\bdd{X}}_{\bdd{a}}$, $\bdd{a} \in \mathcal{V}$, as in \cref{thm:h1 stable decomp 2}. Choosing
	\begin{align*}
		\bdd{u}_C := \bdd{v}_C + \bdd{w}_C \quad \text{and} \quad \bdd{u}_{\bdd{a}} := \bdd{v}_{\bdd{a}} + \bdd{w}_{\bdd{a}}
	\end{align*}
	gives the first identity in \cref{eq:stokes stable decomp lam decomp} and
	\begin{multline}
		a_{\lambda}(\bdd{u}_C, \bdd{u}_C) + \sum_{ \bdd{a} \in \mathcal{V}} a_{\lambda}(\bdd{u}_{\bdd{a}}, \bdd{u}_{\bdd{a}}) 
		= a_{\lambda}(\bdd{v}_C, \bdd{v}_C) + \sum_{ \bdd{a} \in \mathcal{V}} a_{\lambda}( \bdd{v}_{\bdd{a}}, \bdd{v}_{\bdd{a}})   \\
		+ a_{\lambda}(\bdd{w}_C, \bdd{w}_C)  + \sum_{ \bdd{a} \in \mathcal{V}} a_{\lambda}( \bdd{w}_{\bdd{a}}, \bdd{w}_{\bdd{a}})  .
	\end{multline}
	Thanks to \cref{eq:h1 stable decomp decomp 2}, there holds
	\begin{align*}
		a_{\lambda}(\bdd{w}_C, \bdd{w}_C)  + \sum_{ \bdd{a} \in \mathcal{V}} a_{\lambda}( \bdd{w}_{\bdd{a}}, \bdd{w}_{\bdd{a}}) &\leq  2(\mu + \lambda) \left\{ \|\bdd{\varepsilon}(\bdd{w}_C) \|^2  + \sum_{ \bdd{a} \in \mathcal{V}} \| \bdd{\varepsilon}( \bdd{w}_{\bdd{a}}) \|^2 \right\} \\
		&\leq C \left\{ a(\bdd{w}, \bdd{w}) + \lambda \| \bdd{\varepsilon}(\bdd{w})\|^2 \right\}.
	\end{align*}
	Applying \cite[Lemma 7.1]{AinCP22SCIP}, we obtain $\| \bdd{w}\|_1 \leq C \beta_X^{-1} \|\dive \bdd{w}\| $, and so
	\begin{align*}
		a_{\lambda}(\bdd{w}_C, \bdd{w}_C)  + \sum_{ \bdd{a} \in \mathcal{V}} a_{\lambda}( \bdd{w}_{\bdd{a}}, \bdd{w}_{\bdd{a}})  
		\leq C (1 + \beta_X^{-2}) a_{\lambda}(\bdd{w}, \bdd{w}).
	\end{align*} 
	The inequality in \cref{eq:stokes stable decomp lam decomp} now follows from \cref{eq:proof:tilde u orthog decomp,eq:divfree stable decomp decomp}.
\end{proof}

With \cref{thm:stokes stable decomp lam} in hand, we are in a position to prove \cref{thm:main result}:

\begin{proof}[Proof of \cref{thm:main result}]
	
	Let $\bdd{u} \in \bdd{X}_D$ be given.
	
	\textbf{Step 1: Cauchy-Schwarz inequality.} 
	Let $\bdd{u}_I \in \bdd{X}_I$, $\bdd{u}_C \in \bdd{X}_C$, and $\bdd{u}_{\bdd{a}} \in \bdd{X}_{\bdd{a}}$,  $\bdd{a} \in \mathcal{V}$, be any functions such that
	\begin{align*}
		\bdd{u} = \bdd{u}_I + \bdd{u}_C + \sum_{ \bdd{a} \in \mathcal{V}} \bdd{u}_{\bdd{a}}.
	\end{align*}
	The Cauchy-Schwarz inequality gives for each $K \in \mathcal{T}$,
	\begin{multline*}
		a_{\lambda, K} \left( \bdd{u}_I + \bdd{u}_C  + \sum_{ \bdd{a} \in \mathcal{V}_K} \bdd{u}_{\bdd{a}}, \bdd{u}_I + \bdd{u}_C  + \sum_{ \bdd{a} \in \mathcal{V}_K} \bdd{u}_{\bdd{a}} \right) \\
		\leq 5 \left\{ a_{\lambda, K}(\bdd{u}_I, \bdd{u}_I) +  a_{\lambda, K}( \bdd{u}_C, \bdd{u}_C ) + \sum_{ \bdd{a} \in \mathcal{V}_K}  a_{\lambda, K}( \bdd{u}_{\bdd{a}}, \bdd{u}_{\bdd{a}}) \right\}
	\end{multline*}
	and the bound
	\begin{align*}
		a_{\lambda}(\bdd{u}, \bdd{u}) \leq  5 \left\{ a_{\lambda}(\bdd{u}_I, \bdd{u}_I) +  a_{\lambda}( \bdd{u}_C, \bdd{u}_C ) + \sum_{ \bdd{a} \in \mathcal{V}}  a_{\lambda}( \bdd{u}_{\bdd{a}}, \bdd{u}_{\bdd{a}}) \right\}
	\end{align*}
	follows on summing over all elements.
	
	\textbf{Step 2: Stability for $\tilde{\bdd{X}}_B$.} Let $\tilde{\bdd{u}}_C \in \tilde{\bdd{X}}_C $ and $\tilde{\bdd{u}}_{\bdd{a}} \in \tilde{\bdd{X}}_{\bdd{a}}$, $\bdd{a} \in \mathcal{V}$, be given by \cref{thm:stokes stable decomp lam} applied to $\mathbb{S} \bdd{u} \in \tilde{\bdd{X}}_B$ and let $\bdd{u}_I = \bdd{u} - \mathbb{S} \bdd{u} \in \bdd{X}_I$. By \cref{eq:stokes stable decomp lam decomp}, we have
	\begin{align*}
		\bdd{u} = \bdd{u}_I + \tilde{\bdd{u}}_C + \sum_{ \bdd{a} \in \mathcal{V}} \tilde{\bdd{u}}_{\bdd{a}}.
	\end{align*}
	\Cref{eq:stokes seminorm continuity element,eq:div stokes ext id} then give
	\begin{align*}
		a_{\lambda}(\mathbb{S} \bdd{u}, \mathbb{S} \bdd{u}) = 2\mu \| \bdd{\varepsilon}(\mathbb{S} \bdd{u}) \|^2 + \lambda \| \Pi_I^{\perp} \dive \bdd{u} \|^2 \leq C \mu \| \bdd{\varepsilon}(\bdd{u}) \|^2 + \lambda \| \dive \bdd{u} \|^2 \leq C a_{\lambda}(\bdd{u}, \bdd{u}),
	\end{align*}
	where $\Pi_I^{\perp}$ is defined in \cref{eq:piiperp def}. Additionally,
	\begin{align*}
		a_{\lambda}(\bdd{u}_I, \bdd{u}_I) \leq 2 \left\{ a_{\lambda}(\bdd{u}, \bdd{u}) + a_{\lambda}(\mathbb{S}\bdd{u}, \mathbb{S}\bdd{u}) \right\} \leq C a_{\lambda}(\bdd{u}, \bdd{u}).
	\end{align*}
	Thus, \cref{eq:stokes stable decomp lam decomp} then gives
	\begin{align}
		\label{eq:stable decomp lam continuity}
		a_{\lambda}(\bdd{u}_I, \bdd{u}_I) + a_{\lambda}(\tilde{\bdd{u}}_C, \tilde{\bdd{u}}_C) + \sum_{ \bdd{a} \in \mathcal{V}} a_{\lambda}(\tilde{\bdd{u}}_{\bdd{a}}, \tilde{\bdd{u}}_{\bdd{a}}) 
		\leq C (1 + \beta_X^{-2}) a_{\lambda}(\bdd{u}, \bdd{u}).
	\end{align}
	\Cref{thm:main result} in the case $\bdd{X}_B = \tilde{\bdd{X}}_B$ now follows from Step 1 and \cref{eq:stable decomp lam continuity} by \cite[Remark 3.1]{Zhang92multilevel}.
	
	\textbf{Step 3: Stability for general $\bdd{X}_B$.} Let $\bdd{u}_C = \mathbb{T}_B \tilde{\bdd{u}}_C$ and $\bdd{u}_{\bdd{a}} = \mathbb{T}_B \tilde{\bdd{u}}_{\bdd{a}}$, $\bdd{a} \in \mathcal{V}$, where the $\tilde{\bdd{u}}_C$ and $\tilde{\bdd{u}}_{\bdd{a}}$ are given in Step 2. Since $(\bdd{v} - \mathbb{T}_B \bdd{v})|_{\partial K} = \bdd{0}$ for all $K \in \mathcal{T}$, $\bdd{u}_C \in \bdd{X}_C$ and $\bdd{u}_{\bdd{a}} \in \bdd{X}_{\bdd{a}}$. As a result, 
	\begin{align*}
		\mathbb{T}_B \bdd{u} = \bdd{u}_C + \sum_{ \bdd{a} \in \mathcal{V}} \bdd{u}_{\bdd{a}}
	\end{align*}
	thanks to \cref{eq:stokes stable decomp lam decomp}. Moreover, \cref{eq:t and s equivalences} and \cref{eq:stokes stable decomp lam decomp} give
	\begin{align*}
		a_{\lambda}(\bdd{u}_C, \bdd{u}_C) + \sum_{ \bdd{a} \in \mathcal{V}} a_{\lambda}(\bdd{u}_{\bdd{a}}, \bdd{u}_{\bdd{a}})  & \leq (1 + \tau_B^2) \left\{ a_{\lambda}(\tilde{\bdd{u}}_C, \tilde{\bdd{u}}_C) + \sum_{ \bdd{a} \in \mathcal{V}} a_{\lambda}(\tilde{\bdd{u}}_{\bdd{a}}, \tilde{\bdd{u}}_{\bdd{a}})  \right\} \\
		&\leq C (1 + \beta_X^{-2})(1 + \tau_B^2) a_{\lambda}(\bdd{u}, \bdd{u}).
	\end{align*}
	Letting $\bdd{u}_I := \bdd{u} - \mathbb{T}_B \bdd{u}$ and arguing similarly as in Step 2 gives
	\begin{align*}
		a_{\lambda}(\bdd{u}_I, \bdd{u}_I) \leq 2(1 + \tau_B^2)\{ a_{\lambda}(\bdd{u}, \bdd{u}) + a_{\lambda}(\mathbb{S}\bdd{u}, \mathbb{S}\bdd{u}) \} \leq C(1 + \tau_B^2) a_{\lambda}(\bdd{u}, \bdd{u}).
	\end{align*}
	Collecting results, we have
	\begin{align*}
		\bdd{u} = \bdd{u}_I + \bdd{u}_C + \sum_{ \bdd{a} \in \mathcal{V}} \bdd{u}_{\bdd{a}}
	\end{align*}
	with 
	\begin{align*}
		a_{\lambda}(\bdd{u}_I, \bdd{u}_I) + a_{\lambda}(\bdd{u}_C, \bdd{u}_C) + \sum_{ \bdd{a} \in \mathcal{V}} a_{\lambda}(\bdd{u}_{\bdd{a}}, \bdd{u}_{\bdd{a}})		\leq C (1 + \beta_X^{-2})(1 + \tau_B^2) a_{\lambda}(\bdd{u}, \bdd{u}).
	\end{align*}
	\Cref{eq:main result} now follows from the above relation and Step 1 thanks to \cite[Remark 3.1]{Zhang92multilevel}.
\end{proof}

\appendix

\section{Korn Inequalities}

The following Korn inequalities follow from \cite[Corollary 4.3]{Brenner04} and standard scaling arguments:
\begin{lemma}
	For every $K \in \mathcal{T}$, there holds
	\begin{alignat}{2}
		\label{eq:h1 korn inequality}
		\inf_{\bdd{r} \in \bdd{RM}} \left\{ h_K^{-1} \|\bdd{u} - \bdd{r}\|_{ K} + |\bdd{u} - \bdd{r}|_{1, K} \right\}  &\leq C \|\bdd{\varepsilon}(\bdd{u})\|_{K} \qquad & &\forall \bdd{u} \in \bdd{H}^1(K), \\
		\label{eq:h1 korn inequality patch}
		\inf_{\bdd{r} \in \bdd{RM}} \left\{ h_K^{-1} \|\bdd{v} - \bdd{r}\|_{\mathcal{T}_K} + |\bdd{v} - \bdd{r}|_{1, \mathcal{T}_K} \right\} &\leq C \|\bdd{\varepsilon}(\bdd{v})\|_{\mathcal{T}_K} \qquad & &\forall \bdd{v} \in \bdd{H}^1(\mathcal{T}_K),
	\end{alignat}
	where $C$ depends only on the shape regularity of $\mathcal{T}$.
\end{lemma}
\noindent 
A similar result holds for $H^2$ functions:
\begin{lemma}
	\label{lem:h2 korn inequality}
	Let $\mathcal{U} := \spann\{ 1, x, y, x^2 + y^2 \}$. For every $K \in \mathcal{T}$, $\psi \in H^2(K)$, and $\rho \in H^2(\mathcal{T}_K)$, there holds
	\begin{alignat}{1}
		\label{eq:h2 korn inequality}
		\inf_{u \in \mathcal{U}} \left\{ h_K^{-2} \|\psi - u\|_{K} + h_K^{-1} |\psi - u|_{1, K} + |\psi - u|_{2, K} \right\}   &\leq C \|\bdd{\varepsilon}(\vcurl \psi)\|_{K}, \\
		\label{eq:h2 korn inequality patch}
		\inf_{u \in \mathcal{U}} \left\{ h_K^{-2} \|\rho - u\|_{K} + h_K^{-1} |\rho - u|_{1, K} + |\rho - u|_{2, K} \right\} &\leq C \|\bdd{\varepsilon}(\vcurl \rho)\|_{\mathcal{T}_K}, 
	\end{alignat}
	where $C$ depends only on the shape regularity of $\mathcal{T}$.
\end{lemma}
\begin{proof}
	Let $K \in \mathcal{T}$ and $\rho \in H^2(\mathcal{T}_K)$. By the Bramble-Hilbert Lemma (see e.g. \cite[Lemma 4.3.8]{Brenner08}), there holds
	\begin{align*}
		\inf_{u \in \mathcal{U}} \sum_{j=0}^{2} h_K^{j-2}  |\rho - u|_{j, \mathcal{T}_K} \leq C \inf_{u \in \mathcal{U}} \sum_{j=0}^{1} h_K^{j-1}  |\vcurl \rho - \vcurl u|_{j, \mathcal{T}_K},
	\end{align*}
	where we used that $|\vcurl (\rho - u)|_{j,\mathcal{T}_K} = |\rho - u|_{j+1,\mathcal{T}_K}$. Since $\vcurl \mathcal{U} = \bdd{RM}$, we obtain
	\begin{align*}
		\inf_{u \in \mathcal{U}} \sum_{j=0}^{1} h_K^{j-1}  |\vcurl \rho - \vcurl u|_{j, \mathcal{T}_K} &= \inf_{\bdd{r} \in \bdd{RM}} \sum_{j=0}^{1} h_K^{j-1}  |\vcurl \rho - \bdd{r}|_{j, \mathcal{T}_K} 
		\\
		&\leq C \|\bdd{\varepsilon}(\vcurl \rho)\|_{\mathcal{T}_K}
	\end{align*}
	by \cref{eq:h1 korn inequality patch}. The proof of \cref{eq:h2 korn inequality} follows similar lines.
\end{proof}

\section{Discrete Exact Sequence Properties}

The structure of the various finite element spaces used in the foregoing work and their inter-relations played a key role in the analysis. The first result shows that the finite element spaces satisfy exact sequence properties:
\begin{theorem}
	\label{thm:full complex}
	Each row of the following complex is an exact sequence
	\begin{subequations}
		\label{eq:full complex}
		\begin{alignat}{9}
			\label{eq:exact sequence full spaces 2}
			&0 \quad&& \xrightarrow{ \ \ \ \subset\ \ \ } \quad && \Sigma_ {D} \quad && \xrightarrow{\ \ \vcurl \ \ } \quad && \bdd{X}_{D} \quad && \xrightarrow{ \ \ \dive \ \ } \quad && \dive \bdd{X}_{D} \quad && \xrightarrow{\qquad} \quad && 0, \\
			& && && \shortparallel_{\phantom{0}} && && \shortparallel_{\phantom{0}} && && \shortparallel_{\phantom{0}} && && \notag \\
			\label{eq:exact sequence interior spaces 2}
			&0 && \xrightarrow{ \ \ \ \subset\ \ \ } && \Sigma_ {I} && \xrightarrow{\ \ \vcurl \ \ } && \bdd{X}_{I} && \xrightarrow{ \ \ \dive \ \ } && \dive \bdd{X}_{I} && \xrightarrow{\qquad} && 0, \\
			&  && && \oplus_{\phantom{0}} && && \oplus_{\phantom{0}} && && \oplus_{\phantom{0}} && && \notag \\
			\label{eq:exact sequence exterior spaces}
			&0 && \xrightarrow{ \ \ \ \subset\ \ \ } && \tilde{\Sigma}_ {B} && \xrightarrow{\ \ \vcurl \ \ } && \tilde{\bdd{X}}_{B} && \xrightarrow{ \ \ \dive \ \ } && \dive \tilde{\bdd{X}}_{B} && \xrightarrow{\qquad} && 0,
		\end{alignat}
	\end{subequations}
	where $\Sigma_D$ \cref{eq:sigmad def}, $\Sigma_I$ \cref{eq:interior sigma spaces} and $\tilde{\Sigma}_B$ \cref{eq:boundary sigma spaces} are discrete $H^2(\Omega)$-conforming spaces.
\end{theorem}
\begin{proof}
	(i) Let $\phi \in \Sigma_D$ be given. By Poincar\'{e}'s inequality the bilinear form \\
	\noindent $a_0(\vcurl \cdot, \vcurl \cdot)$ is elliptic on $\Sigma_I$, and so there exists $\phi_I \in \Sigma_I$ such that $a(\vcurl (\phi - \phi_I), \vcurl \psi) = 0$ for all $\psi \in \Sigma_I$. The function $\tilde{\phi} := \phi - \phi_I$ then satisfies $\tilde{\phi} \in \tilde{\Sigma}_B$ by construction. Thus, $\Sigma_D = \Sigma_I \oplus \tilde{\Sigma}_B$. The relations $\bdd{X}_D = \bdd{X}_I \oplus \tilde{\bdd{X}}_B$ and $\dive \bdd{X}_D = \dive \bdd{X}_I \oplus \dive \tilde{\bdd{X}}_B$ follow immediately from \cref{rem:tilde tau}.
	
	(ii) The exactness of the first row \cref{eq:exact sequence full spaces 2} is a consequence of \cite[Lemma A.2]{AinCP22SCIP}, while the exactness of the second row \cref{eq:exact sequence interior spaces 2} follows from \cite[Theorem 3.1]{AinCP19StokesIII}. Finally, the exactness of the last row \cref{eq:exact sequence exterior spaces} may be proved along similar lines to \cite[Theorem 3.5]{AinCP19StokesIII}.
\end{proof}

\noindent The next result shows that operators $\mathbb{B}$ \cref{eq:biharmonic extension}, $\mathbb{S}$ \cref{eq:stokes extension}, and $\Pi_I^{\perp}$ \cref{eq:piiperp def} satisfy a commuting diagram property: 
\begin{lemma}
	The following diagram commutes:
	\begin{equation}
		\label{eq:stokes ext curl}
		\begin{tikzcd}[row sep=huge,column sep=large]
			\Sigma_D \arrow{r}{\vcurl} \arrow[swap]{d}{\mathbb{B}} & \bdd{X}_D \arrow{r}{\dive} \arrow[swap]{d}{\mathbb{S}} & \dive \bdd{X}_D \arrow[swap]{d}{\Pi_I^{\perp}} \\
			\tilde{\Sigma}_B \arrow{r}{\vcurl} & \tilde{\bdd{X}}_B \arrow{r}{\dive} & \dive \tilde{\bdd{X}}_B
		\end{tikzcd} 
	\end{equation}
\end{lemma}
\begin{proof}
	Let $\psi \in \Sigma_D$ and $\bdd{e} := \vcurl \mathbb{B} \psi  - \mathbb{S} \vcurl \psi$. Since $\vcurl \tilde{\Sigma}_B \subseteq \tilde{\bdd{X}}_B$ by \cref{eq:exact sequence exterior spaces},  $\bdd{e} \in \tilde{\bdd{X}}_B$ by linearity. Thanks to \cref{eq:stokes extension 3,eq:biharmonic extension 2,eq:biharmonic extension 3}, $\bdd{e} \in \bdd{X}_I$, and so $\bdd{e} \equiv \bdd{0}$ by the first column of \cref{eq:full complex}. The commutativity of \cref{eq:stokes ext curl} now follows from \cref{eq:div stokes ext id}.
\end{proof}

\section{Static Condensation satisfies Condition \cref{eq:tau condition}}

\begin{theorem}
	\label{thm:elastic harmonic continuity element}
	For $\bdd{u} \in \bdd{X}_D$ and $K \in \mathcal{T}$, the operator $\mathbb{H}$ defined in \cref{eq:eh definition} satisfies
	\begin{align}
		\label{eq:eh seminorm continuity element}
		\| \bdd{\varepsilon}(\mathbb{H} \bdd{u})\|_{K}^2 + \lambda \mu^{-1} \| \Pi_I \dive \mathbb{H} \bdd{u}\|_{K}^2 &
		\leq 4 \| \bdd{\varepsilon}(\mathbb{S} \bdd{u})\|_{K}^2 \leq C \|\bdd{\varepsilon}(\bdd{u})\|_{K}^2,
	\end{align}
	where $\mathbb{S}$ is defined in \cref{eq:stokes extension}. Consequently, \cref{eq:tau condition} holds with $\bdd{X}_B$ chosen as in \cref{eq:static cond choice} with $\tau_B$ independent of $\mu$, $\lambda$, $h$, and $p$.
\end{theorem}
\begin{proof}
	Let $\bdd{u} \in \bdd{X}_D$ and $K \in \mathcal{T}$. The function $\bdd{u}_I := \mathbb{H} \bdd{u}|_K - \mathbb{S} \bdd{u}|_K$ satisfies $\bdd{u}_I \in \bdd{X}_I(K)$ and
	\begin{align}
		\label{eq:proof:ui equation}
		a_{\lambda,K}(\bdd{u}_I, \bdd{v}) &= -2\mu( \bdd{\varepsilon}(\mathbb{S} \bdd{u}), \bdd{\varepsilon}(\bdd{v}) )_K \qquad \forall \bdd{v} \in \bdd{X}_I(K)
	\end{align}
	since $(\dive \mathbb{S} \bdd{u}, \dive \bdd{v})_K = 0$ for all $\bdd{v} \in \bdd{X}_I(K)$ by \cref{eq:stokes extension 2}. Choosing $\bdd{v} = \bdd{u}_I$ in \cref{eq:proof:ui equation} and using the Cauchy-Schwarz inequality then gives
	\begin{align*}
		\| \bdd{\varepsilon}(\bdd{u}_I) \|_{K}^2 + \frac{\lambda}{2\mu} \|\dive \bdd{u}_I\|_K^2  \leq \| \bdd{\varepsilon}(\mathbb{S} \bdd{u}) \|_{K} \| \bdd{\varepsilon}(\bdd{u}_I) \|_{K} \leq \frac{1}{2} \left\{  \| \bdd{\varepsilon}(\mathbb{S} \bdd{u}) \|_{K}^2 +    \| \bdd{\varepsilon}(\bdd{u}_I) \|_{K}^2 \right\}. 
	\end{align*}
	Moreover, $\Pi_I \dive \mathbb{H} \bdd{u}|_{K} = \dive \bdd{u}_I$ since $\Pi_I \dive \mathbb{S} \bdd{u} = 0$ by \cref{eq:stokes extension 2}. \Cref{eq:eh seminorm continuity element} now follows from the Schwarz inequality and \cref{eq:stokes seminorm continuity element}.
\end{proof}	

\section{An Efficient Inexact Interior Solver}
\label{sec:inexact interior}

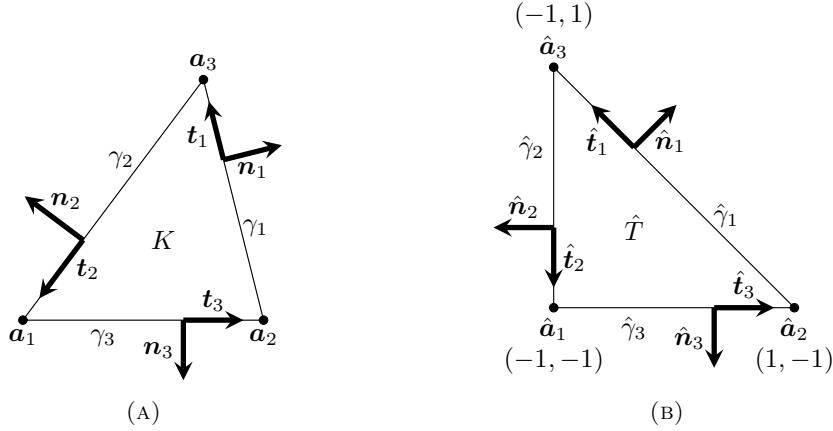
\begin{figure}[hbt]
	\centering
	\begin{subfigure}[b]{0.45\linewidth}
		\centering
		\begin{tikzpicture}[scale=0.8]
			\filldraw (-2,0) circle (2pt) node[align=center,below]{$\bdd{a}_1$}
			-- (2,0) circle (2pt) node[align=center,below]{$\bdd{a}_2$}	
			-- (1, 4) circle (2pt) node[align=center,above]{$\bdd{a}_3$}
			-- (-2,0);
			
			\coordinate (a1) at (-2,0);
			\coordinate (a2) at (2,0);
			\coordinate (a3) at (1, 4);
			
			\coordinate (e113) at ($(a2)!1/3!(a3)$);
			\coordinate (e123) at ($(a2)!2/3!(a3)$);
			\coordinate (e1231) at ($(a2)!2/3+1/sqrt(17)!(a3)$);
			
			\coordinate (e213) at ($(a3)!1/3!(a1)$);
			\coordinate (e223) at ($(a3)!2/3!(a1)$);
			\coordinate (e2231) at ($(a3)!2/3+1/sqrt(17)!(a1)$);
			
			\coordinate (e313) at ($(a1)!1/3!(a2)$);
			\coordinate (e323) at ($(a1)!2/3!(a2)$);
			\coordinate (e3231) at ($(a1)!2/3+1/4!(a2)$);
			
			\draw ($(e113)+(0.2,0.2)$) node[align=center]{$\gamma_1$};
			\draw ($(e213)+(0,0)$) node[align=center,left]{$\gamma_2$};
			\draw ($(e313)+(0,0)$) node[align=center,below]{$\gamma_3$};
			
			\draw[line width=2, -stealth] (e123) -- (e1231);
			\draw ($($(e123)!0.5!(e1231)$)+(-0.3,-0.1)$) node[align=center]{$\unitvec{t}_1$};
			\draw[line width=2, -stealth] (e123) -- ($(e123)!1!-90:(e1231)$);
			\draw ($($(e123)!0.5!-90:(e1231)$)+(0.,-0.3)$) node[align=center]{$\unitvec{n}_1$};
			\draw[line width=2, -stealth] (e223) -- (e2231);
			\draw ($($(e223)!0.5!(e2231)$)+(0.1,0)$) node[align=center, right]{$\unitvec{t}_2$};
			\draw[line width=2, -stealth] (e223) -- ($(e223)!1!-90:(e2231)$);
			\draw ($($(e223)!0.5!-90:(e2231)$)+(0.2,0)$) node[align=center, above]{$\unitvec{n}_2$};
			\draw[line width=2, -stealth] (e323) -- (e3231);
			\draw ($($(e323)!0.5!(e3231)$)+(0,0)$) node[align=center, above]{$\unitvec{t}_3$};
			\draw[line width=2, -stealth] (e323) -- ($(e323)!1!-90:(e3231)$);
			\draw ($($(e323)!0.5!-90:(e3231)$)+(0,0)$) node[align=center, left]{$\unitvec{n}_3$};
			
			\draw (1/3, 4/3) node(T){$K$};
			
		\end{tikzpicture}	
		\caption{}
		\label{fig:general triangle}
	\end{subfigure}
	\hfill
	\begin{subfigure}[b]{0.45\linewidth}
		\centering
		\begin{tikzpicture}[scale=0.8]
			\filldraw (0,0) circle (2pt) node[align=center,below]{$\hat{\bdd{a}}_1$ \\ $(-1,-1)$}
			-- (4,0) circle (2pt) node[align=center,below]{$\hat{\bdd{a}}_2$\\ $(1,-1)$}	
			-- (0,4) circle (2pt) node[align=center,above]{$(-1,1)$ \\ $\hat{\bdd{a}}_3$}
			-- (0,0);
			
			\coordinate (a1) at (0,0);
			\coordinate (a2) at (4,0);
			\coordinate (a3) at (0,4);
			
			\coordinate (e113) at ($(a2)!1/3!(a3)$);
			\coordinate (e123) at ($(a2)!2/3!(a3)$);
			\coordinate (e1231) at ($(a2)!2/3+1/sqrt(32)!(a3)$);
			
			\coordinate (e213) at ($(a3)!1/3!(a1)$);
			\coordinate (e223) at ($(a3)!2/3!(a1)$);
			\coordinate (e2231) at ($(a3)!2/3+1/4!(a1)$);
			
			\coordinate (e313) at ($(a1)!1/3!(a2)$);
			\coordinate (e323) at ($(a1)!2/3!(a2)$);
			\coordinate (e3231) at ($(a1)!2/3+1/4!(a2)$);
			
			\draw ($(e113)+(0.2,0.2)$) node[align=center]{$\hat{\gamma}_1$};
			\draw ($(e213)+(0,0)$) node[align=center,left]{$\hat{\gamma}_2$};
			\draw ($(e313)+(0,0)$) node[align=center,below]{$\hat{\gamma}_3$};
			
			\draw[line width=2, -stealth] (e123) -- (e1231);
			\draw ($($(e123)!0.5!(e1231)$)+(-0.25,-0.25)$) node[align=center]{$\hat{\unitvec{t}}_1$};
			\draw[line width=2, -stealth] (e123) -- ($(e123)!1!-90:(e1231)$);
			\draw ($($(e123)!0.5!-90:(e1231)$)+(0.25,-0.25)$) node[align=center]{$\hat{\unitvec{n}}_1$};
			\draw[line width=2, -stealth] (e223) -- (e2231);
			\draw ($($(e223)!0.5!(e2231)$)+(0,0)$) node[align=center, right]{$\hat{\unitvec{t}}_2$};
			\draw[line width=2, -stealth] (e223) -- ($(e223)!1!-90:(e2231)$);
			\draw ($($(e223)!0.5!-90:(e2231)$)+(0,0)$) node[align=center, above]{$\hat{\unitvec{n}}_2$};
			\draw[line width=2, -stealth] (e323) -- (e3231);
			\draw ($($(e323)!0.5!(e3231)$)+(0,0)$) node[align=center, above]{$\hat{\unitvec{t}}_3$};
			\draw[line width=2, -stealth] (e323) -- ($(e323)!1!-90:(e3231)$);
			\draw ($($(e323)!0.5!-90:(e3231)$)+(0,0)$) node[align=center, left]{$\hat{\unitvec{n}}_3$};

			\draw (4/3, 4/3) node(T){$\hat{T}$};
		\end{tikzpicture}		
		\caption{}
		\label{fig:reference triangle}
	\end{subfigure}
	\caption{Notation for (a) general triangle $K$ and (b) reference triangle $\hat{T}$.}
\end{figure}

For each element $K \in \mathcal{T}$, let $\bdd{F}_K : \hat{T} \to K$ denote an invertible affine mapping with Jacobian $D\bdd{F}_K$, where $\hat{T}$ is the reference triangle shown in \cref{fig:reference triangle}. Given a function $\bdd{u} \in \bdd{X}_D$ and element $K \in \mathcal{T}$, let $\check{\bdd{u}}_K := D\bdd{F}_K^{-1} \bdd{u} \circ \bdd{F}_K$ be the Piola transformation of $\bdd{u}$ and introduce the following bilinear form
\begin{align}
	\label{eq:inexact bilinear def}
	b_{\lambda}(\bdd{u}, \bdd{v}) := \sum_{K \in \mathcal{T}} b_{\lambda, K}(\bdd{u}, \bdd{v}),
\end{align}
where
\begin{align}
	\label{eq:inexact bilinear element def}
	b_{\lambda, K}(\bdd{u}, \bdd{v}) := |K| \int_{\hat{T}} \{ 2\mu \bdd{\varepsilon}(\check{\bdd{u}}_K) : \bdd{\varepsilon}(\check{\bdd{v}}_K ) + \lambda (\dive \check{\bdd{u}}_K)(\dive \check{\bdd{v}}_K) \} \ d\bdd{x} \qquad \forall \bdd{u}, \bdd{v} \in \bdd{X}_D.
\end{align}
Later, in \cref{lem:blamk equivalences}, we show that the bilinear form $b_{\lambda, K}(\cdot,\cdot)$ is bounded and coercive on $\bdd{X}_I(K)$, and so the problems 
\begin{align}
	\label{eq:local interior inexact}
	\bdd{u}_{I,K} \in \bdd{X}_I(K) : \qquad b_{\lambda, K}(\bdd{u}_{I,K}, \bdd{v}) = r(\bdd{v}) \qquad \forall \bdd{v} \in \bdd{X}_I(K), \ K \in \mathcal{T},
\end{align}
are uniquely solvable. Consequently, the boundary space $\check{\bdd{X}}_B$ defined in \cref{eq:inexact interior choice} is well-defined for this choice \cref{eq:inexact bilinear def}. 

Similarly to \cref{eq:static cond choice} in the case of $\bdd{X}_B$, we may characterize $\check{\bdd{X}}_B$ as the image of $\bdd{X}_D$ under the operator $\check{\mathbb{H}} : \bdd{X} \to \bdd{X}$ defined by the rule
\begin{subequations}
	\label{eq:ehhat definition}
	\begin{alignat}{2}
		\label{eq:ehhat definition 1}
		b_{\lambda}( \check{\mathbb{H}} \bdd{u}, \bdd{v} ) &= 0 \qquad & &\forall \bdd{v} \in \bdd{X}_I, \\
		\label{eq:ehhat definition 2}
		(\check{\mathbb{H}} \bdd{u} - \bdd{u})|_{\partial K} &= \bdd{0} \qquad 
		& & \forall K \in \mathcal{T}.
	\end{alignat}
\end{subequations}
Now let $\bdd{u} \in \bdd{X}_D$ be given. Then, $\bdd{u} = \bdd{u}_I + \check{\mathbb{H}} \bdd{u}$, where $\bdd{u}_I := \bdd{u} - \check{\mathbb{H}} \bdd{u} \in \bdd{X}_I$ by \cref{eq:ehhat definition 2}. $\check{\mathbb{H}} \bdd{u} \in \check{\bdd{X}}_B$ by \cref{eq:ehhat definition 1} and so $\bdd{X}_D = \bdd{X}_I \oplus \check{\bdd{X}}_B$.
In other words, $\bdd{X}_D =  \bdd{X}_I \oplus \check{\bdd{X}}_B$ for $\check{\bdd{X}}_B$ given in \cref{eq:inexact interior choice}, and $\check{\mathbb{H}}$ is the operator $\mathbb{T}_B$ for the choice \cref{eq:inexact interior choice} with $b_{\lambda}(\cdot,\cdot)$ defined in \cref{eq:inexact bilinear def}. Moreover, the operator $\check{\mathbb{H}}$ satisfies \cref{eq:tau condition} with $\tau_B$ independent of $\mu$, $\lambda$, $h$, and $p$:
\begin{theorem}
	\label{thm:elastic harmonic hat continuity element}
	For $\bdd{u} \in \bdd{X}$ and $K \in \mathcal{T}$, the operator $\check{\mathbb{H}}$ satisfies
	\begin{align}
		\label{eq:ehhat seminorm continuity element}
		\| \bdd{\varepsilon}(\check{\mathbb{H}} \bdd{u})\|_{K}^2 + \lambda \mu^{-1} \| \Pi_I \dive \check{\mathbb{H}} \bdd{u}\|_{K}^2 &
		\leq C \| \bdd{\varepsilon}(\mathbb{S} \bdd{u})\|_{K}^2 \leq C \|\bdd{\varepsilon}(\bdd{u})\|_{K}^2,
	\end{align}
	where $\mathbb{S}$ is defined in \cref{eq:stokes extension}.
\end{theorem}
In summary, the space $\check{\bdd{X}}_B$ may be used in place of the choice $\bdd{X}_B$ corresponding to static condensation \cref{eq:static cond choice} without degrading the performance of the preconditioner (up to a constant $C$).

The main advantage of choosing the bilinear form as in \cref{eq:inexact bilinear def} is that the solutions to \cref{eq:local interior inexact} may be computed efficiently as follows. Let $\bdd{X}_I(\hat{T}) =  \spann \{\hat{\psi}_1 \unitvec{e}_1, \ldots,$ $\hat{\psi}_n \unitvec{e}_1, \hat{\psi}_1 \unitvec{e}_2,  \ldots, \hat{\psi}_n \unitvec{e}_2 \}$, where $\{ \hat{\psi}_i \}_{i=1}^{n}$ is a basis for $\mathcal{P}_{p}(\hat{T}) \cap H^1_0(\hat{T})$ and $\{\unitvec{e}_i\}$ is the standard basis for $\mathbb{R}^2$. A basis for the space $\bdd{X}_I(K)$ is then obtained by taking the standard pullback of the basis for $\bdd{X}_I(\hat{T})$. Let $\widehat{\bdd{B}}$ be the stiffness matrix corresponding to the reference bilinear form:
\begin{align*}
	\vec{v}^T \widehat{\bdd{B}} \vec{u} = \int_{\hat{T}} \left\{ 2\mu \bdd{\varepsilon}(\bdd{u}) : \bdd{\varepsilon}(\bdd{v}) + \lambda (\dive \bdd{u})(\dive \bdd{v}) \right\} \ d\bdd{x} \qquad \forall \bdd{u}, \bdd{v} \in \bdd{X}_I(\hat{T}),
\end{align*}
where $\vec{u}$ and $\vec{v}$ are the degrees of freedom of $\bdd{u}$ and $\bdd{v}$. Direct verification then shows that
\begin{align*}
	b_{\lambda,K}(\bdd{u}, \bdd{v}) = |K| \vec{v}^T \bdd{G}^{-1} \widehat{\bdd{B}} \bdd{G}^{-T} \vec{u} \qquad \forall \bdd{u}, \bdd{v} \in \bdd{X}_I(K),
\end{align*}
where $\bdd{G}^{-1} := D\bdd{F}_K^{-1} \otimes \bdd{I}_{n}$, $\otimes$ denotes the Kronecker product, and $\bdd{I}_n$ is the $n\times n $ identity matrix. Thus, the solutions to the problems \cref{eq:local interior inexact} may be computed by applying the action of the matrix 
\begin{align}
	\label{eq:inexact matrix}
	|K|^{-1} \bdd{G}^{T} \widehat{\bdd{B}}^{-1} \bdd{G}, \quad \text{where } \bdd{G} = D\bdd{F}_K \otimes \bdd{I}_{n}.
\end{align}
The main cost associated with \cref{eq:inexact matrix} is to compute the action of $\widehat{\bdd{B}}^{-1}$. The fact that $\widehat{\bdd{B}}$ is independent of the element geometry means that factoring the matrix $\widehat{\bdd{B}}$ incurs a one time setup cost of $\mathcal{O}(p^6)$. Computing the action of the matrix  \cref{eq:inexact matrix} on each element requires $\mathcal{O}(p^2)$ operations to apply $\bdd{G}$ and $\bdd{G}^T$ and in addition $\mathcal{O}(p^4)$ operations to apply $\widehat{\bdd{B}}^{-1}$ using the factored form, giving an overall application cost of $\mathcal{O}(|\mathcal{T}| p^4)$ operations.

As shown in \cref{thm:elastic harmonic hat continuity element}, $b_{\lambda}(\cdot,\cdot)$ satisfies \cref{eq:blam equiv alam}  with $\tau_B$ independent of $\mu$, $\lambda$, $h$, and $p$, and we show below in \cref{lem:blamk equivalences} that $b_{\lambda}(\cdot,\cdot)$ satisfies \cref{eq:blam equiv alam}. Consequently, $b_{\lambda}(\cdot,\cdot)$ is a valid choice for an inexact interior solve as described in \cref{sec:inexact}, and the solution to \cref{eq:local interior inexact} may be computed more efficiently than using an exact interior solve $\bdd{A}_{II}^{-1}$ in static condensation. 

\subsection{Stability of Inexact Interior Solve}

We now turn to the proofs of the aforementioned claims.
\begin{lemma}
	\label{lem:blamk equivalences}
	For any $\bdd{v} \in \bdd{X}_I$ and $K \in \mathcal{T}$, there holds
	\begin{align}
		\label{eq:blamk equivlances}
		C_1 a_{\lambda,K}(\bdd{v}, \bdd{v}) \leq
		b_{\lambda,K}(\bdd{v}, \bdd{v}) \leq C_2 a_{\lambda,K}(\bdd{v}, \bdd{v}),
	\end{align}
	where $C_1$ and $C_2$ depend only on shape regularity. Consequently, $b_{\lambda}(\cdot,\cdot)$ defined in \cref{eq:inexact bilinear def} satisfies \cref{eq:blam equiv alam}.
\end{lemma}
\begin{proof}
	Let $\bdd{v} \in \bdd{X}_I$ be given and  $\check{\bdd{v}} = D\bdd{F}_K^{-1} \bdd{v} \circ \bdd{F}_K$. Then,
	\begin{align*}
		|K|^{-1} b_{\lambda,K}(\bdd{v}, \bdd{v}) = 2\mu \| \bdd{\varepsilon}(\check{\bdd{v}})\|_{\hat{T}}^2 + \lambda \|\dive \check{\bdd{v}} \|_{\hat{T}}^2.
	\end{align*}
	The chain rule shows that $\dive \check{\bdd{v}} = (\dive \bdd{v}) \circ \bdd{F}_K$ on $\hat{T}$ and so $|K|^{1/2} \|\dive \check{\bdd{v}} \|_{\hat{T}} =  \|\dive \bdd{v} \|_{K}$. Applying the chain rule once more and using shape regularity gives
	\begin{align}
		\label{eq:proof:hatv one sided eps bound}
		\| \bdd{\varepsilon}(\check{\bdd{v}})\|_{\hat{T}}^2 \leq | \check{\bdd{v}}|_{1, \hat{T}}^2 \leq C h_K^{-2} | \bdd{v}|_{1, K}^2 \leq C |K|^{-1} | \bdd{v}|_{1, K}^2.
	\end{align}
	Since $\bdd{v}|_{K} \in \bdd{H}^1_0(K)$, using \cite[eq. (3.19), p. 299]{Braess07} gives $|\bdd{v}|_{1,K} \leq \sqrt{2} \|\bdd{\varepsilon}(\bdd{v})\|_K$, and so $b_{\lambda,K}(\bdd{v}, \bdd{v}) \leq C a_{\lambda,K}(\bdd{v}, \bdd{v})$. Conversely, using the relation $\bdd{v} = D\bdd{F}_K \check{\bdd{v}} \circ \bdd{F}_K^{-1}$ gives
	\begin{align}
		\label{eq:proof:hatv one sided eps bound 2}
		\| \bdd{\varepsilon}(\bdd{v})\|_{K}^2 \leq |\bdd{v}|_{1,K}^2 \leq C h_K^2  |\check{\bdd{v}}|_{1,\hat{T}}^2 \leq C |K|  \| \bdd{\varepsilon}(\check{\bdd{v}})\|_{\hat{T}}^2, 
	\end{align}
	and so $a_{\lambda,K}(\bdd{v}, \bdd{v}) \leq C b_{\lambda,K}(\bdd{v}, \bdd{v})$.
\end{proof}

\begin{proof}[Proof of \cref{thm:elastic harmonic hat continuity element}]
	Let $\bdd{u} \in \bdd{X}$ and $K \in \mathcal{T}$. The function $\bdd{u}_I := \check{\mathbb{H}} \bdd{u}|_K - \mathbb{S} \bdd{u}|_K$ satisfies $\bdd{u}_I \in \bdd{X}_I(K)$ and
	\begin{align}
		\label{eq:proof:uihat equation}
		b_{\lambda,K}(\bdd{u}_I, \bdd{v}) &= -2\mu( \bdd{\varepsilon}( \check{\bdd{s}} ), \bdd{\varepsilon}(\check{\bdd{v}}) )_{\hat{T}} \qquad \forall \bdd{v} \in \bdd{X}_I(K),
	\end{align}
	where $\bdd{s} := \mathbb{S} \bdd{u}$ and we used that 
	\begin{align*}
		(\dive \check{\bdd{s}}, \dive \check{\bdd{v}})_{\hat{T}} = (\dive (\mathbb{S} \bdd{u}) \circ \bdd{F}_K, (\dive \bdd{v}) \circ \bdd{F}_K)_{\hat{T}} = |K|^{-1} (\dive \mathbb{S} \bdd{u}, \dive \bdd{v})_{K} = 0
	\end{align*}
	for all $\bdd{v} \in \bdd{X}_I(K)$ by \cref{eq:stokes extension 2}. Choosing $\bdd{v} = \bdd{u}_I$ in \cref{eq:proof:uihat equation} and using the Cauchy-Schwarz inequality then gives
	\begin{align*}
		\| \bdd{\varepsilon}(\check{\bdd{u}}_I) \|_{\hat{T}}^2 + \lambda\mu^{-1} \|\dive \check{\bdd{u}}_I\|_{\hat{T}}^2  \leq \| \bdd{\varepsilon}(\check{\bdd{s}}) \|_{\hat{T}} \| \bdd{\varepsilon}(\check{\bdd{u}}_I) \|_{\hat{T}} \leq \frac{1}{2} \left\{   \| \bdd{\varepsilon}(\check{\bdd{s}}) \|_{\hat{T}}^2 +    \| \bdd{\varepsilon}(\check{\bdd{u}}_I) \|_{\hat{T}}^2 \right\}.
	\end{align*}
	Thanks to \cref{eq:proof:hatv one sided eps bound 2} and the relation $\dive \check{\bdd{u}}_I = (\dive \bdd{u}_I) \circ \bdd{F}_K$, we have
	\begin{align*}
		\| \bdd{\varepsilon}(\bdd{u}_I) \|_{K}^2 + \lambda\mu^{-1} \|\dive \bdd{u}_I\|_{K}^2  \leq C |K| \| \bdd{\varepsilon}(\check{\bdd{s}}) \|_{\hat{T}}^2 \leq C |\mathbb{S} \bdd{u} |_{1,K}^2,
	\end{align*}
	where $C$ depends only on shape regularity and we used that \cref{eq:proof:hatv one sided eps bound} holds for any $\bdd{v} \in \bdd{H}^1(K)$. Arguing as in the proof of \cref{thm:elastic harmonic continuity element}, we obtain $\Pi_I \dive \check{\mathbb{H}} \bdd{u}|_{K} = \dive \bdd{u}_I$, and the triangle inequality then gives
	\begin{align*}
		\| \bdd{\varepsilon}(\check{\mathbb{H}} \bdd{u}) \|_{K}^2 + \lambda\mu^{-1} \|\Pi_I \dive \check{\mathbb{H}} \bdd{u}\|_{K}^2  \leq C |\mathbb{S} \bdd{u} |_{1,K}^2.
	\end{align*}
	For any rigid body motion $\bdd{r} \in \bdd{RM}$, $\check{\bdd{r}}$ is a linear function satisfying $\dive \check{\bdd{r}} = \dive \bdd{r} = 0$, and so for any $\bdd{v} \in \bdd{X}_I(K)$, there holds
	\begin{align*}
		b_{\lambda, K}(\bdd{r}, \bdd{v}) 
		=  2\mu (\bdd{\varepsilon}(\check{\bdd{r}}),  \nabla \check{\bdd{v}} )_{\hat{T}}
		= 2\mu (\dive \bdd{\varepsilon}(\check{\bdd{r}}), \check{\bdd{v}})_{\hat{T}} = 0.
	\end{align*}
	Consequently,  $\check{\mathbb{H}} \bdd{r} = \bdd{r}$. Since $\mathbb{S} \bdd{r} = \bdd{r}$, we obtain
	\begin{align*}
		\| \bdd{\varepsilon}(\check{\mathbb{H}} \bdd{u}) \|_{K}^2 + \lambda\mu^{-1} \|\Pi_I \dive \check{\mathbb{H}} \bdd{u}\|_{K}^2 \leq C \inf_{\bdd{r} \in \bdd{RM}} |\mathbb{S} \bdd{u} - \bdd{r} |_{1,K}^2.
	\end{align*}
	\Cref{eq:ehhat seminorm continuity element} now follows from Korn's inequality \cref{eq:h1 korn inequality} and \cref{eq:stokes seminorm continuity element}.
\end{proof}	

\section{Auxiliary Interpolation Operators}

Let $I = (-1, 1)$ and let $H^{1/2}_{L}(I)$ and $H^{1/2}_R(I)$ denote the spaces
\begin{align*}
	H^{1/2}_{L}(I) &:= \{ v \in H^{1/2}(I) : (1+t)^{-1/2} v \in L^2(I) \}, \\
	H^{1/2}_R(I)  &:= \{ v \in H^{1/2}(I) : (1-t)^{-1/2} v \in L^2(I) \}, 
\end{align*}
equipped with the norms
\begin{align*}
	\| v\|_{H^{1/2}_L(I)}^2 &:= \| v \|_{H^{1/2}(I)}^2 + \int_{I} \frac{|v(t)|^2}{1 + t} \ dt, \\
	\| v\|_{H^{1/2}_R(I)}^2 &:= \| v \|_{H^{1/2}(I)}^2 + \int_{I} \frac{|v(t)|^2}{1-t} \ dt.
\end{align*}
In addition, the space $H^{1/2}_{00}(I)$ is defined to be the intersection of $H^{1/2}_{L}(I)$ and $H^{1/2}_R(I)$. 

We formally define operators $\mathcal{A}_R$ and  $\mathcal{M}_{R}$ by the rules
\begin{align*}
	\mathcal{A}_R v(t) := \frac{1}{1+t} \int_{-1}^{t} v(s) \ ds -  \frac{1 + t}{2} \bar{v} \quad \text{and} \quad \mathcal{M}_{R} v(t) = \frac{1-t}{2} v(t), \qquad t \in I,
\end{align*}
where $\bar{v}$ denotes the average value of $v$ on $I$. In particular, if $v \in \mathcal{P}_p(I)$, $p \geq 1$, then $\mathcal{A}_R v \in \mathcal{P}_p(I)$ and $\mathcal{M}_{R} v \in \mathcal{P}_{p+1}(I)$ satisfy $\mathcal{A}_R v(-1) = \mathcal{M}_{R} v(-1) = v(-1)$ and $\mathcal{A}_R v(1) = \mathcal{M}_{R} v(1) = 0$. The operators $\mathcal{A}_R$ and $\mathcal{M}_R$ have the following useful mapping properties:
\begin{lemma}
	\label{lem:hardy0 properties}
	The operators $\mathcal{A}_{R} : H^{1/2}(I) \to H_{R}^{1/2}(I)$, $(I - \mathcal{A}_R) : H^{1/2}(I) \to H^{1/2}_{L}(I)$, and $(\mathcal{M}_{R} - \mathcal{A}_R) : H^{1/2}(I) \to H^{1/2}_{00}(I)$ are continuous: i.e. for all $v \in H^{1/2}(I)$, there holds
	\begin{align}
		\label{eq:hardy0 h12 h1200}
		\| \mathcal{A}_R v \|_{H_{R}^{1/2}(I)} + \| (I - \mathcal{A}_R) v \|_{H^{1/2}_L(I)} + \| (\mathcal{M}_{R} - \mathcal{A}_R) v \|_{H^{1/2}_{00}(I)} \leq C \| v\|_{1/2, I}.
	\end{align}
\end{lemma}
\begin{proof}
	Define the Hardy averaging operator $\mathcal{A}$ by the rule
	\begin{align*}
		\mathcal{A} v(t) := \frac{1}{1+t} \int_{-1}^{t} v(s) \ ds, \qquad t \in I,
	\end{align*} 
	so that $A_R v(t) = \mathcal{A} v(t) - (1+t)\bar{v}/2$. For $n \in \mathbb{N}_0$, the Cauchy-Schwarz inequality and \cite[Lemma 3.1]{Ain09poly} give
	\begin{align}
		\label{eq:hardy0 hn}
		\| \mathcal{A}_R v \|_{n, I} \leq C \left\{ \|\mathcal{A} v\|_{n, I} +  |\bar{v}_I| \right\} \leq C_n  \left\{ \|v\|_{n, I} + \|v\|_I \right\}  \leq C_n \|v\|_{n,I} \qquad \forall v \in H^n(I).
	\end{align}	
	
	Now let $v \in H^1(I)$. Thanks to the embedding $H^1(I) \hookrightarrow C^0(\bar{I})$, $\mathcal{A}_R v(1) = (I - \mathcal{A}_R) v(-1) = (\mathcal{M}_{R} - \mathcal{A}_R) v(\pm 1) = 0$. Consequently, $\mathcal{A}_R$ is a bounded map of $H^1(I)$ into $H^1_R(I) :=  \{ w \in H^1(I) : w(1) = 0 \}$, $(I - \mathcal{A}_R)$ is a bounded map of $H^1(I)$ into $H^1_L(I) := \{ w \in H^1(I) : w(1) = 0 \}$, and $(\mathcal{M}_{R} - \mathcal{A}_R)$ is a bounded map of $H^1(I)$ into $H^1_0(I)$. \Cref{eq:hardy0 h12 h1200} then follows from \cref{eq:hardy0 hn} and interpolation (see e.g. \cite[p. 66, Theorem 11.7]{Lions12}). 
\end{proof}

\subsection{Endpoint Modifications of Averaging Operator}

\begin{lemma}
	\label{lem:hardy00 properties}
	Let $p \geq 4$. For all $v \in\mathcal{P}_p(I)$, there exists a polynomial $\mathcal{B}_R v \in \mathcal{P}_p(I)$ satisfying $(\mathcal{B}_{R} v)^{(n)}(-1) = v^{(n)}(-1)$ and $(\mathcal{B}_{R} v)^{(n)}(1) = 0$ for $n \in \{0, 1\}$, $\int_{I} \mathcal{B}_R v(t) \ dt = 0$, and  
	\begin{align}
		\label{eq:hardy00 h12 h1200}
		\| \mathcal{B}_R v \|_{H_{R}^{1/2}(I)} + \| (I - \mathcal{B}_R) v \|_{H^{1/2}_L(I)} + \| (\mathcal{M}_{R} - \mathcal{B}_R) v \|_{H^{1/2}_{00}(I)}
		\leq C \| v\|_{1/2, I}.
	\end{align}
\end{lemma}
\begin{proof}
	For $n \in \mathbb{N}_0$, let $P_n^{(\alpha,\beta)}$ denote the Jacobi polynomial \cite{Szego75} and define $\Phi_{p} \in \mathcal{P}_p(I)$ as follows:
	\begin{align*}
		\Phi_{p}(t) := \frac{1}{4} (1-t^2) (1-t) \frac{P_{p-3}^{(3,3)}(t)}{P^{(3,3)}_{p-3}(-1)}, \qquad t \in I.
	\end{align*}
	It may be shown \cite[Lemma B.1]{AinCP19Precon} that $\Phi_{p}$ satisfies
	\begin{align}
		\label{eq:Phi properties}
		\Phi_{p}^{(n)}(-1) = \delta_{ m n}, \quad \Phi_{p}^{(n)}( 1) = 0, \ 0 \leq n \leq 1, \quad \text{with} \quad \|\Phi_{p}\|_{I} \leq C p^{-3}.
	\end{align}
	The fact that $\Phi_{p}$ has a repeated zero of (at least) order $1$ at $\pm 1$ means that $\Phi_{p} \in H^{1/2}_{00}(I)$. By interpolation, the inverse inequality  \cite[p. 95, Theorem 5.1]{Bern07}, and \cref{eq:Phi properties}, we obtain
	\begin{align}
		\label{eq:Phi hm1200 decay}
		\| \Phi_{p} \|_{H^{1/2}_{00}(I)} \leq C \|\Phi_{p} \|_{I}^{1/2} \| \Phi_{p} \|_{1,I}^{1/2} \leq C p \|\Phi_{p} \|_{I} \leq C p^{-2}.
	\end{align}
	
	Now let $v \in\mathcal{P}_p(I)$ be given and define
	\begin{alignat*}{2}
		\mathcal{C}_{R} v(t) &:= \mathcal{A}_R v(t) + (v - \mathcal{A}_R v)'(-1) \Phi_{p}(t) +  (\mathcal{A}_R v)'(1) \Phi_{p}(-t), \qquad & &t \in I, \\
		\mathcal{B}_R v(t) &:= \mathcal{C}_R v(t) - \overline{\mathcal{C}_R v} \frac{(1-t^2)^2}{\int_{I} (1-t^2)^2 \ dt} \qquad & & t \in I.
	\end{alignat*}
	The properties $(\mathcal{B}_{R} v)^{(n)}(-1) = v^{(n)}(-1)$ and $(\mathcal{B}_{R} v)^{(n)}(1) = 0$ for $n \in \{0, 1\}$ and $\int_{I} \mathcal{B}_R v(t) \ dt = 0$ follow from \cref{eq:Phi properties}.
	
	Thanks to \cref{eq:Phi properties,eq:hardy0 h12 h1200}, there holds
	\begin{align*}
		\| \mathcal{C}_R v \|_{H_{R}^{1/2}(I)}  &\leq \|\mathcal{A}_{R} v\|_{H_{R}^{1/2}(I)} + \left\{  \|v\|_{W^{1,\infty}(I)} + \|\mathcal{A}_R v\|_{W^{1,\infty}(I)} \right\} \| \Phi_{1,p} \|_{H^{1/2}_{00}(I)} \\
		&\leq C \left( \|v\|_{1/2, I} +  \left\{  \|v\|_{W^{1,\infty}(I)} + \|\mathcal{A}_R v\|_{W^{1,\infty}(I)} \right\} p^{-2} \right).
	\end{align*}
	Applying the inverse inequality \cite[p. 95, Theorem 5.1]{Bern07}  gives
	\begin{align*}
		\|v\|_{W^{1,\infty}(I)} \leq C p^{2} \|v\|_{1/2, I} \quad \text{and} \quad  \|\mathcal{A}_R v\|_{W^{1,\infty}(I)} \leq C p^2 \|\mathcal{A}_R v\|_{1/2, I}.
	\end{align*} 
	Using \cref{eq:hardy0 h12 h1200} once more shows that $\| \mathcal{C}_R v \|_{H_{R}^{1/2}(I)} \leq  C \|v\|_{1/2, I}$. Since $(1-t^2)^2$ vanishes at $\pm 1$, $(1-t^2)^2 \in H^{1/2}_{00}(I)$, and so
	\begin{align*}
		\| \mathcal{B}_R v \|_{H_{R}^{1/2}(I)}  &\leq \| \mathcal{C}_R v \|_{H_{R}^{1/2}(I)}  + C \| \mathcal{C}_R v\|_{I} \leq C \|v\|_{1/2, I}.
	\end{align*}
	The remaining terms in \cref{eq:hardy00 h12 h1200} may be bounded analogously using \cref{eq:hardy0 h12 h1200}.	
\end{proof}

\subsection{Interpolation Operators on Triangles}

Let $\hat{T}$ be the reference element labeled as in \cref{fig:reference triangle} and let $\{\xi_i\}_{i=1}^{3}$ be the barycentric coordinates on $\hat{T}$. Given a polynomial $\bdd{v} \in \bm{\mathcal{P}}_{p}(\hat{T})$, we parameterize the traces of $\bdd{v}$ on $\hat{\gamma}_2$ and $\hat{\gamma}_3$ as follows:
\begin{alignat*}{2}
	\bdd{v}|_{\hat{\gamma}_2}(t) &= \bdd{v}_2(t) = \bdd{v}\left( \frac{1-t}{2} \hat{\bdd{a}}_1 + \frac{1+t}{2} \hat{\bdd{a}}_3 \right), \qquad & & t \in I, \\ 
	\bdd{v}|_{\hat{\gamma}_3}(t) &= \bdd{v}_3(t) = \bdd{v} \left( \frac{1-t}{2} \hat{\bdd{a}}_1 + \frac{1+t}{2} \hat{\bdd{a}}_2 \right), \qquad & &t \in I.
\end{alignat*}

\begin{theorem}
	\label{thm:c0 vertex interp ref element}
	For each $\bdd{v} \in \bm{\mathcal{P}}_{p}(\hat{T})$, there exists a polynomial $\mathcal{G} \bdd{v} \in \bm{\mathcal{P}}_p(T)$ satisfying
	\begin{enumerate}
		\item[(1)] $D^{\alpha} \mathcal{G} \bdd{v}(\hat{\bdd{a}}_1) = D^{\alpha} \bdd{v}(\hat{\bdd{a}}_1)$, $|\alpha| \leq 1$, and 
		$\mathcal{G} \bdd{v}|_{\hat{\gamma}_i} = \begin{cases}
			\bdd{0} & i = 0, \\
			\mathcal{B}_R \bdd{v}_i & i = 1, 2.
		\end{cases}$
			
		\item[(2)] $\| \mathcal{G} \bdd{v} \|_{1,\hat{T}} \leq C \|\bdd{v}\|_{1, \hat{T}}$.
		
		\item[(3)]$\| \xi_1 \bdd{v} -  \mathcal{G} \bdd{v} \|_{\bdd{H}^{1/2}_{00}(\hat{\gamma}_i)} \leq C \| \bdd{v} \|_{1/2, \hat{\gamma}_i}$, $i=2,3$.
		
		\item[(4)] If $\bdd{v}|_{\hat{\gamma}_i} = \bdd{0}$, then $\mathcal{G} \bdd{v}|_{\hat{\gamma}_i} = \bdd{0}$.
	\end{enumerate}
\end{theorem}
\begin{proof}
	Let $\bdd{v} \in \bm{\mathcal{P}}_{p}(\hat{T})$ and define $\bdd{f} : \bdd{L}^2(\partial \hat{T}) \to \mathbb{R}^2$ by the rule
	\begin{align*}
		\bdd{f}|_{\hat{\gamma}_i}(t) = \mathcal{B}_{R} \bdd{v}_i(t), \qquad i=2, 3 \quad \text{and} \quad \bdd{f}|_{\hat{\gamma}_1} = \bdd{0}.
	\end{align*}
	Thanks to \cref{lem:hardy00 properties}, $\bdd{f}|_{\hat{\gamma}_i}(\hat{\bdd{a}}_1) = \bdd{v}(\hat{\bdd{a}}_1)$ for $i=2,3$ and $\bdd{f}$ vanishes at $\hat{\bdd{a}}_2$ and $\hat{\bdd{a}}_3$. Consequently $\bdd{f}$ is a continuous piecewise polynomial on $\partial \hat{T}$. Thanks to \cite[Theorem 7.4]{BCMP91} applied component-wise, there exists $\mathcal{G} \bdd{v} \in \bm{\mathcal{P}}_p(\hat{T})$ satisfying
	\begin{align*}
		\mathcal{G} \bdd{v}|_{\partial \hat{T}} = \bdd{f} \quad \text{and} \quad \| \mathcal{G} \bdd{v} \|_{1, \hat{T}} \leq C \|\bdd{f}\|_{1/2, \partial T}.
	\end{align*}
	(1) now follows from \cref{lem:hardy00 properties} on noting that the gradient of $\bdd{v}$ can be expressed in terms of $\bdd{v}_2'$ and $\bdd{v}_3'$. To show (2), we express the $\bdd{H}^{1/2}(\partial \hat{T})$ as follows:
	\begin{align*}
		\|{\bdd{f}}\|_{1/2, \partial \hat{T}}^2 &\leq C \left\{  \sum_{i \in \{2, 3\}} \|\bdd{f}_i\|_{H^{1/2}_{R}(I)}^2 +  \int_{I} \frac{|{\bdd{f}}_2(t) - {\bdd{f}}_3(t)|^2}{1+t} \ dt   \right\}  \\
		&\leq C \left\{  \sum_{i \in \{2, 3\}} \left( \|\mathcal{B}_R \bdd{v}_i\|_{H^{1/2}_{R}(I)}^2 + \| (I - \mathcal{B}_R) \bdd{v}_i \|_{H^{1/2}_{L}(I)}^2 \right) \right. \\
		&\qquad \qquad  \left. \vphantom{\sum_{i \in \{2, 3\}}}  + \int_{I} \frac{|\bdd{v}_2(t) - \bdd{v}_3(t)|^2}{1+t} \ dt \right\}  \\
		&\leq C \|\bdd{v} \|_{1/2, \partial \hat{T}}^2 \leq C \|\bdd{v} \|_{1, \hat{T}}^2,
	\end{align*}
	where we used \cref{eq:hardy00 h12 h1200} and the trace theorem. (3) follows from \cref{eq:hardy00 h12 h1200}, while (4) is a consequence of the identity $\mathcal{B}_R \bdd{0} = \bdd{0}$.
\end{proof}

\Cref{thm:c0 vertex interp ref element} remains true if the reference element is replaced by an element $K \in \mathcal{T}$. Specifically, let $\gamma \in \mathcal{E}$ be an edge with endpoints $\bdd{a}, \bdd{b} \in \mathcal{V}$ and parameterize $\gamma$ as follows:
\begin{align*}
	\varphi_{\gamma, \bdd{a}}(t) = \frac{1-t}{2} \bdd{a} + \frac{1+t}{2}  \bdd{b}, \qquad t \in I.
\end{align*}
Define the operator $\mathcal{B}_{\gamma}^{\bdd{a}} : \bdd{X} \to \bm{\mathcal{P}}_p (\gamma)$ by the rule
\begin{align*}
	\mathcal{B}_{\gamma}^{\bdd{a}} \bdd{v} = \mathcal{B}_{R} (\bdd{v} \circ \varphi_{\gamma, \bdd{a}}) \circ \varphi_{\gamma, \bdd{a}}^{-1}.
\end{align*}
The following corollary is an immediate consequence of \cref{thm:c0 vertex interp ref element} and a standard scaling argument:
\begin{corollary}
	\label{thm:c0 vertex interp element}
	Let $K \in \mathcal{T}$, $\bdd{a} \in \mathcal{V}_K$. There exists a linear operator $\mathcal{I}_{K}^{\bdd{a}} : \bdd{X} \to \bm{\mathcal{P}}_{p}(K)$ satisfying
	\begin{enumerate}
		\item[(1)] $D^{\alpha} \mathcal{I}_K^{\bdd{a}} \bdd{v}(\bdd{a}) = D^{\alpha} \bdd{v}(\bdd{a})$, $|\alpha| \leq 1$, and $\mathcal{I}_K^{\bdd{a}} \bdd{v}|_{\gamma} = \begin{cases}
			\mathcal{B}_{\gamma}^{\bdd{a}} \bdd{v} & \gamma \in \mathcal{E}_{K} \cap \mathcal{E}_{\bdd{a}}, \\
			\bdd{0} & \gamma \in \mathcal{E}_{K} \setminus \mathcal{E}_{\bdd{a}},
		\end{cases}$ 
		
		\item[(2)] $h_K^{-1} \| \mathcal{I}_{K}^{\bdd{a}} \bdd{v} \|_{K} + | \mathcal{I}_{K}^{\bdd{a}} \bdd{v} |_{1,K} \leq C \{ h_K^{-1} \| \bdd{v} \|_{K} +  | \bdd{v} |_{1,K}  \}$.
		
		\item[(3)] For  $\gamma \in \mathcal{E}_K \cap \mathcal{E}_{\bdd{a}}$, there holds
		\begin{multline}
			\label{eq:c0 vertex interp h1200}
			|\gamma|^{-1/2} \| \xi_{\bdd{a}} \bdd{v} -  \mathcal{I}_{K}^{\bdd{a}} \bdd{v} \|_{\gamma} + | \xi_{\bdd{a}} \bdd{v} -  \mathcal{I}_{K}^{\bdd{a}} \bdd{v} |_{\bdd{H}^{1/2}_{00}(\gamma)} \leq C \{ |\gamma|^{-1/2} \| \bdd{v} \|_{\gamma} + | \bdd{v} |_{1/2, \gamma} \}.
		\end{multline}
		
		\item[(4)] If $\bdd{v}|_{\gamma} = \bdd{0}$ for some $\gamma \in \mathcal{E}_K$, then $\mathcal{I}_{K}^{\bdd{a}} \bdd{v}|_{\gamma} = \bdd{0}$.
	\end{enumerate}
\end{corollary}

Finally, we require the analogue of \cref{thm:c0 vertex interp element} that interpolates vertex gradients for functions in $\Sigma$:
\begin{theorem}
	\label{thm:c1 vertex interp element}
	Let $K \in \mathcal{T}$, $\bdd{a} \in \mathcal{V}_K$, and parameterize $\gamma \in \mathcal{E}_K \cap \mathcal{E}_{\bdd{a}}$ by $\varphi_{\gamma, \bdd{a}}$. There exists a linear operator $\mathcal{J}_{K}^{\bdd{a}} : \Sigma \to \mathcal{P}_{p}(K)$ satisfying
	\begin{enumerate}
		\item[(1)] $\mathcal{J}_{K}^{\bdd{a}} \psi(\bdd{a}) = 0$, $D^{\alpha} \mathcal{J}_{K}^{\bdd{a}} \psi(\bdd{a}) = D^{\alpha} \psi(\bdd{a})$, $1 \leq |\alpha| \leq 2$, and
		\begin{align*}
			\mathcal{J}_K^{\bdd{a}} \psi|_{\gamma}(t) &= \begin{cases}
				\int_{-1}^{t} \mathcal{B}_R (( \psi \circ \varphi_{\gamma, \bdd{a}} )')(s) \ ds & \gamma \in \mathcal{E}_{K} \cap \mathcal{E}_{\bdd{a}}, \\
				0 &  \gamma \in \mathcal{E}_K \setminus \mathcal{E}_{\bdd{a}}, 
			\end{cases} \\
			\partial_{n_{\gamma}} \mathcal{J}_K^{\bdd{a}} \psi|_{\gamma}(t) &= \begin{cases}
				\mathcal{B}_R ( (\partial_{n_{\gamma}} \psi) \circ \varphi_{\gamma, \bdd{a}} )(t) & \gamma \in \mathcal{E}_{K} \cap \mathcal{E}_{\bdd{a}}, \\
				0 &  \gamma \in \mathcal{E}_K \setminus \mathcal{E}_{\bdd{a}}.
			\end{cases}
		\end{align*}		
		
		\item[(2)] $h_K^{-2} \| \mathcal{J}_{K}^{\bdd{a}} \psi \|_{K} + h_K^{-1} | \mathcal{J}_{K}^{\bdd{a}} \psi |_{1,K} + | \mathcal{J}_{K}^{\bdd{a}} \psi |_{2, K} \leq C \{ h_K^{-1} | \psi |_{1,K} + | \psi |_{2, K} \}$.
		
		\item[(3)] For  $\gamma \in \mathcal{E}_{\bdd{a}} \cap \mathcal{E}_K$, there holds
		\begin{multline}
			\label{eq:c1 vertex interp h1200}
			|\gamma|^{-1/2} \| \xi_{\bdd{a}} \nabla \psi - \nabla \mathcal{J}_{K}^{\bdd{a}} \psi \|_{\gamma} + | \xi_{\bdd{a}} \nabla \psi - \nabla \mathcal{J}_{K}^{\bdd{a}} \psi |_{\bdd{H}^{1/2}_{00}(\gamma)} \\ 
			\leq C \{ |\gamma|^{-1/2} \| \nabla \psi\|_{\gamma} + |\nabla  \psi|_{1/2, \gamma} \}.
		\end{multline}
		
		\item[(4)] If $\nabla \psi|_{\gamma} = \bdd{0}$ for some $\gamma \in \mathcal{E}_{K}$, then $D^{\beta} \mathcal{J}_K^{\bdd{a}} \psi|_{\gamma} = 0$ for all $|\beta| \leq 1$.
	\end{enumerate}
\end{theorem}
\begin{proof}
	Let $K \in \mathcal{T}$, $\bdd{a} \in \mathcal{V}_K$ and $\psi \in \Sigma$. Let $\bdd{v} = \nabla (\psi \circ \bdd{F}_K) \in \bm{\mathcal{P}}_{p}(\hat{T})$, where $\bdd{F}_K : \hat{T} \to K$ is any invertible affine map with $\bdd{F}_K(\hat{\bdd{a}}_1) = \bdd{a}$, and let $\mathcal{G} \bdd{v} \in \bm{\mathcal{P}}_p(\hat{T})$ be given by \cref{thm:c0 vertex interp ref element}. We define $f, g : L^2(\partial \hat{T}) \to \mathbb{R}$ by the rules
	\begin{align*}
		f(\bdd{x}) = \int_{\hat{\bdd{a}}_1}^{\bdd{x}} \mathcal{G} \bdd{v} \cdot \hat{\unitvec{t}} \ dS \quad \text{and} \quad g(\bdd{x}) = (\mathcal{G} \bdd{v} \cdot \hat{\unitvec{n}})(\bdd{x}), \qquad \bdd{x} \in \partial \hat{T},
	\end{align*}
	where the path integral is taken in the counter-clockwise direction around $\partial \hat{T}$. Clearly $f$ and $g$ are piecewise polynomial functions and satisfy the following:
	\begin{enumerate}
		\item[(i)] Thanks to \cref{thm:c0 vertex interp ref element} (1) and \cref{lem:hardy00 properties}, $\int_{\partial \hat{T}} \mathcal{G} \bdd{v} \cdot \hat{\unitvec{t}} \ dS = 0$ and so $f$ is continuous.
		
		\item[(ii)] $\bdd{\sigma} := \partial_t f \hat{\unitvec{t}} + g \hat{\unitvec{n}} = \mathcal{G} \bdd{v}$ on $\partial \hat{T}$ is continuous since $\mathcal{G} \bdd{v}$ is polynomial.  
		
		\item[(iii)] For $\gamma, \gamma' \in \mathcal{E}_{\bdd{a}} \cap \mathcal{E}_K$, there holds
		\begin{align*}
			\partial_{t_{\gamma}} \bdd{\sigma}|_{\gamma}(\bdd{a}) \cdot \unitvec{t}_{\gamma'} &= \partial_{t_{\gamma}}  \mathcal{G} \bdd{v}(\bdd{a}) \cdot \unitvec{t}_{\gamma'} = \partial_{t_{\gamma}} \nabla (\psi \circ \bdd{F}_K)(\bdd{a}) \cdot \unitvec{t}_{\gamma'} \\
			&= \partial_{t_{\gamma'}} \nabla (\psi \circ \bdd{F}_K)(\bdd{a}) \cdot \unitvec{t}_{\gamma} = \partial_{t_{\gamma}'}  \mathcal{G} \bdd{v}(\bdd{a}) \cdot \unitvec{t}_{\gamma =} \partial_{t_{\gamma'}} \bdd{\sigma} |_{\gamma'}(\bdd{a}) \cdot \unitvec{t}_{\gamma}
		\end{align*}
		by \cref{thm:c0 vertex interp ref element} (1) and \cref{lem:hardy00 properties}. Moreover, $\partial_{t_{\gamma}} \bdd{\sigma}_{\gamma}|_{\gamma}(\bdd{b}) = \partial_{t_{\gamma}}  \mathcal{G} \bdd{v}(\bdd{b}) = \bdd{0}$ for $\gamma \in \mathcal{E}_K$ and $\bdd{b} \in \mathcal{V}_K \setminus \{\bdd{a}\}$ by \cref{lem:hardy00 properties}.
	\end{enumerate}
	Thanks to \cite[Corollary 2.3]{AinCP19Extension}, there exists $\rho \in \mathcal{P}_p(\hat{T})$ satisfying
	\begin{align*}
		\rho|_{\partial \hat{T}} = f, \quad \partial_n \rho|_{\partial \hat{T}} = g, \quad \text{and} \quad \|\rho\|_{2,\hat{T}} \leq C \{ \|f\|_{\partial \hat{T}} + \|\bdd{\sigma}\|_{1/2, \partial \hat{T}} \}.
	\end{align*}
	Applying the Cauchy-Schwarz inequality and using the relation $\bdd{\sigma} := \mathcal{G} \bdd{v}$ on $\partial \hat{T}$ and \cref{thm:c0 vertex interp ref element} (2), we obtain
	\begin{align*}
		\|f\|_{\partial \hat{T}} + \|\bdd{\sigma}\|_{1/2, \partial \hat{T}} \leq C \| \mathcal{G} \bdd{v} \|_{1/2, \partial \hat{T}} \leq C \| \mathcal{G} \bdd{v}\|_{1, \hat{T}} \leq C \| \bdd{v}\|_{1, \hat{T}} = \| \nabla (\psi \circ \bdd{F}_K) \|_{1, \hat{T}}.
	\end{align*}
	Defining $\mathcal{J}_K^{\bdd{a}} \psi := \rho \circ \bdd{F}_K^{-1}$ and applying a standard scaling argument give
	\begin{multline*}
		h_K^{-2} \| \mathcal{J}_{K}^{\bdd{a}} \psi \|_{K} + h_K^{-1} | \mathcal{J}_{K}^{\bdd{a}} \psi |_{1,K} + | \mathcal{J}_{K}^{\bdd{a}} \psi |_{2, K} \\
		\leq C \| \rho\|_{2, \hat{T}} \leq C \| \nabla (\psi \circ \bdd{F}_K) \|_{1, \hat{T}} \leq C \{ h_K^{-1} | \psi |_{1,K} + | \psi |_{2, K} \}.
	\end{multline*}
	(1), (3), and (4) now follow from \cref{thm:c0 vertex interp ref element} (1), (3), and (4), respectively.
\end{proof}

\section{Stability of the Decompositions}

\begin{lemma}
	\label{lem:general stability}
	Let $Y^0 := \tilde{\bdd{X}}_B$ and $Y^1 := \tilde{\Sigma}_B$ be partitioned as follows:
	\begin{align*}
		Y^i = Y^i_C + \sum_{ \bdd{a} \in \mathcal{V}} Y^i_{\bdd{a}},
	\end{align*}
	where $Y^0_C = \tilde{\bdd{X}}_C$, $Y^1_C = \tilde{\Sigma}_C$, $Y^0_{\bdd{a}} = \tilde{\bdd{X}}_{\bdd{a}}$, and $Y^1_{\bdd{a}} = \tilde{\Sigma}_{\bdd{a}}$. For $i \in \{0, 1\}$ and $y \in Y^i$, there exist $y_C \in Y_C^i$ and $y_{\bdd{a}} \in Y_{\bdd{a}}^i$, $\bdd{a} \in \mathcal{V}$, satisfying
	\begin{align}
		\label{eq:general stability}
		y = y_C + \sum_{ \bdd{a} \in \mathcal{V}} y_{\bdd{a}} \quad \text{and} \quad \| \bdd{\varepsilon}(\vcurl^i y_C) \|^2 + \sum_{ \bdd{a} \in \mathcal{V}}  \| \bdd{\varepsilon}(\vcurl^i y_{\bdd{a}}) \|^2 \leq C \| \bdd{\varepsilon}(\vcurl^i y) \|^2,
	\end{align}
	where $\vcurl^0 := I$ and $\vcurl^1 := \vcurl$.
\end{lemma}
\begin{proof}
	Let
	\begin{align*}
		\mathcal{R}_C^i := \begin{cases}
			\mathcal{I}_C & i = 0, \\
			\mathcal{J}_C & i= 1,
		\end{cases}
		\ \ 
		\mathcal{R}_{\bdd{a}}^i := \begin{cases}
			\mathcal{I}_{\bdd{a}} & i = 0, \\
			\mathcal{J}_{\bdd{a}} & i= 1,
		\end{cases} 
		\quad \bdd{a} \in \mathcal{V},
		\ \
		\mathcal{R}_{\gamma}^i := \begin{cases}
			\mathcal{I}_{\gamma} & i = 0, \\
			\mathcal{J}_{\gamma} & i= 1,
		\end{cases} 
		\quad \gamma \in \mathcal{E},
	\end{align*}
	where the operators are given by \cref{lem:coarse comp 2,lem:c0 vertex comp aux,lem:c0 edge interp,lem:h2 coarse comp,lem:c1 vertex comp aux,lem:h2 edge comp}, and define
	\begin{align*}
		y_C := \mathcal{R}_C^i y, \quad z_{\bdd{a}} := \mathcal{R}_{\bdd{a}}^i (y - y_C), \quad \text{and} \quad z_{\gamma} := \mathcal{R}_{\gamma}^i \left(y - y_C - \sum_{ \bdd{a} \in \mathcal{V}} y_{\bdd{a}} \right),  
	\end{align*}
	where $\bdd{a} \in \mathcal{V}$ and $\gamma \in \mathcal{E}$. Let $z := y_C + \sum_{ \bdd{a} \in \mathcal{V}} z_{\bdd{a}} + \sum_{ \gamma \in \mathcal{E}} z_{\gamma}$. On each edge $\gamma \in \mathcal{E}$, $D^{\alpha} z = D^{\alpha} y$ for all $|\alpha| \leq i$ by construction. Thus, $z \equiv y$ by \cref{rem:stokes extension boundary,rem:biharmonic traces}.
	
	We now show that this decomposition is stable. Let $K \in \mathcal{T}$. Applying \cref{eq:coarse continuity 2}, \cref{eq:c0 vertex comp aux h1 cont}, \cref{eq:h2 coarse continuity}, and \cref{eq:c1 vertex comp aux h2 cont} give $\|\bdd{\varepsilon}(\vcurl^{i} y_C)\|_{K}^2 \leq C \|\bdd{\varepsilon}(\vcurl^{i} y)\|_{\mathcal{T}_K}^2$ and
	\begin{align*}
		\sum_{ \bdd{a} \in \mathcal{V}_K} \| \bdd{\varepsilon}(\vcurl^{i} z_{\bdd{a}}) \|_{K}^2 \leq C |y - \mathcal{R}_C^i y|_{i,K}^2 \leq C  \|\bdd{\varepsilon}(\vcurl^i y) \|_{\mathcal{T}_K}^2.
	\end{align*}
	Moreover, using \cref{eq:c0 vertex comp aux h1200 cont}, \cref{eq:c1 vertex comp aux h1200 cont}, the trace theorem, \cref{eq:coarse continuity 2}, and \cref{eq:h2 coarse continuity}, we obtain
	\begin{align*}
		\sum_{ \gamma \in \mathcal{E}_K} \| \bdd{\varepsilon}(\vcurl^{i} z_{\gamma})\|_{K}^2 &\leq C \sum_{ \gamma \in \mathcal{E}_K} \left| \nabla^{i} (y - y_C) - \sum_{ \bdd{a} \in \partial \gamma} \nabla^{i} z_{\bdd{a}}  \right|_{{H}^{1/2}_{00}(\gamma)}^2 \\
		&\leq C \sum_{ \gamma \in \mathcal{E}_K} \sum_{ \bdd{a} \in \partial \gamma} |\xi_{\bdd{a}} \nabla^i (y - y_C - z_{\bdd{a}})|_{{H}^{1/2}_{00}(\gamma)}^2  \\
		&\leq C \{ h_K^{-1/2} \| \nabla^{i}(y - y_C) \|_{\partial K} + |\nabla^{i} (y -y_C)|_{1/2,\partial K} \} \\
		&\leq C \{ h_K^{-1} \| \nabla^{i} (y - y_C) \|_{K} + |\nabla^{i} (y - y_C)|_{1,K} \} \\
		&\leq C  \|\bdd{\varepsilon}(\vcurl^{i} y) \|_{\mathcal{T}_K}^2. 
	\end{align*}
	Collecting results gives
	\begin{align}
		\label{eq:proof:general stable decomp aux}
		&\|\bdd{\varepsilon}(\vcurl^i y_C) \|_{K}^2 + 	\sum_{ \bdd{a} \in \mathcal{V}}  \| \bdd{\varepsilon}(\vcurl^i z_{\bdd{a}}) \|_{K}^2 + \sum_{ \gamma \in \mathcal{E}} \| \bdd{\varepsilon}(\vcurl^i z_{\gamma})\|_{K}^2 \leq C \|\bdd{\varepsilon}(\vcurl^i y) \|_{\mathcal{T}_K}^2. 
	\end{align}
	
	Let $y_{\bdd{a}} := z_{\bdd{a}} + \frac{1}{2} \sum_{ \gamma \in \mathcal{E}_{\bdd{a}}} z_{\gamma}$. By \cref{lem:c0 vertex comp aux}, \cref{lem:c1 vertex comp aux}, \cref{eq:c0 edge interp prop}, and \cref{eq:h2 edge comp prop}, $y_{\bdd{a}} \in Y_{\bdd{a}}$. Moreover, the relation $z = z_C + \sum_{ \bdd{a} \in \mathcal{V}} z_{\bdd{a}}$ follows by construction, and \cref{eq:general stability} follows from \cref{eq:proof:general stable decomp aux} and summing over the elements.
\end{proof}

\bibliographystyle{amsplain}
\bibliography{references}
\end{document}